\documentclass[11pt,a4paper,reqno,final]{article}
\usepackage{graphicx}
\usepackage{amsmath}
\usepackage{empheq}
\usepackage[margin=1.4in]{geometry}
\usepackage{authblk}
\usepackage{lineno}
\usepackage{longtable}

\usepackage{array}
\newcolumntype{P}[1]{>{\centering\arraybackslash}p{#1}}
\newcolumntype{M}[1]{>{\centering\arraybackslash}m{#1}}

\newcommand{\dx}{\mathrm{d}x}

\newcommand{\fmin}{\mathcal{F}_{\text{min}}}
\newcommand{\fmax}{\mathcal{F}_{\text{max}}}
\newcommand{\thr}{{\mathrm{thr}}}
\newcommand{\ath}{{\alpha_\thr}}
\newcommand{\Hdx}[1]{H_{\partial x}^{1,#1}}
\newcommand{\alast}{\alpha^{\mathrm{R}}}
\newcommand{\edit}[1]{#1}

\def\delx{\ensuremath\dfrac{\partial}{\partial x}}

\newcommand{\ats}[3][2]{\alpha_{#3}^{#2}}
\newcommand{\up}[3][2]{u_{#3}^{#2\,+}}
\newcommand{\um}[3][2]{u_{#3}^{#2\,-}}

\newcommand{\brate}[3][2]{b_{#3}^{#2}}
\newcommand{\drate}[3][2]{d_{#3}^{#2}}

\newcommand{\halfarrow}{\rightharpoonup}
\newcommand{\supp}{\text{supp}}
\newcommand{\dcurly}[1]{\{\!\!\{#1\}\!\!\}}
\newcommand{\ml}{\textsc{ml}}
\newcommand{\CFL}{\textsc{CFL}}

\usepackage{multirow}
\usepackage[table,xcdraw]{xcolor}

\usepackage{amsmath,amsthm,amssymb}
\usepackage{mathrsfs}
\numberwithin{equation}{section}
\usepackage[labelformat=simple]{subcaption}

\usepackage{enumitem}
\usepackage{mathtools}  
\mathtoolsset{showonlyrefs} 
\allowdisplaybreaks

\theoremstyle{definition}

\newtheorem{case}{Case}

\theoremstyle{plain}
\newtheorem{theorem}{Theorem}[section]
\newtheorem{proposition}[theorem]{Proposition}
\newtheorem{definition}[theorem]{Definition}

\newtheorem{lemma}[theorem]{Lemma}
\newtheorem{remark}[theorem]{Remark}

\newtheorem{discrete}[theorem]{Discrete scheme}

\definecolor{refkey}{rgb}{0,0,1}
\definecolor{labelkey}{rgb}{0,0,1}

\begin{document}
\title{Convergence analysis of a numerical scheme for a tumour growth model}
\author[1]{J\'{e}r\^{o}me Droniou\thanks{mail: \href{mailto:jerome.droniou@monash.edu}{jerome.droniou@monash.edu}}}
\author[2]{Neela Nataraj\thanks{mail: \href{mailto:neela.nataraj@iitb.ac.in}{neela.nataraj@iitb.ac.in}}}
\author[3]{Gopikrishnan C. Remesan\thanks{mail: \href{mailto:gopikrishnan.chirappurathuremesan@monash.edu}{gopikrishnan.chirappurathuremesan@monash.edu}}}
\date{\today}
\affil[1]{\small{School of Mathematics, Monash University, Victoria 3800, Australia}}
\affil[2]{\small{Department of Mathematics, Indian Institute of Technology Bombay, Mumbai, Maharashtra 400076, India}}
\affil[3]{\small{IITB -- Monash Research Academy, Indian Institute of Technology Bombay, Mumbai, Maharashtra 400076, India}}
\maketitle

\begin{abstract}
\edit{We consider a one--spatial dimensional tumour growth model~\cite{breward_2001,breward_2002,breward_2003}  that consists of three dependent variables of space and time: volume fraction of tumour cells, velocity of tumour cells, and nutrient concentration.} The model variables satisfy a coupled system of semilinear advection equation (hyperbolic), simplified linear Stokes equation (elliptic), and semilinear diffusion equation (parabolic) with appropriate conditions on the time--dependent boundary, which is governed by an ordinary differential equation. We employ a reformulation of the model defined in a larger, fixed time--space domain to overcome some theoretical difficulties related to the time--dependent boundary. This reformulation reduces the complexity of the model by removing the need to explicitly track the time--dependent boundary, but nonlinearities in the equations, noncoercive operators in the simplified Stokes equation, and interdependence between the unknown variables still challenge the proof of suitable \emph{a priori} estimates.  A numerical scheme that employs a finite volume method for the hyperbolic equation, a finite element method for the elliptic equation, and a backward Euler in time--mass lumped finite element in space method for the parabolic equation is developed. \edit{We establish the existence of a time interval $(0,T_{\ast})$ over which, using compactness techniques, we can extract a convergent subsequence of the numerical approximations. The limit of any such convergent subsequence is proved to be a weak solution of the continuous model in an appropriate sense, which we call a threshold solution.} Numerical tests and justifications that confirm the theoretical findings conclude the paper.
\end{abstract}

\section{Introduction}
\label{sec:intro}
One spatial dimensional tumour growth models are usually obtained by assuming that a higher spatial dimensional tumour grows radially~\cite{Adam1986229,Byrne20021,Mahmoud20171353,Zhuang201886}. \edit{Such one--dimensional models are much simpler than their intricate higher dimensional versions~\cite{Hogea2006,hubbard,Perfahl,scuime}. However, theoretical and computational difficulties offered by even these simplified one--dimensional versions are severe.} The time--dependent boundary, noncoercive coefficient functions, nonlinearities, and the strong coupling between the equations are a few challenges worth mentioning. \edit{ In this article, we consider a tumour growth model proposed by C.~J.~W.~Breward et al.~\cite{breward_2001,breward_2002,breward_2003}}. The model assumes that the tumour cells (cell phase) are embedded in a fluid medium (fluid phase), see Figure~\ref{fig:rad_sym}. The mechanical interactions between these two phases along with the differential distribution of the limiting nutrient, which is oxygen in this case, cause the growth or depletion of the tumour. The relative volume of the cell phase is called the cell volume fraction, the velocity by which the cells are moving is called the cell velocity, and the concentration of the limiting nutrient is quantified by the oxygen tension; these three time--space dependent variables are denoted by $\check{\alpha},\,\check{u},\,$ and $\check{c}$, respectively. \edit{  Detailed aspects of the modelling can be found in the  works by  C.~J.~W.~Breward et al.~\cite{breward_2003} and H.~Byrne et al.~\cite{byrne_2003}.}

\begin{figure}[htp]
	\resizebox{0.9\textwidth}{!}{
	\centering	
		\begin{minipage}{6cm}
\begin{subfigure}[b]{1\textwidth}
	\includegraphics[scale=0.7]{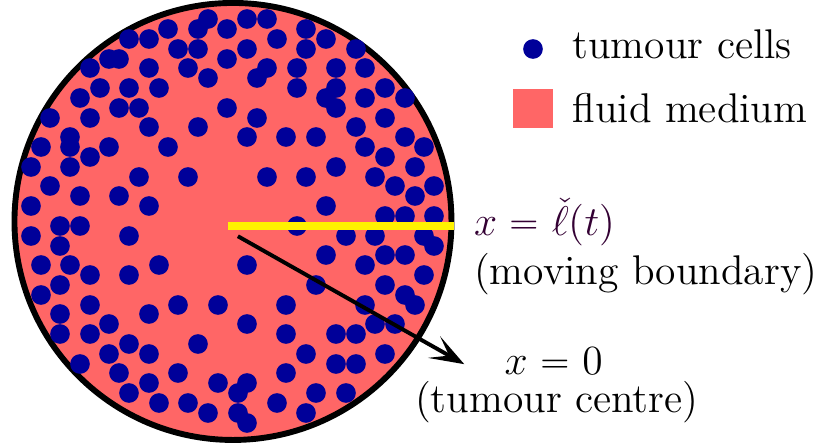}
	\caption{Two--phase model of a radially symmetric tumour.}
	\label{fig:rad_sym}
\end{subfigure}
\end{minipage}
\hspace{0.2cm}
	\begin{minipage}{7cm}
\begin{subfigure}[b]{1\textwidth}
	\centering
\includegraphics[scale=0.8]{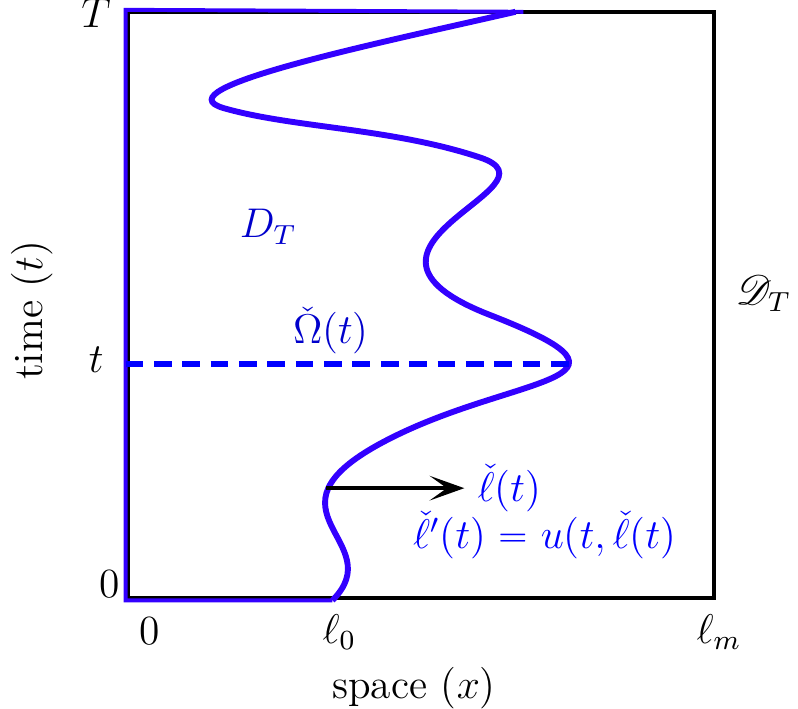}
\caption{Time--space domain $D_T$ and its bounding box $\mathscr{D}_T = (0,T) \times (0,\ell_m)$.}
\label{fig:tum_domains}
\end{subfigure}
\end{minipage}
}
\caption{Radially symmetric tumour and corresponding time--space domains.}
\end{figure}
\medskip
\noindent \textbf{Presentation of the mathematical model}\medskip\\
The tumour growth under the current investigation is over the finite time interval $(0,T)$, where $T > 0$ and \edit{ all the variables and parameters are dimensionless.} Let $\check{\ell} : (0,T) \rightarrow \mathbb{R}$ be a function of time, whose dynamics will be specified later, and set  $\check{\Omega}(t) := (0,\check{\ell}(t))$. Define the time--space domain $D_T: = \cup_{0 < t < T}(\{t\}\times \check{\Omega}(t))$, and its bounding box $\mathscr{D}_T := (0,T) \times (0,{{\ell_m}})$, where ${{\ell_m}} > \check{\ell}(t)$ for $t \in(0,T)$ -- see Figure~\ref{fig:tum_domains}. The unknowns $\check{\alpha}, \check{u}$, and $\check{c}$ are real valued functions defined on $D_T$ and they depend on both space and time. The model seeks variables $(\check{\alpha},\check{u},\check{c},\check{\ell})$ such that,~on $D_T$,
\begin{subequations}
\label{system:nd_sys}
\begin{align}
\label{eqn:vol_fraction}
\dfrac{\partial \check{\alpha}}{\partial t} + \delx (\check{u} \check{\alpha}) &= \check{\alpha} f(\check{\alpha},\check{c}), \\
\label{eqn:cel_velocity}
\dfrac{k \check{u} \check{\alpha}}{1 - \check{\alpha}} - \mu \dfrac{\partial}{\partial x} \left(\check{\alpha} \dfrac{\partial \check{u}}{\partial x}\right) &= -\dfrac{\partial}{\partial x} \left(\mathscr{H}(\check{\alpha}) \right), \\
\label{eqn:oxygen_tension}
\dfrac{\partial \check{c}}{\partial t} - \lambda\dfrac{\partial^2 \check{c}}{\partial x^2} &= -\dfrac{Q \check{\alpha} \check{c}}{1 + \widehat{Q}_1|\check{c}|}, \textrm{ and }  \\
\label{eqn:bd_velocity}
\check{\ell}'(t) &= \check{u}(t,\check{\ell}(t)),
\end{align}
with initial conditions 
\begin{gather}
\check{\alpha}(0,x) = \alpha_0(x),\;\check{c}(0,x) = c_0(x)\;\ \forall x \in \check{\Omega}(0),\quad\;\check{\ell}(0) = \ell_0,
\label{eqn:in_cond}
\end{gather}
and boundary conditions 
\begin{align}
 \label{eqn:bdr_cond_1}
&\check{u}(t,0) = 0,\;\mu \dfrac{\partial \check{u}}{\partial x}(t,\check{\ell}(t)) = \dfrac{(\check{\alpha}(t,\check{\ell}(t))  - \alast)^{+}}{(1 - \check{\alpha}(t,\check{\ell}(t)))^2}, \\
&\dfrac{\partial \check{c}}{\partial x}(t,0) = 0,\; \text{ and }\;\check{c}(t,\check{\ell}(t)) = 1 \quad \forall t \in (0,T).  \label{eqn:bdr_cond_2}
\end{align}
\end{subequations}
Here, 
\[
f(\check{\alpha},\check{c}) := \frac{(1 + s_1)(1 - \check{\alpha})\check{c}}{1 + s_1 \check{c}} - \frac{s_2 + s_3 \check{c}}{1 + s_4 \check{c}},\quad \mathscr{H}(\check{\alpha}) := \frac{\check{\alpha}(\check{\alpha} - \alast)^{+}}{(1 - \check{\alpha})^2},
\]
and $a^{+}$ and $a^{-}$ used in the sequel are defined by $a^{+} := \max(a,0)$ and $a^{-} := -\min(a,0)$. The positive constants $s_1,\,s_2,\,s_3,$ and $s_4$ control the cumulative production rate of the tumour cells, $\check{\alpha} f(\check{\alpha}, \check{c})$. The constant $\alast$ regulates repulsive and attractive interactions between the tumour cells. The positive constant $k$ controls traction between the cell and fluid phases, whereas $\mu$ is the viscosity coefficient in the cell phase. The fluid phase is assumed to be inviscid. The diffusivity coefficient of oxygen is denoted by $\lambda$. The constants $Q$ and $\widehat{Q}_1$ are nonnegative, and control the oxygen consumption rate by the tumour cells. For more details on physical constants, refer to the reviews~\cite{Byrne20061563,Roose2007179} and the references therein. Assume that
\begin{equation}\label{eq:bd.alpha0}
0 < m_{01} \le \alpha_0 \le m_{02}< 1\quad\mbox{on $\check{\Omega}(0)$},
\end{equation}
where $m_{01}$ and $m_{02}$ are constants, that $c_0$ is continuous, and that $0 \leq c_{0} \le 1$ on $\check{\Omega}(0)$. Physical motivations used to obtain the  boundary conditions are briefly sketched in Table~\ref{tab:boundary} of Appendix~\ref{appen_A}.

\edit{The original oxygen source term~$-Q\check{\alpha} \check{c}/ (1+\widehat{Q}_1\check{c})$ of~\cite{breward_2002} is modified in~\eqref{eqn:oxygen_tension} to ensure the nonnegativity of oxygen tension (which represents a concentration). Since we will construct a solution of \eqref{system:nd_sys} such that $\check{c}$ is positive, this substitution does not actually modify the model.} Also, the original source term $(\check{\alpha}- \alast)H(\alpha-\alpha_{\mathrm{min}})$ that appears in~\cite{breward_2002},  where $\alpha_{\emph{min}}$ is a constant and $H(s)=0$ if $s<0$, $H(s)=1$ if $s\ge 0$, is replaced by $(\check{\alpha} - \alast)^{+} = (\check{\alpha}- \alast)H(\alpha-\alast)$ in \eqref{eqn:cel_velocity} (through $\mathscr{H}$) and in  \eqref{eqn:bdr_cond_1}. In the case $\alpha_{\emph{min}}\not= \alast$, the nonlinear term $(\check{\alpha}- \alast)H(\alpha-\alpha_{\mathrm{min}})$ is discontinuous with respect to $\check{\alpha}$, which makes any proof of existence of a solution to~\eqref{system:nd_sys} difficult -- and even questions the well-posedness of the model. The continuity of $(\check{\alpha} - \alast)^{+}$ is essential to obtain a priori estimates (see in particular the proof of Proposition~\ref{prop:alpha_sbv}), and to apply limit arguments to the numerical scheme.

\medskip
\noindent \textbf{Literature}\medskip\\
 Despite the fact that tumour growth models have been popular since the seventies~\cite{Byrne20061563,Roose2007179}, the theoretical literature available on this field is very few. Recently, J. Zheng and S. Cui~\cite{Zheng2019} considered existence of solutions for a tumour growth model with volume fraction and pressure in the tumour region as the unknown variables. The model equations in~\cite{Zheng2019} are fully linear, while the boundary conditions are nonlinear, and a local well-posedness result is proved. A similar linear model is considered by C. Calzada et al.~\cite{CarmenCalzada20111335}, and equivalence to an extended problem in a larger domain is proved. A more advanced model is considered by N. Zhang and Y. Tao~\cite{Zhang2013534}, where the nutrient concentration is also considered as a variable and the existence of solutions is obtained by transforming the fixed domain to a unit ball in $\mathbb{R}$. Studies from the numerical analysis point of view are scarce. J. A. Mackenzie and A. Madzvamuse~\cite{MacKenzie2011212} have shown the convergence of a finite difference scheme for a single variable tumour growth model with a nonlinear source term on a time dependent boundary.

 It is shown in~\cite{remesan_1} that the model~\eqref{system:nd_sys} can be recast into an extended model, where \eqref{eqn:vol_fraction} is set in $\mathscr{D}_T$ with $\check{\alpha}$ being extended by $0$ outside $D_T$, the variable $\check{\ell}$ is eliminated, and the variables $\check{u}$ and $\check{c}$ are extended to $\mathscr{D}_T\setminus D_T$ by $0$ and $1$, respectively. However, this model does not allow any uniform lower bounds on $\check{\alpha}$ inside the computational domain $\mathscr{D}_T$, which means that the velocity equation \eqref{eqn:cel_velocity} can lose its coercivity properties. In the present work, we therefore consider a modification of this extended model, hereafter called the \emph{threshold model}, in which we introduce a (small) threshold which determines the computational domain used for $\check{u}$ and $\check{c}$ (see Figure \ref{fig:tum_domains}). 
\medskip\\
\noindent \textbf{Contributions}\medskip\\
    The formulation of a numerical scheme for the threshold model with a suitable notion of solution, and analysis of the same to obtain the convergence of the iterates, are the primary objectives of this article. This approach has the added benefit of establishing the existence of a solution. The computational cost of re--meshing $\check{\Omega}(t)$ in such a way that an appropriate Courant--Friedrich--Lewy condition (CFL) is satisfied at each time step can be reduced significantly by using the threshold model and extension to a fixed domain~\cite{remesan_1}. We summarise the main contributions of this article here.
   \begin{itemize}[leftmargin=3\labelsep]
   		\setlength\itemsep{-0.2em}
   	\item A numerical scheme based on finite volume and Lagrange $\mathbb{P}^1$--finite element methods is designed such that the physical properties of the system~\eqref{system:nd_sys} are preserved -- in particular, positivity and boundedness of oxygen tension (see Lemma~\ref{lemma:c_pos}) and conservation of mass by volume fraction (see Lemma~\ref{lemma:mass_con}) .
   	\item Bounded variation estimates for the volume fraction, $H^1$ and $L^\infty$ estimates for the cell velocity, and spatial and temporal estimates for the derivatives of oxygen tension are obtained. 
   	\item The convergence analysis of numerical solutions for a tumour growth model that caters for the variables volume fraction, cell velocity and nutrient concentration is studied; to the best of our knowledge, it is the first convergence analysis of this kind.
   	\item It is established that the limit of (any subsequence of) the numerical solutions is indeed a solution to the threshold model, thus proving the existence of a solution for this model.
   	\item Results of numerical experiments that substantiate the theory developed are presented. \newpage
      \end{itemize}
   \noindent \textbf{Organisation}\medskip \\
   This paper is organised in the following way. This section is introductory; in Section~\ref{sec:equivalent}, we define the weak solution to the threshold model and in Section~\ref{sec:discretisation}, a numerical scheme is formulated. In Section~\ref{sec:main_thm}, the main theorems are stated. The compactness and convergence properties of the numerical solutions are derived in Section~\ref{sec:proof_main_1}. In Section~\ref{sec:convergence}, we show that the limit of numerical solutions obtained in Section~\ref{sec:proof_main_1} is a solution to the threshold model in an appropriate sense. In Section~\ref{sec:num_results}, we present numerical results of examples, and discuss the optimal time below which a solution exists. In Section~\ref{discussion},  possible extensions of the current work to other models in single and several spatial dimensions are discussed. We provide the expansions of notations and indexing abbreviations in Appendix~\ref{appen_A}. Mass conservation properties satisfied by the continuous variables of~\eqref{system:nd_sys} and discrete variables in the Discrete scheme~\ref{defn:semi_disc_soln} are presented in  Lemma~\ref{lemma:mass_con} and~\ref{lemma:disc_conser} in Appendix~\ref{appen_C}. The nonnegativity and boundedness satisfied by the oxygen tension is proved in Lemma~\ref{lemma:c_pos}.  A series of classical results used in this article are presented in Appendix~\ref{appen_B}.
   
    This article is set in such a way that an overall reading of Sections~\ref{sec:intro}--\ref{sec:main_thm}, steps \ref{is.1}--\ref{is.4} of Section~\ref{sec:Well_posedness}, steps \ref{cr.1}--\ref{cr.7} of Section~\ref{sec:compactness_est} and steps \ref{ca.1}--\ref{ca.4} of Section~\ref{sec:convergence} helps to understand the gist of the paper. Proofs of the steps mentioned above in their respective sections provide the details. We conclude this section by introducing a few notations used in the article.
\subsubsection*{Notations}

The notation $\nabla_{t,x}$ stands for $(\partial_t,\partial_x)$. The notation $(\cdot,\cdot)_{X}$ is the standard $L^2$ inner product in $X\subset\mathbb{R}^d$, $d \geq 1$. We define the norms $||u||_{0,X} := (u,u)_{X}^{1/2}$ and $||u||_{k,X} := \sum_{\pmb{j}, |\pmb{j}| \leq k} |\partial_x^{\pmb{j}} u|_{0,X}$, where $\pmb{j}$ is a multi--index.  The vector space $\mathbb{P}^1(X)$ is the collection of all polynomials of degree $\le 1$ on $X$.  A consolidated presentation of the continuous and discrete model variables is provided in Table~\ref{tab:not} of Appendix~\ref{appen_A}. For a detailed description of various notions of tumour radii, refer to Table~\ref{tab:tum_len} in Section~\ref{sec:equivalent}.

\section{Threshold model and well-posedness}
\label{sec:equivalent}
We introduce the notion of a threshold solution. A constant and positive parameter, $\ath$, characterises each threshold solution. The source term $\check{\alpha} f(\check{\alpha},\check{c})$ in~\eqref{eqn:vol_fraction} is modified to $(\check{\alpha} - \ath)^{+}f(\check{\alpha},\check{c})$, and the tumour radius at time $t$, $\check{\ell}(t)$, is the smallest number above which the cell volume fraction $\check{\alpha}(t,x)$ is entirely below $\ath$. 
In the limiting case $\ath$ approaches zero, the continuous function $(\check{\alpha} - \ath)^{+}f(\check{\alpha},\check{c})$ approaches $\check{\alpha} f(\check{\alpha},\check{c})$, and the tumour radius is the smallest number above which no tumour cells are present. Theorem~3 in~\cite{remesan_1} proves that the threshold solution with $\ath=0$ and the weak solution of the model~\eqref{system:nd_sys} are equivalent. In fact, this is a consequence of the fact that the weak divergence of the vector field $(\check{\alpha}, \check{u}\check{\alpha})$, which is equal to $-\check{\alpha} f(\check{\alpha},\check{c})$, belongs to $L^2(\mathscr{D}_T)$. Let $B_T$ be defined by $\{(t,\check{\ell}(t)) : t \in (0,T) \}$. The square integrability of the weak divergence of $(\check{\alpha}, \check{u}\check{\alpha})$ implies that the jump in the normal component of $(\check{\alpha}, \check{u}\check{\alpha})$ across $B_T$ is zero, which reduces to $(\check{\alpha}, \check{u}\check{\alpha})_{\vert_{B_T}} \cdot (-\check{\ell}'(t),1) = 0$. From this, condition~\eqref{eqn:bd_velocity} can be deduced, if $\check{\alpha}$ is positive, which is the one of the reasons of why we need to ensure that the discrete and threshold solutions remain positive.

However, in Definition~\ref{defn:ext_soln} we further relax the condition to be satisfied by the tumour radius. In \ref{ts.2}, we only demand that the volume fraction of the tumour cells outside the time--space domain must be less than or equal to  $\ath$ (see Figure~\ref{fig:tum_domains_full}). The convergence analysis is this article assures the existence of such a domain. It remains unsolved whether such a domain is unique and, if at all unique, coincides with the time--space domain wherein the tumour radius satisfies  \eqref{eqn:bd_velocity}. Two different notions of tumour radii are discussed so far and are summarised in Table~\ref{tab:tum_len}.

\begin{table}[h!]
	\centering
	\begin{tabular}{|M{1.4cm}|M{11cm}|}
		\hline
		Notation         &  \multicolumn{1}{>{\centering\arraybackslash}m{110mm}|}{Description} \\ \hline \hline
		$\check{\ell}$   &  $\begin{aligned}
		\text{solution to the ordinary differential equation } \\
		\left\{
		\begin{array}{r l}
		\check{\ell}'(t) &= \check{u}(t,\check{\ell}(t)),  \\
		\check{\ell}(0) &= \ell_0,
		\end{array}
		\right.
		\end{aligned}$     
		
		tumour radius used in the continuous model provided in~\cite{breward_2002}.  \\ \hline
		$\ell$         &   $ \forall\, x \ge \ell,\;\;\alpha(t,x) \le \ath \quad \quad (\star)$           
		
		Condition $(\star)$ is to be satisfied by $\ell(t)$, so that $(\alpha, c, u, \Omega)$ with $\Omega(t) = (0,\ell(t))$ is a Threshold solution  in the sense of Definition~\ref{defn:ext_soln}.          \\ \hline
	\end{tabular}
	\caption{Description of various notions of tumour radii.}
	\label{tab:tum_len}
\end{table}

The introduction of the threshold into the definition of the domain and in the source term helps to obtain boundedness and bounded variation estimates for the numerical solution of~\eqref{eqn:vol_fraction}, and thus enables the numerical scheme to converge to the weak form~\eqref{eqn:weak_vf}. The source term modification is also a way to account for the fact that, in the absence of sufficient amount of cells, the reaction term that drives their growth remains dormant. The details presented in Subsection~\ref{sec:numcom} complement this discussion.

Each threshold solution in the sense of Definition~\ref{defn:ext_soln} corresponds to a pair of prefixed constants $m_{11}$ and $m_{12}$, which ensure the positivity and boundedness (strictly below 1) of the volume fraction in $D_T^{\thr}$ defined by~\ref{ts.2} in Definition~\ref{defn:ext_soln}.  
\begin{figure}[htp]
	\centering
	\includegraphics[scale=0.9]{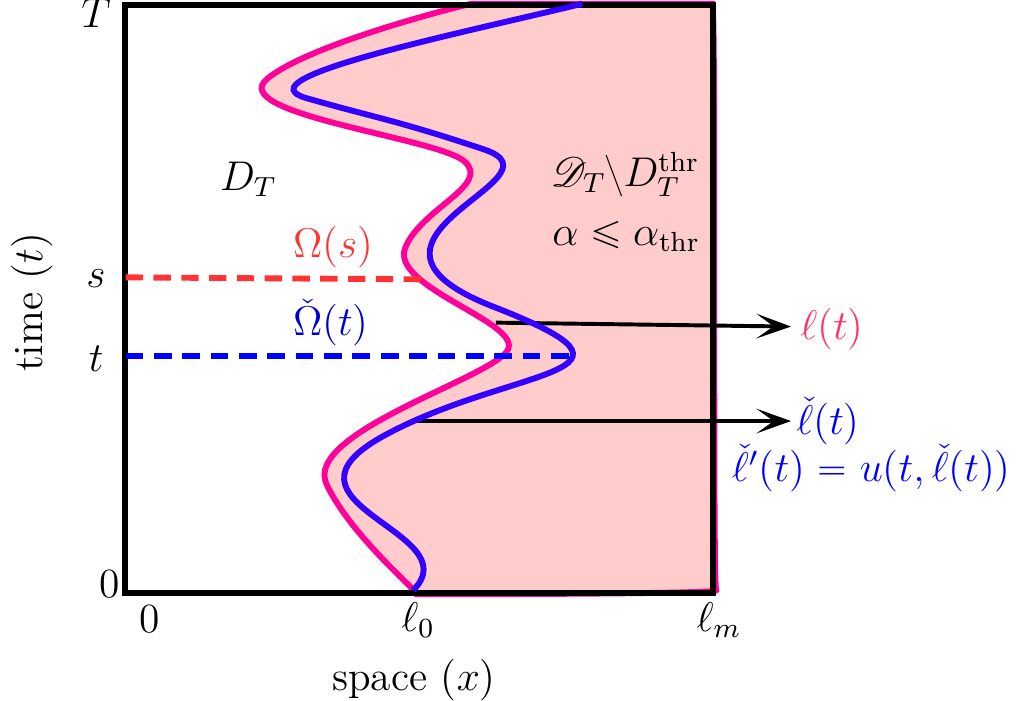}
	\caption{Tumour radii and time--space domains: $D_T$ is time--space domain (region to the left of the blue curve) defined by~\eqref{system:nd_sys}, $D_T^{\thr}$ (region to the left of the pink curve) is the time--space domain defined by the Threshold solution~\ref{defn:ext_soln} and $\mathscr{D}_T$ is the bounding box $(0,T) \times (0,\ell_m)$.}
	\label{fig:tum_domains_full}
\end{figure}

Recall that $(\cdot,\cdot)_{X}$ is the standard $L^2$ inner product on a set $X \subset \mathbb{R}^d$, $d \ge 1$. The domain $D_T^{\thr}$~defined by~\ref{ts.2} in Definition~\ref{defn:ext_soln} is open and bounded. Define the following vector spaces on $D_T^{\thr}$:
\begin{align}
\Hdx{u}(D_{T}^{\thr}) :={}& \{v \in L^2(D_{T}^{\thr})\,:\,\partial_x  v \in L^2(D_T^\thr)\text{ and } v(t,0) = 0\ \forall\,t \in (0,T)\}, \; \text{and}\\
\Hdx{c}(D_{T}^{\thr}) :={}& \{v \in L^2(D_T^\thr)\,:\,\partial_x  v \in L^2(D_T^\thr)\text{ and } v(t,\ell(t)) = 0\ \forall\,t \in (0,T)\}.
\end{align}
Define the inner product on the vector space $\Hdx{\varrho}(D_{T}^{\thr})$, where $\varrho \in \{u,c\}$, as follows: for $w, v \in \Hdx{\varrho}(D_{T}^{\thr})$
	\begin{align}
	\left(w,v \right)_{\Hdx{\varrho}(D_{T}^{\thr})} := (w,v)_{D_{T}^{\thr}} + (\partial_x w,\partial_x v)_{D_{T}^{\thr}} .
	\label{eqn:in_hdx}
	\end{align}
	The inner product~\eqref{eqn:in_hdx} induces a norm $||w||_{\Hdx{\varrho}(D_{T}^{\thr})}$ for which $\Hdx{\varrho}(D_{T}^{\thr})$ is a Hilbert space. Since for each $v\in\Hdx{\varrho}(D_T^{\thr})$ and each $t \in (0,T)$ the time slice $v(t,\cdot)$ belongs to $H^1(0,\ell(t))$, the zeroth order traces are well defined and the quantities $v(t,0)$ and $v(t,\ell(t))$ are meaningful.

\begin{subequations}
\begin{definition}[Threshold solution]
\label{defn:ext_soln}
Let $0 < m_{11} < m_{12} < 1$ be fixed constants that satisfy $m_{11} \le m_{01}$ and $m_{12} \ge m_{02}$, where $m_{01},m_{02}$ satisfy \eqref{eq:bd.alpha0}. Fix a threshold value $\ath \in (0,1)$.  A threshold solution (with threshold $\ath \in (0,1)$) and domain $D_{T}^{\thr}$ of the threshold model in $\mathscr{D}_T$ corresponding to the constants $m_{11}$ and $m_{12}$ is a 4-tuple $(\alpha,u,c,\Omega)$ such that the following conditions hold. 
\begin{enumerate}[label= $\mathrm{(TS.\arabic*)}$,ref=$\mathrm{(TS.\arabic*)}$,leftmargin=\widthof{(TS.4)}+1.2\labelsep]
\item\label{ts.1} The volume fraction $\alpha \in L^{\infty}(\mathscr{D}_T)$ is such that, for all $\varphi \in \mathscr{C}_c^{\infty} ([0,T)\times (0,{{\ell_m}}))$,
\begin{multline}
\label{eqn:weak_vf}
\int_{\mathscr{D}_T} ((\alpha,u \alpha)\cdot\nabla_{t,x}\varphi + (\alpha - \ath)^{+}\,f(\alpha,c)\,\varphi)\,\mathrm{d}t\,\mathrm{d}x \\
+ \int_{\Omega(0)} \varphi(0,x)\,\alpha_{0}(x)\,\mathrm{d}x = 0,
\end{multline}
and it holds $0 < m_{11} \leq \alpha_{\vert{\Omega(t)}} \leq m_{12} < 1$ for every $t \in [0,T)$.

\item\label{ts.2} The set $D_T^{\thr}$ is of the form $D_T^{\thr} = \cup_{0 < t <T}(\{t\} \times \Omega(t))$, where $\Omega(t) = (0,\ell(t))$, and we have $\alpha \le \ath$ on $\mathscr{D}_T\setminus D_T^{\thr}.$

\item\label{ts.3}  The velocity $u$ is such that $u\in \Hdx{u}(D_{T}^{\thr})$ and, for all $v \in \Hdx{u}(D_{T}^{\thr})$,
\begin{equation}
\label{eqn:weak_vel}
\int_{0}^{T}a^{t}(u(t,\cdot),v(t,\cdot))\,\mathrm{d}t = \int_{0}^T\mathcal{L}^{t}(v(t,\cdot))\,\mathrm{d}t,
\end{equation}
where $a^{t} : H^{1}(\Omega(t)) \times H^{1}(\Omega(t)) \rightarrow \mathbb{R}$ is the bilinear form and $\mathcal{L}^t : H^{1}(\Omega(t)) \rightarrow \mathbb{R}$ is the linear form defined by:
\begin{align}
\label{eqn:vel_bifm_con}
a^t(u,v) &=  k\left( \frac{\alpha}{1-\alpha} u,v \right)_{\Omega(t)} + \mu\left(\alpha \partial_x u, \partial_x v\right)_{\Omega(t)} \text{ and } \\
\label{eqn:vel_lfm_con}
\mathcal{L}^t(v) &=  \left( \mathscr{H}(\alpha) , \partial_x v \right)_{\Omega(t)}.
\end{align}
 We extend $u$ to $\mathscr{D}_T$ by setting $u\vert_{\mathscr{D}_T\setminus \overline{D_T^\thr}} := 0$. 
 
\item\label{ts.4} The oxygen tension $c$ is such that $c - 1 \in H_{\partial x}^{1,c}(D_{T}^{\thr})$, $c \ge 0$ and, for all $v \in H_{\partial x}^{1,c}(D_T^\thr)$ such that $\partial_t v\in L^2(D_T^\thr)$, 
\begin{multline}
-\int_{D_T^\thr} c\,\partial_t v\,\mathrm{d}x\,\mathrm{d}t + \lambda \int_{D_T^\thr} \partial_x c\,\partial_x v\,\mathrm{d}x\,\mathrm{d}t - \int_{\Omega(0)} c_{0}(x) v(0,x)\,\mathrm{d}x  \\
\label{eqn:weak_ot}
+ Q\int_{D_T^\thr} \dfrac{\alpha\,c\,v}{1 + \widehat{Q}_1|c|} \,\mathrm{d}x\,\mathrm{d}t = 0.
\end{multline}
We extend $c$ to $\mathscr{D}_T$ by setting $c_{\vert{\mathscr{D}_T\setminus \overline{D_T^\thr}}} := 1$.
\end{enumerate}
\end{definition}
\end{subequations}

\begin{figure}[h!]
	\centering
	\includegraphics[scale=1]{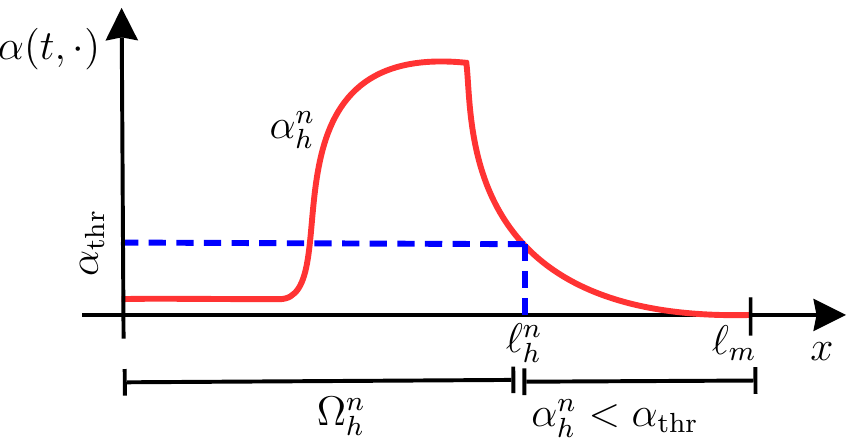}
	\caption{Selection of $\ell_h^n$ based on the value of $\alpha_h^n$.}
	\label{fig:alpha_thrfig}
\end{figure}
Given the bounds of $\alpha$ in Definition~\ref{defn:ext_soln}, it can easily be checked that $a^t$ is uniformly continuous and coercive on $H^1(\Omega(t))$, and that $\mathcal{L}^{t}$ is uniformly continuous on $H^1(\Omega(t))$. To prove existence of a solution for~\eqref{eqn:weak_vf} we need uniform supremum norm bounds on $u$,\,$\partial_x u$~\cite[p.~153]{eymard}, and $c$. Part of the analysis of the model consists in proving that $u$ and $\partial_x u$ satisfy uniform supremum norm bounds at the discrete level, which leads to the existence of a discrete solution for~\eqref{eqn:weak_vf} with uniformly bounded variation, and limit of which is a solution of~\eqref{eqn:weak_vf}. The boundedness of $\alpha$ helps to obtain existence of solutions to~\eqref{eqn:weak_ot}. However, strong convergence of discrete solutions of~\eqref{eqn:weak_vf} is needed to obtain convergence of~\eqref{eqn:weak_vel} and~\eqref{eqn:weak_ot}. It is readily noted that the bounds on $\alpha$, $u$, and $c$ are interdependent, and our analysis also addresses this issue.
 
\section{Discretisation}
\label{sec:discretisation}
We discretise~\eqref{eqn:vol_fraction} using a finite volume method,~\eqref{eqn:cel_velocity} using a Lagrange $\mathbb{P}^1$--finite element method, and~\eqref{eqn:oxygen_tension} using backward Euler in time and $\mathbb{P}^1$--mass lumped finite element method in space. The space and time variables are discretised as follows.
     Let $0 = x_0 < \cdots < x_{J} = {{\ell_m}}$ be a uniform spatial discretisation with $h := x_{j+1} - x_{j}$, and
     $0 = t_0 < \cdots < T_{N} = T$ be a uniform temporal discretisation with $\delta := t_{n+1} - t_{n}$. The numbers $h$ and $\delta$ are called the spatial and temporal discretisation factors.
Define the intervals $\mathcal{X}_j := (x_j,x_{j+1})$ and $\mathcal{T}_n := [t_{n},t_{n+1})$. The node--centred intervals are defined by $\widetilde{\mathcal{X}}_{j} := (x_{j} - h/2,x_{j} + h/2)$ for $j = 1,\ldots,J - 1$, $\widetilde{\mathcal{X}}_{0} := [x_0,x_0 + h/2]$, and $\widetilde{\mathcal{X}}_{J} := [x_{J} - h/2,x_{J}]$. We let $\pmb{\chi}_{\widetilde{\mathcal X}_j}$ be the characteristic function of $\widetilde{\mathcal X}_j$, that is, $\pmb{\chi}_{\widetilde{\mathcal X}_j}=1$ on $\widetilde{\mathcal X}_j$, and $\pmb{\chi}_{\widetilde{\mathcal X}_j}=0$ outside $\widetilde{\mathcal X}_j$. For any real valued function $f$ on $\mathbb{R}$, define the pointwise average $\{\!\!\{f\}\!\!\}_{\mathcal{X}_{j}} = (f(x_{j}) + f(x_{j+1}))/2$.   Define the extended initial data as follows: $\forall\;x \in (0,\ell_m)$
\begin{align}
\alpha_{0}^{\mathrm{e}}(x) := 
\left\{
\begin{array}{c l}
\alpha_0(x) & \text{if}\;\;x \in \Omega(0), \\
0 &  \text{otherwise}.
\end{array}
\right.
\text{ and }\;\;
c_{0}^{\mathrm{e}}(x) := 
\left\{
\begin{array}{c l}
c_0(x) & \text{if}\;\;x \in \Omega(0), \\
1 &  \text{otherwise}.
\end{array}
\right.
\end{align}

\begin{discrete}
Define 
\begin{itemize}
	\setlength\itemsep{-0.1em}
	\item $\alpha_h^0$ by $\alpha_{h}^0 := \alpha_j^{0} = {}{} \frac{1}{h}\int_{\mathcal{X}_j} \alpha_0^{\mathrm{e}}(x)\,\mathrm{d}x$ on $\mathcal{X}_j$ for $0 \leq j \leq J - 1$,
	\item $c_h^0$ by $c_{h}^0 \in \mathbb{P}^1(\mathcal{X}_j)$ for $0 \leq j \leq J-1$ and $c_h^0(x_j) := c_0^{\mathrm{e}}(x_j)$ for $0 \leq j \leq J$, and
	\item $\Omega_{h}^0 := (0,\ell_{h}^0)$, where $\ell_h^0 = 1$.
\end{itemize}
Fix a threshold $\ath\in (0,1)$ and $\ell_m>\ell_0$ such that $(0,\ell_0) \subset (0,\ell_m)$ and $D_T^{\thr} \subset \mathscr{D}_T$.
Obtain $u_h^0$ from \ref{ds.c} below by taking $n = 0$. Then, construct a finite sequence of 3--tuple of functions $\left(\alpha_{h}^n,u_{h}^n,c_{h}^n\right)_{ \{0 < n \leq N-1 \}}$ on $(0,\ell_m)$ as in \ref{ds.a}--\ref{ds.d} described now.
\label{defn:semi_disc_soln}
\begin{enumerate}[label= $\mathrm{(DS.\alph*)}$,ref=$\mathrm{(DS.\alph*)}$,leftmargin=\widthof{(DS.a)}+1.2\labelsep]
	 \item \label{ds.a} Set $\alpha_{h}^n := \alpha_j^{n}$ on $\mathcal{X}_j$ for $0 \leq j \leq J-1$, where \vspace{-0.7em}
		\begin{multline}
	\dfrac{1}{\delta} (\ats{n}{j} - \ats{n-1}{j}) \\
	+ \dfrac{1}{h}{}\left[\up{(n-1)}{j+1}\ats{n-1}{j} - \um{(n-1)}{j+1}\ats{n-1}{j+1} - \up{(n-1)}{j}\ats{n-1}{j-1} + \um{(n-1)}{j}\ats{n-1}{j}\right]  \\
	= (\ats{n-1}{j} - \ath)^{+}(1-\ats{n-1}{j}) b_{j}^{n-1} - (\ats{n}{j} - \ath)^{+}d_{j}^{n-1},
	\label{eqn:upwind}
	\end{multline}
where $u_j^n = u_h^n(x_j)$, $b_{j}^n = \dcurly{(1 + s_1)c_h^n/(1 + s_1c_{h}^n)}_{\mathcal{X}_j}$, and $d_{j}^n = \dcurly{(s_2 + s_3c_h^n)/(1 + s_4c_{h}^n)}_{\mathcal{X}_j}$. Note that, when $j=0$, $u_0^{(n-1)}=0$ and thus the value of $\ats{n-1}{-1}$  can be arbitrarily fixed, say for example $\ats{n-1}{-1} = m_{11}$.
	
	\item \label{ds.b} Set \label{item:al_cut_off} $\Omega_h^n := (0,\ell_{h}^n)$, where the recovered radius at step $n$, $\ell_{h}^{n}$, is provided by  $\ell_{h}^{n} = \min\{ x_j :  \ats{n}{j} < \ath\;\text{ on } (x_j,{{\ell_m}})\}$. 
	
 \item \label{ds.c} Set the conforming $\mathbb{P}^1$ finite element space on $\Omega_h^n$, and its subspace with homogeneous boundary condition at $x = 0$, by
    \begin{align}
    \label{eqn:p1_space}
       \mathcal{S}_{h}^n &:= \left\{v_{h}^n \in \mathscr{C}^0(\overline{\Omega_{h}^n}):\, v_{h\vert_{\mathcal{X}_j}}^n \in \mathbb{P}^1(\mathcal{X}_j) 
        \text{ for }  0 \leq j < J_n:= \ell_h^n/h \right\} \text{ and }\ \\
        \label{eqn:p1_vel_space}
        \mathcal{S}_{0,h}^n &:= \{v_h^n \in \mathcal{S}_{h}^n : v_h^n(0) = 0\}.
    \end{align}
    Then, \vspace{-0.5cm}
	\begin{align}
		u_h^n := \left\{\begin{array}{l l}
    	\widetilde{u}_h^n  & \text{ on } \Omega_h^n, \\
    0 & \text{ on } (0,L)\setminus\overline{\Omega_h^n},
    \end{array}\right. 
    \label{eqn:uhn}
	\end{align}
	where $\widetilde{u}_h^n \in S_{0,h}^n$ satisfies 
    \begin{equation}
        a_{h}^n(\widetilde{u}_h^n,v_{h}^n) = \mathcal{L}_h^n(v_{h}^n)\quad \forall\,v_h^n \in S_{0,h}^n,
        \label{eqn:dweak_vel}
    \end{equation}
    with $a_h^n : \mathcal{S}_h^n \times \mathcal{S}_h^n \rightarrow \mathbb{R}$ and $\mathcal{L}_h^n : \mathcal{S}_h^n \rightarrow \mathbb{R}$ defined by
    \begin{align}
\label{eqn:dbform_vel}
a_h^n(w,v) &= k\left( \frac{\alpha_h^n}{1-\alpha_h^n}w, v \right)_{\Omega_{h}^n} + \mu\left(\alpha_h^n \partial_x w,\partial_x v \right)_{\Omega_{h}^n} \text{ and } \\
\label{eqn:dlform_vel}
\mathcal{L}_h^n(v) &= \left(\mathscr{H}(\alpha_h^n),  \partial_x v\right)_{\Omega_{h}^n}.
\end{align}  
\item \label{ds.d} Define the finite dimensional vector spaces
\begin{align}
\setlength\itemsep{-0.5em}
 \label{eqn:p1_ot_space}
    \mathcal{S}_{h,0}^n &:= \{v_h^n \in \mathcal{S}_{h}^n : v_h^n(\ell_h^n) = 0\} \text{ and } \\
    \label{eqn:p0_space}
    \mathcal{S}_{h,\ml} &:= \bigg\{ w_{h} : w_{h} = \sum_{j=0}^J w_j \pmb{\chi}_{\widetilde{\mathcal{X}}_{j}},\,w_j \in \mathbb{R},\,0 \leq j \leq J \bigg\},
\end{align}
and the mass lumping operator $\Pi_h : \mathscr{C}^0([0,\ell_m]) \rightarrow \mathcal{S}_{h,\ml}$ such that 
$\Pi_h w = \sum_{j=0}^{J} w(x_j) \pmb{\chi}_{\widetilde{\mathcal{X}}_{j}}$.
Then, \vspace{-0.5cm}
\begin{align}
	c_h^n := \left\{\begin{array}{l l}
	\widetilde{c}_h^n &\text{ on } \Omega_h^n, \\
	1 & \text{ on } (0,\ell_m)\setminus \overline{\Omega_h^{n}},
\end{array}\right.
\label{eqn:chn}
\end{align}
	 where $\widetilde{c}_h^{n}\in \mathcal{S}_h^{n}$ satisfies the boundary condition $\widetilde{c}_h^{n}(\ell_h^n) = 1$ and the following discrete equation, in which $\Pi_h \widetilde{c}_h^n := \left( \Pi_h c_h^n \right)_{\vert \Omega_h^n}$: for all $v_h^{n} \in S_{h,0}^{n}$, it holds
\begin{align}
    (\Pi_h \widetilde{c}_h^{n}, \Pi_h v_{h}^{n})_{\Omega_h^n} &- (\Pi_h c_h^{n-1}, \Pi_h v_{h}^{n})_{\Omega_h^n} + \delta \lambda (\partial_{x} \widetilde{c}_h^{n}, \partial_{x} v_{h}^{n})_{\Omega_h^n} \nonumber \\
    &= -Q\delta \left(\dfrac{\alpha_{h}^n\,\Pi_h \widetilde{c}_h^{n}}{1 + \widehat{Q}_1 \left|\Pi_h c_h^{n-1}\right|}, \Pi_h v_{h}^{n} \right)_{\Omega_h^n}.
    \label{eqn:bform_ot}
\end{align}
\end{enumerate}
\end{discrete}

\edit{The Discrete scheme~\ref{defn:semi_disc_soln} provides a family of discrete spatial functions at each time index $n$, $0 \le n < N$, from which a time--space function can be reconstructed.}

\begin{definition}[Time--reconstruct]
For a family of functions $(f_h^n)_{\{ 0\leq n < N \}}$ on a set $X$, define the time--reconstruct $f_{h,\delta} : (0,T)\times X \rightarrow \mathbb{R}$ as $f_{h,\delta} := f_{h}^n$ on $\mathcal{T}_n$ for $0 \leq n < N$.
\label{defn:time-recon}
\end{definition}
\begin{definition}[Discrete solution]
\label{defn:disc_soln}
The 4-tuple $(\alpha_{h,\delta},u_{h,\delta},c_{h,\delta},\ell_{h,\delta})$, where $\alpha_{h,\delta}$, $u_{h,\delta}$, $c_{h,\delta}$, and $\ell_{h,\delta}$ are the respective time--reconstructs corresponding to the families $(\alpha_{h}^n)_n,\,(u_{h}^n)_n,\,(c_{h}^n)_n$, and $(\ell_h^n)_n$ obtained from~\ref{ds.a}--\ref{ds.d}, is called the discrete threshold solution.
\end{definition}

\subsection{Comments on the numerical method}
\label{sec:numcom}
This subsection substantiates the particular choices of numerical methods used to compute the discrete solution in Definition~\ref{defn:disc_soln}.  
\medskip\\
\noindent \textbf{Volume fraction equation}\medskip\\
The volume fraction equation~\eqref{eqn:vol_fraction} is a continuity equation with the source term $\check{\alpha} f(\check{\alpha},\check{c})$, and the conserved variable $\check{\alpha}$ (see Lemma~\ref{lemma:mass_con}) is transported with a velocity $\check{u}$. Finite volume methods are the natural choice of numerical methods that  preserve conservation property at the discrete level~\cite{leveque_2002}. An upwinding finite volume scheme is used in~\eqref{eqn:upwind}. Upwinding treats the boundary flux values differently depending on the direction (sign) of the velocity as in~\eqref{eqn:upwind_flux}, see~\cite[p.~159, Eq. (6.7)]{eymard}. If velocity at the node $x_j$ is positive (resp.\ negative), then the material towards that node is upwinded from the control volume $\mathcal{X}_{j-1}$ (resp.\ $\mathcal{X}_j$).  This means that the flux at the boundary $x_{j}$ between any two intervals $\mathcal{X}_{j-1}$ and $\mathcal{X}_j$ is approximated by: for any $t \in (0,T)$
\begin{align}
    (u_{c}\alpha)(t,\cdot)_{\vert{x_j}} \approx u_{h,\delta}(t,x_{j})^{+} \alpha_{h,\delta}(t,\cdot)_{\vert{\mathcal{X}_{j-1}}} - u_{h,\delta}(t,x_{j})^{-} \alpha_{h,\delta}(t,\cdot)_{\vert{\mathcal{X}_{j}}}.
    \label{eqn:upwind_flux}
\end{align}
Therefore, the spatial difference $(u_{c}\alpha)(t,\cdot)_{\vert{x_{j+1}}} - (u_{c}\alpha)(t,\cdot)_{\vert{x_{j}}}$ at $t = t_{n-1}$ is approximated as
\begin{align}
(u_{c}\alpha)(t,\cdot)_{\vert{x_{j+1}}} - (u_{c}\alpha)(t,\cdot)_{\vert{x_{j}}} \approx{}& \left( u_{j+1}^{(n-1)+}\ats{n-1}{j} - u_{j+1}^{(n-1)-} \ats{n-1}{j+1} \right) \\
&- \left( u_{j}^{(n-1)+} \ats{n-1}{j-1} - u_{j}^{(n-1)-} \ats{n-1}{j} \right),
\label{eqn:up_spatial}
\end{align}
which leads to~\eqref{eqn:upwind}.  The upwinding flux~\eqref{eqn:upwind_flux} is one of the simplest numerical fluxes that leads to a stable scheme. 

The upwind method~\eqref{eqn:upwind} introduces significant numerical diffusion in the discrete solution $\alpha_{h,\delta}$. Hence, if we locate the time-dependent boundary $\ell_{h}^n$ as $\min\{x_j: \alpha_{h}^n = 0 \text{ on } (x_j,{{\ell_m}}]\}$, then $\ell_{h,\delta}$ will have notable deviation from the exact solution, which will further tamper the quality of the solutions $u_{h,\delta}$ and $c_{h\,\delta}$. To eliminate this propagating error, the boundary point $\ell_{h}^n$ is located by $\min\{x_j: \alpha_{h}^n < \ath \text{ on } (x_j,{{\ell_m}}]\}$ (see Figure~\ref{fig:alpha_thrfig}). However, the residual volume fraction of $\ath$ on $[\ell_h^n,\ell_m]$ might cause the reaction term $\check{\alpha} f(\check{\alpha},\check{c})$ to contribute a spurious growth; the modification $(\check{\alpha} - \ath)^{+}f(\check{\alpha},\check{c})$ overcomes this problem. More importantly, $\ath$ acts as a lower bound on the value of $\alpha_{h,\delta}$ on $\mathcal{X}_{J_n - 1}$ (the right most control volume in $(0,\ell_h^n)$) at each time $t_n$. A detailed numerical study of the dependence of the discrete solution on $\ath$ and the optimal choice of $\ath$ that minimises the error incurred in $\ell_{h,\delta}$ is done in \cite{remesan_1}.
\medskip\\
\noindent \textbf{Velocity equation}\medskip\\
The velocity equation~\eqref{eqn:cel_velocity} is elliptic with Dirichlet boundary condition at $x = 0$ and Neumann boundary condition at $x = \ell_{h}^n$ for each $t_n$, and hence the Lagrange $\mathbb{P}^1$ finite element method is used to discretise~\eqref{eqn:cel_velocity}. A specific benefit of using conforming finite elements for approximating the velocity is that it naturally provides nodal values (degrees of freedom of the scheme) of $u_{h,\delta}$ at the boundaries of each $\mathcal{X}_j$; these nodal velocities can be directly used in the finite volume discretisation of~\eqref{eqn:upwind_flux} to compute fluxes at the control volume interfaces.
\medskip\\
\noindent \textbf{Oxygen tension equation}\medskip\\
The choice of time--implicit mass lumped finite element method~\cite[Section 7.3.5]{droniou2018gradient} for the oxygen tension equation~\eqref{eqn:oxygen_tension} is substantiated mainly by two reasons. Firstly, the choice of mass lumping as opposed to a standard Lagrange $\mathbb{P}^1$  finite element method is important to obtain a discrete maximum principle for $c_{h,\delta}$. Secondly, the backward time procedure ensures the $L^2(0,T;H^1(0,{{\ell_m}}))$ stability of the mass lumped solutions. This is essential to prove  Propositions~\ref{prop:phi_c_compact} and~\ref{prop:step_ce5} that lead to the compactness and convergence of the iterates.  \edit{Also, the mass lumping operator $\Pi_h$ used in~\ref{ds.d} preserves the $L^1$ norm of a piecewise linear function, and thus only locally redistributes the total amount of material whose concentration is specified by $c_{h,\delta}(t,\cdot)$ at each time $t \in (0,T)$. }

\section{Main theorems}
\label{sec:main_thm}
Define the function $\widehat{u}_{h,\delta}$ on $\mathscr{D}_T$ such that for every $t \in (0,T)$,
\begin{equation}
\widehat{u}_{h,\delta}(t,\cdot) = \left\{
\begin{array}{l l}
u_{h,\delta}(t,\cdot) & \text{ in }(0,\ell_{h,\delta}(t)], \\
u_{h,\delta}(t,\ell_{h,\delta}(t)) & \text{ in } (\ell_{h,\delta}(t),{{\ell_m}}).
\end{array}
\right.
\label{eqn:uhd_const}
\end{equation}
 \edit{The function $\widehat{u}_{h,\delta}$ is the constant extension of $u_{h,\delta}(t,\cdot)$  to $(\ell_{h,\delta}(t),\ell_m)$. Note that $\widehat{u}_{h,\delta}$ is continuous on the contrary to $u_{h,\delta}$ (see Figure~\ref{fig:uhhatuh}). This continuity is necessary to ensure the existence of a square integrable weak derivative.}

\begin{figure}[h!] 
	\centering
	\begin{subfigure}[b!]{6.5cm}
		\includegraphics[scale=0.8]{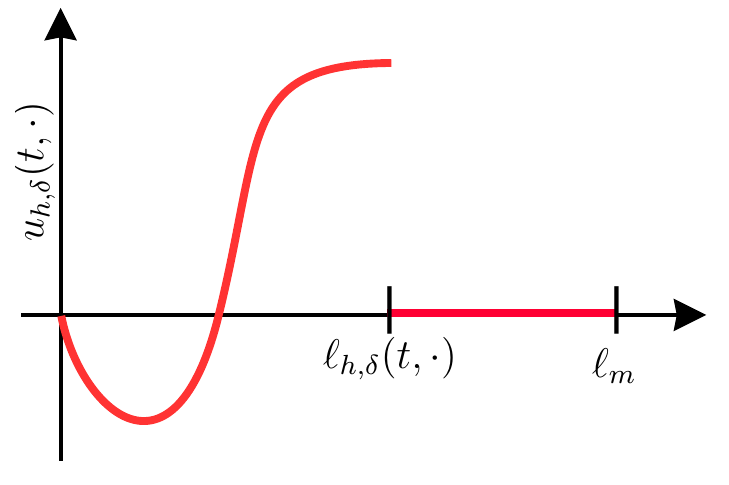}
		\caption{$u_{h,\delta}(t,\cdot)$.}
		\label{fig:uhdelta}
	\end{subfigure}
	\begin{subfigure}[b!]{6.5cm}
		\includegraphics[scale=0.8]{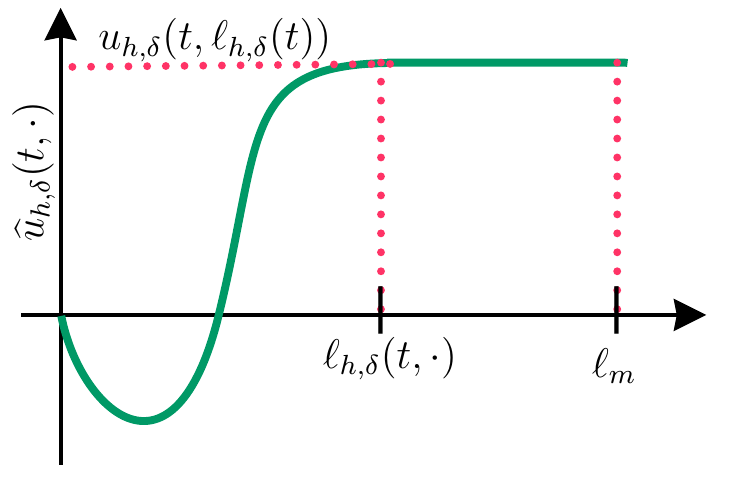}
		\caption{$\widehat{u}_{h,\delta}(t,\cdot)$.}
		\label{fig:hatuhdelta}
	\end{subfigure}
	\caption{The left--hand side plot illustrates the discontinuous function $u_{h,\delta}$ and the right--hand side plot illustrates the continuous modification $\widehat{u}_{h,\delta}$.}
	\label{fig:uhhatuh}
\end{figure}

\edit{The notation $\Pi_{h,\delta}c_{h,\delta}$ denotes the mass lumping operator $\Pi_h$ applied to $c_{h,\delta}(t,\cdot)$ for each $t \in (0,T)$.} Define the Hilbert spaces: 
\begin{align*}
L_{c}^2(0,T;H^1(0,{{\ell_m}})) :={}& \{f \in L^2(0,T;H^1(0,{{\ell_m}}))\,:\, f(t,\ell(t)) = 0\; \mbox{ for a.e. } t\in(0,T) \},\\
L_{u}^2(0,T;H^1(0,{{\ell_m}})) :={}& \{f \in L^2(0,T;H^1(0,{{\ell_m}}))\,:\, f(t,0) = 0\;\mbox{ for a.e. } t\in (0,T) \}.
\end{align*}
The main results of this article are stated in Theorem~\ref{thm:main_thm_1} and~\ref{thm:main_thm_2}.

\begin{theorem}[Compactness]
\label{thm:main_thm_1}
Let the properties stated below be true.
\begin{itemize}
    \item The initial volume fraction $\alpha_{0}$ belongs to $BV(0,\ell_m)$ and satisfies \eqref{eq:bd.alpha0}.
    \item The discretisation parameters $h$ and $\delta$ satisfy the following conditions:
    \begin{small}
		\begin{equation}      
		\label{eqn:cfl_condition}
    \begin{aligned}
      \rho\,\mathscr{C}_{\CFL} \leq \dfrac{\delta}{h} \le \mathscr{C}_{\CFL} := \dfrac{\sqrt{a_{\ast}}\mu}{2{{\ell_m}}}\dfrac{|1 - a^{\ast}|^2}{|a^{\ast} - \alast|} \ \text{ and }\ \delta < \min\left(\dfrac{1 - \rho}{s_2},\,\dfrac{2(1 - \rho)}{1 + s_2}\right),
    \end{aligned}
		\end{equation}
	\end{small}
    where $\rho,\,a_{\ast}$ and $a^{\ast}$ are constants chosen such that $\rho < 1$, $0 < a_{\ast} < m_{01}$, and  $m_{02} < a^{\ast}$. 
\end{itemize}
\edit{Then, there exists a finite time $T_{\ast}$ depending on the choice of $\rho,\,a_{\ast}$, and $a^{\ast}$, a subsequence (denoted by the same indices as of the sequence) of the family of functions $\{(\alpha_{h,\delta},\widehat{u}_{h,\delta},c_{h,\delta},\ell_{h,\delta})\}_{h,\delta}$,  and a 4-tuple of functions $(\alpha,\widehat{u},c,\ell)$}  such that, setting $\mathscr{D}_{T_{\ast}} = (0,T_{\ast}) \times (0,\ell_m)$, it holds 
\[
\alpha \in BV(\mathscr{D}_{T_{\ast}}),\;c \in L_{c}^2(0,T_{\ast};H^1(0,{{\ell_m}})),\,\widehat{u} \in L_u^2(0,T_{\ast};H^1(0,\ell_m)),\;\ell \in BV(0,T_{\ast}), 
\]
and as $h,\,\delta \rightarrow 0$,
\begin{itemize}
		\setlength\itemsep{-0.4em}
	\item $\alpha_{h,\delta} \rightarrow \alpha$ almost everywhere and in $L^\infty$-weak$^{\,\star}$ on $\mathscr{D}_{T_{\ast}}$, 
	\item $\Pi_{h,\delta} c_{h,\delta} \rightarrow c$ strongly in $L^2(\mathscr{D}_{T_{\ast}})$ and $\partial_x c_{h,\delta} \halfarrow \partial_x c $ weakly in $L^2(\mathscr{D}_{T_{\ast}})$,
	\item $\widehat{u}_{h,\delta} \halfarrow \widehat{u}$ and $\partial_x \widehat{u}_{h,\delta} \halfarrow \partial_x \widehat{u}$ weakly in $L^2(\mathscr{D}_{T_{\ast}})$, and
	\item $\ell_{h,\delta} \rightarrow \ell$ almost everywhere in $(0,T_{\ast})$.
\end{itemize}
\end{theorem}
\begin{theorem}[Convergence]
Let $(\alpha,\widehat{u},c,\ell)$ be the limit of any subsequence of the numerical approximations $\{(\alpha_{h,\delta},\widehat{u}_{h,\delta},c_{h,\delta},\ell_{h,\delta})\}_{h,\delta}$, in the sense of Theorem~\ref{thm:main_thm_1}.
Define $\Omega(t) := (0,\ell(t))$ and the threshold domain $D_{T_\ast}^{\thr} :=  \{(t,x) : x < \ell(t), t \in (0,T_{\ast})\}$, and let $u := \widehat{u}$ on $D_{T_\ast}^{\thr}$ and $u := 0$ on $\mathscr{D}_{T_{\ast}}\setminus D_{T_\ast}^{\thr}$. Then, $(\alpha,u,c,\Omega)$ is a threshold solution in the sense of Definition~\ref{defn:ext_soln} with $T = T_{\ast}$.
	\label{thm:main_thm_2}
\end{theorem}

\begin{remark}[Convergence up to a subsequence]
	In the rest of the article, unless otherwise specified, ``convergence'' of sequences is to be understood up to a subsequence. Hence ``a sequence $(a_n)_n$ converges to a limit $a$'' means that there exists a subsequence $(a_{k_n})_n \subseteq (a_n)_n$ such that $(a_{k_n})_n$ converges to $a$. This concept is classical when analysing the convergence of numerical approximations of non--linear equations, see, e.g., \cite[Section 4.5]{tadmor}, \cite[Section 5.2]{Di-Pietro.Ern:12} or \cite[Chap. 5, 6]{droniou2018gradient}.
\end{remark}

\begin{remark}[Existence of a solution]
	Existence of a threshold solution is ensured by Theorems \ref{thm:main_thm_1} and \ref{thm:main_thm_2}. Theorem \ref{thm:main_thm_2} also shows that if convergence is observed in a numerical simulation, then the limit is necessarily a solution to the threshold model. Finally, as usual in convergence by compactness arguments, if the solution to this model is proved to be unique then the entire sequence of approximations (not just a subsequence) converges to that solution.
\end{remark}

\section{Proof of Theorem~\ref{thm:main_thm_1}}
\label{sec:proof_main_1}
The proof of Theorem~\ref{thm:main_thm_1} involves several steps, which are described here. In Subsection~\ref{sec:Well_posedness}, we prove the following:
\begin{itemize}
		\setlength\itemsep{-0.4em}
	\item[--] existence and uniqueness of the discrete solutions $\alpha_{h,\delta}$, $u_{h,\delta}$, and $c_{h,\delta}$,
	\item[--] boundedness of $u_{h,\delta}$ in various norms,
	\item[--] positivity, boundedness, and bounded variation property of $\alpha_{h,\delta}$, and
	\item[--] positivity and boundedness of $c_{h,\delta}$.
\end{itemize}
 In Subsection~\ref{sec:compactness_est}, we show that the families of functions $\{\alpha_{h,\delta}\}_{h,\delta},\,\{u_{h,\delta}\}_{h,\delta},$ $\{c_{h,\delta}\}_{h,\delta}$, and $\{\ell_{h,\delta}\}_{h,\delta}$ are relatively compact in appropriate spaces.

\subsection{Existence and uniqueness of the iterates}
\label{sec:Well_posedness}
The proof of existence and uniqueness of the discrete solutions $\alpha_{h,\delta}$, $u_{h,\delta}$, and $c_{h,\delta}$ involves many interrelated results. For clarity, we provide a sketch of the steps involved. 

Fix two constants $a^\ast\in (\max(\alast,m_{02}),1)$ and $a_\ast\in  (0,\min(m_{01},\ath))$. We establish the existence of a time $T_*$ (explicitly determined in the analysis), which depends in particular on $a_\ast$ and $a^\ast$, such that the following theorem holds.

\begin{theorem}
For all $n\in \mathbb{N}$ such that $t_n\le T_*$, $\alpha_{h,\delta}(t_n,\cdot)$ and $c_{h,\delta}(t_n,\cdot)$ are well defined. Also, it holds $a_{\ast} < \alpha_{h,\delta}(t_n,\cdot)_{\vert\Omega_{h}^n} < a^{\ast}$ and $0 \leq c_{h,\delta}(t_n,\cdot)_{\vert(0,\ell_m)} \leq 1$. 
\label{thm:induction}
\end{theorem}

	The proof of Theorem~\ref{thm:induction} is done in several steps by strong induction on $n\in \mathbb{N}$. The base case obviously holds, for any choice of $a_\ast$ and $a^\ast$ as above.
	Let $n\in\mathbb{N}$ be such that $t_{n+1} \le T_*$, and assume that Theorem~\ref{thm:induction} holds for the indices $0,\ldots,n$. The inductive steps \ref{is.1}--\ref{is.4} below show that the same holds for the index $n+1$.
	
	In the sequel, $\mathscr{C}$ is a generic constant that depends on $T,\,{{\ell_m}},\,\ell,\,\alast,\,a_{\ast},a^{\ast}$ and the model parameters, as explicitly defined in~\eqref{eqn:vel_dbv}--\eqref{eqn:up_bound}.
		
\begin{enumerate}[label= $\mathrm{(IS.\arabic*)}$,ref=$\mathrm{(IS.\arabic*)}$,leftmargin=\widthof{(IS.4)}+1.2\labelsep]
    \item\label{is.1} We establish that there exists a unique solution~$\widetilde{u}_{h}^{n}$ for the variational problem~\eqref{eqn:dweak_vel} and derive energy estimates. 
    \item\label{is.2} \textbf{Bounded variation and $L^\infty$ estimates on $\alpha_{h,\delta} u_{h,\delta}$:} We show that
\begin{itemize}
    \item[(a)]  $||\mu \alpha_{h,\delta}(t_n,\cdot) \partial_x u_{h,\delta}(t_n,\cdot) - \mathscr{H}(\alpha_{h,\delta}(t_n,\cdot)) ||_{BV(0,{{\ell_m}})} \le \mathscr{C}$, 
    \item[(b)] $||( \mu \alpha_{h,\delta}(t_n,\cdot) \partial_x u_{h,\delta}(t_n,\cdot))^{-}||_{L^\infty(0,{{\ell_m}})} \le \mathscr{C}$, \text{ and }
    \item[(c)] $||\mu \alpha_{h,\delta}(t_n,\cdot) \partial_{x} u_{h,\delta}(t_n,\cdot)||_{L^\infty(0,{{\ell_m}})} \le \mathscr{C}$,
\end{itemize}
where $\mathscr{H}(\alpha) = \alpha(\alpha - \alast)^{+}/(1 - \alpha)^2$. 
\item\label{is.3} \textbf{$L^\infty$ estimates on $\alpha_{h,\delta}$: } It holds $a_{\ast} < \alpha_{h,\delta}(t_{n+1},\cdot)\vert_{\Omega_{h}^{n+1}} < a^{\ast}$.
\item\label{is.4} We show that there exists a unique solution $\widetilde{c}_{h,\delta}(t_{n+1},\cdot)$ to~\eqref{eqn:bform_ot} and that  $0 \leq \widetilde{c}_{h,\delta}(t_{n+1},\cdot)_{\vert(0,\ell_m)} \leq 1$.
\end{enumerate}
The steps (IS.1)--(IS.4) are now performed in Lemmas~\ref{lemma:dvel_exis},~\ref{lemma:dvel_bv},~\ref{lemma:dot_max} and Proposition~\ref{prop:dalpha_max}, respectively. The time $T_*$ is explicitly obtained in the proof of Proposition \ref{prop:dalpha_max}.

\begin{lemma}[Step \ref{is.1}]
\label{lemma:dvel_exis}
 There exists a unique solution $\widetilde{u}_h^n$ to~\eqref{eqn:dweak_vel} and it satisfies the following estimates:
	\begin{multline}
\left|\left| \sqrt{\alpha_{h,\delta}(t_n,\cdot)}\partial_{x}\widetilde{u}_h^n \right|\right|_{0,\Omega_{h}^n} + \left|\left| \dfrac{\sqrt{\alpha_{h,\delta}(t_n,\cdot)} \widetilde{u}_h^n}{\sqrt{1 - \alpha_{h,\delta}(t_n,\cdot)}} \right|\right|_{0,\Omega_{h}^n} \le \left(1 + \frac{1}{\sqrt{k}}\right) \sqrt{\dfrac{{{\ell_m}}}{\mu}}  \dfrac{|a^{\ast} - \alast|}{|1 - a^{\ast}|^2}.
\label{eqn:vel_d_ene_1}
	\end{multline}
\end{lemma}
\begin{proof}
Coercivity and continuity of the bilinear form $a_h^n$ and continuity of the linear form $\mathcal{L}_h^n$ are clear from $0 < a_{\ast} \leq \alpha_{h,\delta}(t_n,\cdot) \leq a^{\ast} < 1$. An application of the Lax--Milgram lemma~\cite[p.~297]{EvansPDE} establishes the existence of a unique discrete solution to~\eqref{eqn:dweak_vel}. A choice of $v_{h}^n = \widetilde{u}_h^n$ in~\eqref{eqn:dweak_vel}, the fact that $0 < \alpha_{h,\delta}(t_n,\cdot) < 1$, and Cauchy--Schwarz inequality in~\eqref{eqn:dlform_vel} yield 
\begin{multline*}
    \mu \left|\left| \sqrt{\alpha_{h,\delta}(t_n,\cdot)}\partial_{x}\widetilde{u}_h^n \right|\right|_{0,\Omega_{h}^n}^2 + k\left|\left| \dfrac{\sqrt{\alpha_{h,\delta}(t_n,\cdot)} \widetilde{u}_h^n}{\sqrt{1 - \alpha_{h,\delta}(t_n,\cdot)}} \right|\right|_{0,\Omega_{h}^n}^2 \\
    \leq \sqrt{{{\ell_m}}} \dfrac{|a^{\ast} - \alast|}{|1 - a^{\ast}|^2} \left|\left| \sqrt{\alpha_{h,\delta}(t_n,\cdot)}\partial_{x}\widetilde{u}_h^n \right|\right|_{0,\Omega_{h}^n},
\end{multline*}
which proves~\eqref{eqn:vel_d_ene_1}. 
\end{proof}

\begin{remark}[$L^\infty$ estimate on velocity]
	Since $\alpha_{h,\delta}(t_n,\cdot)\ge a_\ast$, the estimate  \eqref{eqn:vel_d_ene_1} yields an upper bound on $||\partial_x \widetilde{u}_h^n||_{0,\Omega^n_h}$, which after an application of the boundary condition $\widetilde{u}_h^n(0) = 0$  and a use of Cauchy--Schwarz inequality yields
	\begin{equation}
	||u_{h,\delta}(t_n,\cdot)||_{L^\infty(0,{{\ell_m}})} \leq \dfrac{{{\ell_m}}}{\sqrt{a_{\ast}}\mu} \dfrac{|a^{\ast} - \alast|}{|1 - a^{\ast}|^2}.
	\label{eqn:u_bound}
	\end{equation}
\end{remark}

\begin{lemma}[Step \ref{is.2}]
	\label{lemma:dvel_bv}
	It holds that
	\begin{subequations}
	\begin{align}
	&\hspace{-0.5cm}|| \mu \alpha_{h,\delta}(t_n,\cdot) \partial_x u_{h,\delta}(t_n,\cdot) - \mathscr{H}(\alpha_{h,\delta}(t_n,\cdot)) ||_{BV(0,\ell_m)} \leq {{\ell_m}} \sqrt{\dfrac{k}{\mu}} \dfrac{|a^{\ast} - \alast|}{|1 - a^{\ast}|^{5/2}},
	\label{eqn:vel_dbv}\\
	\label{eqn:um_bound}
	&\hspace{-0.5cm}||( \mu \alpha_{h,\delta}(t_n,\cdot) \partial_x u_{h,\delta}(t_n,\cdot))^{-}||_{L^\infty(0,{{\ell_m}})} 
	\leq  {{\ell_m}} \sqrt{\dfrac{k}{\mu}} \dfrac{|a^{\ast} - \alast|}{|1 - a^{\ast}|^{5/2}}, \text{ and }\\
	&\hspace{-0.5cm}||\mu \alpha_{h,\delta}(t_n,\cdot) \partial_x u_{h,\delta}(t_n,\cdot)||_{L^\infty(0,{{\ell_m}})} \leq {{\ell_m}} \sqrt{\dfrac{k}{\mu}} \dfrac{|a^{\ast} - \alast|}{|1 - a^{\ast}|^{5/2}} + \dfrac{a^{\ast}(a^{\ast} - \alast)}{(1 - a^{\ast})^2}.
	\label{eqn:up_bound}
	\end{align}
	\end{subequations}
\end{lemma}
\begin{proof}
Consider the Lagrange $\mathbb{P}^1$ nodal basis functions $\{ \varphi_{h,j}^n \}_{\{1 \leq j \leq J_n \}}$ of $\mathcal{S}_{0,h}^n$, and choose $v_h^n = \varphi_{h,j}^n$ in~\eqref{eqn:dweak_vel} for $j\in\{1,\ldots,J_n-1\}$, where $J_n = \ell_h^n/h$, to obtain
\begin{subequations}
\begin{multline}
\mu \left(\alpha_{j-1}^n \partial_{x} \widetilde{u}_{h\vert_{\mathcal{X}_{j-1}}}^n  - \alpha_{j}^n \partial_{x} \widetilde{u}_{h\vert_{\mathcal{X}_{j}}}^n  \right) - \left(\mathscr{H}(\alpha_{j}^n) - \mathscr{H}(\alpha_{j-1}^n){} \right)  \\
= -k\int_{x_{j-1}}^{x_{j+1}} \dfrac{\alpha_{h,\delta}(t_n,\cdot)}{1 - \alpha_{h,\delta}(t_n,\cdot)} \widetilde{u}_h^n  \varphi_{h,j}^n\,\mathrm{d}x. 
\label{eqn:bv_1}
\end{multline}
Choose $v_h^n = \varphi_{h,J_n}^n$ in~\eqref{eqn:dweak_vel} to obtain
\begin{equation}
\mu \alpha_{j}^n \partial_{x} \widetilde{u}_{h\vert_{\mathcal{X}_{J_n-1}}}^n - \mathscr{H}(\alpha_{J_n-1}^n)  
= -k\int_{x_{J_n-1}}^{x_{J_n}} \dfrac{\alpha_{h,\delta}(t_n,\cdot)}{1 - \alpha_{h,\delta}(t_n,\cdot)} \widetilde{u}_h^n \varphi_{h,J_n}^n\,\mathrm{d}x. 
\label{eqn:bv_2}
\end{equation}
\end{subequations}
Recall that $u^n_h=\widetilde{u}^n_h$ on $(0,\ell_h^n)$, and that $u^n_h=0=\mathscr{H}(\alpha^n_j)$ outside this interval. Then, for any $j\in\{1,\ldots,J-1\}$, \eqref{eqn:bv_1} and \eqref{eqn:bv_2} imply
\begin{multline*}
\mu \left(\alpha_{j-1}^n \partial_{x} u_{h\vert_{\mathcal{X}_{j-1}}}^n  - \alpha_{j}^n \partial_{x} u_{h\vert_{\mathcal{X}_{j}}}^n  \right) - \left(\mathscr{H}(\alpha_{j}^n) - \mathscr{H}(\alpha_{j-1}^n){} \right)  \\
= -k\int_{x_{j-1}}^{x_{j+1}} \dfrac{\alpha_{h,\delta}(t_n,\cdot)}{1 - \alpha_{h,\delta}(t_n,\cdot)} u_h^n \varphi_{h,j}^n\,\mathrm{d}x,
\end{multline*}
where $\varphi_{h,j}^n=0$ if $j\ge J_n+1$. Then, triangle inequality, summation over  $j=1,\ldots,J-1$, Cauchy--Schwarz inequality,~\eqref{eqn:vel_d_ene_1}, and an observation that $0\le \varphi^n_{h,j-1}+\varphi^n_{h,j} \le 1$ everywhere leads to \eqref{eqn:vel_dbv}. As a consequence, since $\mu \alpha_{h,\delta}(t_n,\cdot) \partial_x u_{h,\delta}(t_n,\cdot) - \mathscr{H}(\alpha_{h,\delta}(t_n,\cdot))$ vanishes at $x=\ell_m$,
\begin{eqnarray*}
||  \mu \alpha_{h,\delta}(t_n,\cdot) \partial_x u_{h,\delta}(t_n,\cdot) - \mathscr{H}(\alpha_{h,\delta}(t_n,\cdot))||_{L^\infty(0,{{\ell_m}})} \leq {{\ell_m}} \sqrt{\dfrac{k}{\mu}} \dfrac{|a^{\ast} - \alast|}{|1 - a^{\ast}|^{5/2}}.
\end{eqnarray*}
Since $0\le \mathscr{H}(\alpha_{h,\delta}(t_n,\cdot))\le 
a^\ast(a^\ast-\alast)/(1-a^\ast)^2$, the bounds~\eqref{eqn:um_bound} and~\eqref{eqn:up_bound} follow.
\end{proof}

The positivity and boundedness of $\alpha_{h,\delta}(t_{n+1},\cdot)$ are shown next. The next proposition establishes the existence of a finite time $T_{\ast}$ such that the strong induction assumption holds in $[0,T_{\ast})$.

\begin{proposition}[Step \ref{is.3}]\label{prop:dalpha_max}
There exists $T_{\ast} > 0$ such that if $n+1 \leq N_{\ast} := T_\ast/\delta$, then 
\begin{equation*}
 a_{\ast}\le \min_{j\,:\,x_{j} \in \Omega_h^{n+1}} \alpha_{j}^{n+1} \le \max_{0 \leq j \leq J-1}\alpha_{j}^{n+1} \leq a^{\ast}.
\end{equation*}
\end{proposition}

\begin{proof}
Substitute $\up{n}{j+1} = u_{j+1}^{n} + \um{n}{j+1}$ and $\um{n}{j} = \up{n}{j} - u_{j}^{n}$ in~\eqref{eqn:upwind} written for $n+1$ instead of $n$ to obtain
\begin{align}
\ats{n+1}{j} + \delta (\ats{n+1}{j} - \ath)^{+} d_{j}^{n}  ={}& \ats{n}{j} + \delta (\ats{n}{j} - \ath)^{+}(1-\ats{n}{j}) b_{j}^{n} -\dfrac{\delta}{h} \ats{n}{j}\left(u^{n}_{j+1} - u^{n}_{j} \right) \nonumber \\
&+ \dfrac{\delta}{h} \left(\um{n}{j+1}(\ats{n}{j+1}-\ats{n}{j}) + \up{n}{j}(\ats{n}{j-1}-\ats{n}{j})\right).
\label{eqn:n1_sten}
\end{align}
Define the linear combination
\begin{multline}
\mathscr{L}(\ats{n}{j-1},\ats{n}{j},\ats{n}{j+1}) := \dfrac{\delta}{h}\up{n}{j}\ats{n}{j-1} + \left( 1 - \frac{\delta}{h}\um{n}{j+1} - \frac{\delta}{h}\up{n}{j} \right)\ats{n}{j} + \dfrac{\delta}{h}\um{n}{j+1}\ats{n}{j+1}.
\label{eqn:cvx_comb}
\end{multline}
 The conditions~\eqref{eqn:cfl_condition} and \eqref{eqn:u_bound} show that all the coefficients in~\eqref{eqn:cvx_comb} are positive, and thus this linear combination is convex. Moreover, \eqref{eqn:n1_sten} can be recast as
\begin{align}
\ats{n+1}{j} + \delta (\ats{n+1}{j} - \ath)^{+} d_{j}^{n} ={}& \mathscr{L}(\ats{n}{j-1},\ats{n}{j},\ats{n}{j+1}) + \delta (\ats{n}{j} - \ath)^{+}(1-\ats{n}{j}) b_{j}^{n} \nonumber \\
&- \delta \ats{n}{j} \partial_x u_{h\vert_{\mathcal{X}_{j}}}^n.
 \label{eqn:n1_sten_2}
\end{align}
Since $0 \le c^n_h\le 1$ (this is the induction hypothesis~\ref{is.4} at step $n$), we have $0\le d^n_j\le s_2$ and $b_j^n\ge 0$.
Then, a use of~\eqref{eqn:up_bound} and the positivity of $1-\alpha_j^n$ in~\eqref{eqn:n1_sten_2} yield
\begin{gather}
    \ats{n+1}{j} (1 + \delta s_2) \geq \min(\alpha_{j-1}^{n},\,\alpha_{j}^{n},\, \alpha_{j+1}^{n}) - \delta \fmin,
    \label{eqn:pos_bound_1}
\end{gather}
where 
\begin{equation*}
    \fmin = \ell_m \dfrac{\sqrt{k}}{\mu^{3/2}} \dfrac{|a^{\ast} - \alast|}{|1 - a^{\ast}|^{5/2}} + \dfrac{1}{\mu}\dfrac{a^{\ast}(a^{\ast} - \alast)}{(1 - a^{\ast})^2}.
\end{equation*}
Step~\ref{ds.b} implies that  $\alpha_{j-1}^n, \alpha_j^n,\alpha_{j+1}^n <  \ath$ for $j \geq J_{n} + 1$. This fact along with an observation that $u_h^{n} = 0$ in $(0,{{\ell_m}}) \setminus \overline{\Omega_h^{n}}$ ensures that the right hand side of~\eqref{eqn:n1_sten_2} is strictly bounded above by $\ath$ (the linear combination remains, and the other terms vanish); hence $\alpha_j^{n+1} < \ath$, for all $j \ge J_{n+1}$.
Thus the domain $\Omega_{h}^{n+1}$ is either a subset of  $\Omega_{h}^{n}$ or equal to $\Omega_{h}^{n} \cup \mathcal{X}_{J_{n}}$. These two cases are considered separately. 

\begin{case}[{$\Omega_{h}^{n+1} \subseteq \Omega_{h}^{n}$}: tumour does not grow in the $(n+1)^{th}$ level]
If $\Omega_{h}^{n+1} = \Omega_{h}^{n}$, the last value $\alpha^{n + 1}_{J_{n+1}-1}$ depends on $\alpha^{n}_{J_{n}-2}$, $\alpha^{n}_{J_{n}-1}$, and $\alpha^{n}_{J_{n}}$ (see Figure~\ref{fig:case_1}). The domain selection procedure \ref{ds.b}  shows  $\alpha_{J_{n+1} - 1}^{n + 1} \geq \ath$. All other values $\alpha_{j}^{n + 1}$ depend on $\alpha_{k}^{n}$ with $k \leq J_{n - 1}$, which are values inside $\Omega_h^n$. Therefore, for all $j \leq J_{n+1} - 1$, by~\eqref{eqn:pos_bound_1}
\begin{equation}
    \ats{n+1}{j} (1 + \delta s_2) \geq \min\bigg\{ \bigg(\min_{k\,:\,x_{k} \in \Omega_h^{n}} \alpha_{k}^{n} \bigg),\ath \bigg\} - \delta\fmin.
    \label{eqn:case1_bound}
\end{equation}
The same argument follows in the case  $\Omega_{h}^{n+1}\subset\Omega_h^n$ (see Figure~\ref{fig:case_12}).
\end{case}
\begin{figure}[h!]
   \centering
   \begin{subfigure}[b]{12cm}
   \centering
       \includegraphics[scale=1.0]{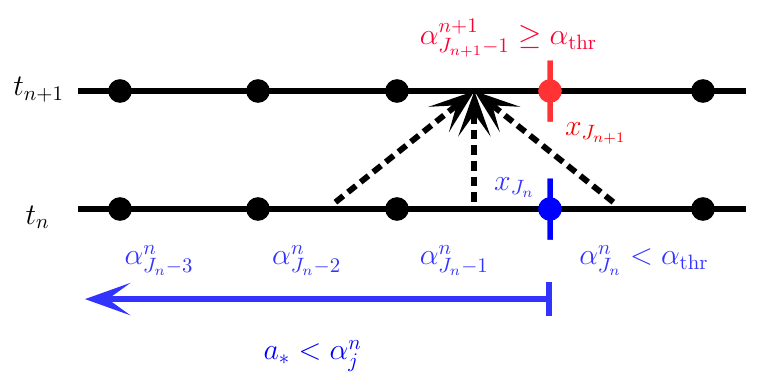}
       \caption{$\Omega_{h}^{n+1} = \Omega_{h}^n$ : observe that in 
       this case $x_{J_{n+1}} = x_{J_n}$.}
        \label{fig:case_1}
   \end{subfigure} \\
    \begin{subfigure}[b!]{6cm}
        \includegraphics[scale=1.0]{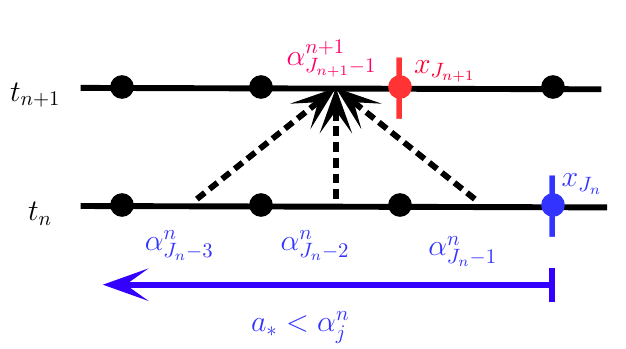}
       \caption{$\Omega_{h}^{n+1} \subset \Omega_{h}^n$.}
        \label{fig:case_12}
    \end{subfigure}
    \begin{subfigure}[b!]{6.2cm}
        \includegraphics[scale=1.0]{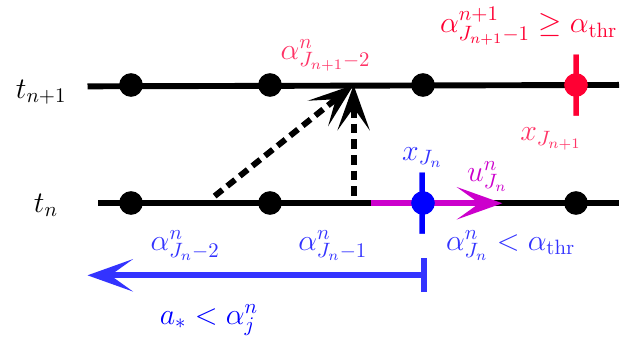}
        \caption{$\Omega_{h}^{n+1} = \Omega_h^n\cup\mathcal{X}_{J_n}$.}
        \label{fig:case_2}
    \end{subfigure}
    \caption{Dependency of $\alpha_{j}^{n+1}$ on $\alpha_{j}^n$. Observe that in Figure~\ref{fig:case_2} the direction of $u_{J_n}^n$ is rightwards, which eliminates the dependency of $\alpha_{J_{n+1}-2}^n$ on $\alpha_{J_n}^n$.}
\end{figure}

\begin{case}[{$\Omega_{h}^{n+1} = \Omega_{h}^{n} \cup \mathcal{X}_{J_{n}}$}: tumour expands]
By the domain selecting procedure \ref{ds.b} we have $\alpha_{J_{n+1}-1}^{n + 1} \geq \ath$ (see Figure~\ref{fig:case_2}). This along with $\alpha_{J_{n}}^{n} < \ath$ and $u_{j}^n = 0$ for $j > J_n$, implies that some volume fraction must flow from $\Omega_h^n$ to $\mathcal{X}_{J_n}$. This implies that $u^n_{J_n}>0$. We note here that our usage of $(\alpha-\ath)^+$ in the source term is essential to ensure this property, since the reaction term cannot yield the growth above $\ath$ in $\mathcal X_{J_n}$. Therefore, since $J_{n+1}-2=J_n-1$ in this case, choosing $j=J_n-1$ in~\eqref{eqn:n1_sten_2}, the term involving $\alpha^n_{j+1}$ vanishes from $\mathscr{L}(\ats{n}{j-1},\ats{n}{j},\ats{n}{j+1})$ (since it is multiplied by $\um{n}{J_n}$) and we obtain
\begin{equation}
    \alpha_{J_{n+1}-2}^{n+1}(1 + \delta s_2) \geq \min(\alpha_{J_{n}-2}^{n},\alpha_{J_{n}-1}^{n}) - \delta \fmin.
    \label{eqn:case2_bound}
\end{equation}
The values $\alpha_{j}^{n + 1}$ with $j \leq J_{n+1} - 3$ can be dealt as in~\eqref{eqn:case1_bound}.
\end{case}
\noindent Combine~\eqref{eqn:case1_bound} and~\eqref{eqn:case2_bound} to obtain, for $j \leq J_{n+1}-1$
\begin{equation*}
    \alpha_{j}^{n + 1}(1 + \delta s_2) \geq  \min\bigg\{ \bigg(\min_{k\,:\,x_{k} \in \Omega_h^{n}} \alpha_{k}^{n} \bigg),\ath \bigg\} - \delta \fmin.
\end{equation*}
A use of $(1 + \delta s_2)^{-1} \geq \exp(-\delta s_2)$ yields
\begin{equation*}
  \min_{j\,:\,x_j\in\Omega^{n+1}_h}\alpha_j^{n+1} \geq \exp(-\delta s_2)\min\bigg\{ \bigg(\min_{k\,:\,x_{k} \in \Omega_h^{n}} \alpha_{k}^{n} \bigg),\ath \bigg\} - \delta\exp(-\delta s_2) \fmin.
\end{equation*}
This relation is obviously also true if the left--hand side is replaced by $\ath$, and therefore,
\begin{align}
    \min\bigg\{\bigg(\min_{j\,:\,x_j\in\Omega^{n+1}_h}\alpha_j^{n+1}\bigg),\ath\bigg\}
\geq{}& \exp(-\delta s_2)\min\bigg\{ \bigg(\min_{k\,:\,x_{k} \in \Omega_h^{n}} \alpha_{k}^{n} \bigg),\ath \bigg\} \nonumber \\ &- \delta\exp(-\delta s_2) \fmin.
    \label{eqn:pos_bound_2}
\end{align}
Define 
\begin{equation*}
y_n=\exp{(s_2 n\delta)}\min\bigg\{ \bigg(\min_{k\,:\,x_{k} \in \Omega_h^{n}} \alpha_{k}^{n} \bigg),\ath \bigg\}.
\end{equation*}
The estimate \eqref{eqn:pos_bound_2} shows that
\begin{equation*}
    y_{n + 1} \geq y_{n} - \delta \exp(s_2 n \delta) \fmin.
\end{equation*}
Write this relation for a generic $k\le n$, and sum over $k=0,\ldots,n$ to obtain 
\begin{equation}
    y_{n + 1} \geq y_{0} - \sum_{n=0}^{n} \delta \exp(s_2 n\delta) \fmin.
    \label{eqn:pos_bound_4}
\end{equation}
The fact that the sum in~\eqref{eqn:pos_bound_4} is the lower Riemann sum for the function $\exp(s_2\,\tau)$ from $\tau = 0$ to $\tau = (n + 1)\delta$ yields
\begin{equation*}
    y_{n + 1} \geq y_{0} - \left(\dfrac{\exp(s_2(n + 1)\delta) - 1}{s_2} \right) \fmin. 
\end{equation*}
Since $y_0=\ath$, a selection of time $t_{n + 1} = (n + 1)\delta$ such that 
\begin{equation}
    t_{n + 1} \le T_m := \dfrac{1}{s_2} \ln{\left(\dfrac{\fmin + s_2 \ath}{\fmin + a_{\ast}s_2} \right)}
    \label{eqn:max_t_tm}
\end{equation}
yields $y_{n + 1} \ge a_{\ast}\exp(s_2 t_{n + 1})$, and this leads to $\min\{\alpha_{j}^{n + 1}\,:\,x_{j} \in \Omega_h^{n + 1}\} \ge a_{\ast}.$ To obtain an upper bound, note that~\eqref{eqn:n1_sten_2} yields 
\begin{align}  \label{eqn:up_bd_1}
    \alpha_{j}^{n+1}  \leq \max_{0 \leq k \leq J-1}{\alpha_k^n} + \delta (1 - \ath)  + \dfrac{\delta}{\mu} ||( \mu \alpha_{h,\delta}(t_n,\cdot) \partial_x u_h^n)^{-}||_{L^\infty(0,{{\ell_m}})}
    \end{align}
\edit{for every $0 \le j \le J-1$}. Define the function 
\begin{equation}
    \fmax = 1 - \ath + \dfrac{\ell_m \sqrt{k}}{a_{\ast}\mu^{3/2}}\dfrac{|a^{\ast} - \alast|}{|1 - a^{\ast}|^{5/2}}.
    \label{eqn:fmax}
\end{equation}
Then,~\eqref{eqn:up_bd_1} and~\eqref{eqn:um_bound} imply
\begin{equation}
    \max_{0 \leq j \leq J-1}\alpha_{j}^{n+1} \leq \max_{0 \leq j \leq J-1}{\alpha_{j}^n} + \delta \fmax.
    \label{eqn:up_bd_2}
\end{equation}
Write this relation for a generic $k\le n$ and sum over $k=0,\ldots,n$ to obtain
\[
    \max_{0 \leq j \leq J-1}\alpha_{j}^{n+1} \le \max_{0 \leq j \leq J-1}{\alpha_{j}^0} + (n + 1)\delta \fmax 
    \leq m_{02} + t_{n + 1}\fmax.
\]
Selection of time $t_{n + 1}$ such that 
\begin{equation}
t_{n + 1} \leq \frac{a^{\ast} - m_{02}}{\fmax} := T_M
\label{eqn:tmax}
\end{equation}
implies $\max_{0 \leq j \leq J-1}\alpha_{j}^{n+1} \le a^{\ast}$. Finally to ensure that the extended domain $(0,\ell_m)$ contains the time--dependent domains $(0,\ell(t))$ for every $t \in [0,T_{\ast}]$ we impose a restriction on $T_{\ast}$.  Since the domain increases at most by $h$ at each time step, and there are  $T_{\ast}/\delta$ such time steps, we set $T_{\ast} < T_{\ell} := \rho\mathscr{C}_{\CFL}(\ell_m-\ell_0)\le \frac{\delta}{h}(\ell_m-\ell_0)$. Choose $T^{\ast}  = \min(T_m,T_M,T_{\ell})$ to conclude the proof. 
\end{proof}

\begin{remark}
	\label{rem:eqviv_pihnorm}
	The norm $||\cdot||_{0,\Omega_{h}^n}$ in the space $\mathcal{S}_{h}^n$ is equivalent to the norm $||\Pi_{h}{\cdot}||_{0,\Omega_{h}^n}$. In fact, we have for all $w \in S_h^n$, $(1/\sqrt{3})||\Pi_{h} w||_{0,\Omega_{h}^n} \leq || w||_{0,\Omega_{h}^n} \leq ||\Pi_{h} w||_{0,\Omega_{h}^n}.$
	This is an easy consequence of estimating $ || w||_{0,\Omega_{h}^n}$ by Simpson's quadrature rule, which is exact for second degree polynomials. 
\end{remark}
\begin{lemma}[Step \ref{is.4}]
	\label{lemma:dot_max}
	The equation \eqref{eqn:bform_ot} has a unique solution $\widetilde{c}_{h}^{n+1}$, and it holds $0\le c^{n+1}_h\le 1$.	 
\end{lemma}

\begin{proof}
	Recall that $x_{J_{n+1}} = \ell_h^{n+1}$, and for $r = n,\,n+1$, define the vector
	\begin{align*}
	\boldsymbol{c}_{h}^{r} &:= [c_h^{r}(x_0),\,c_h^{r}(x_1),\,\ldots,\, c_h^{r}(x_{J_{n+1}-1})].
	\end{align*}
	The vector $\boldsymbol{c}_{h}^{n+1}$ contains the discrete unknowns at $t_{n+1}$. Note that we do not need to compute the nodal value $\boldsymbol{c}_{h}^{n+1}(x_{J_{n+1}})$  at the discrete level  since Dirichlet boundary condition holds at $x_{J_{{n+1}}}$. The matrix equation corresponding to~\eqref{eqn:bform_ot} is 
	\begin{equation}
	\left(M + \delta \lambda D + Q\delta S\right){\bf c}_{h}^{n+1} = M{\bf c}_{h}^{n} - \delta {\bf b}_h,
	\label{eqn:ch_matrix}
	\end{equation}
	where ${\bf b}_h$ is $J_{n+1} \times 1$ vector with entries ${\bf b}_{h,i} = 0$ for $0 \leq i \leq J_{n+1}-2$ and ${\bf b}_{h,J_{n+1}-1} = -\lambda/h$. Here, $M$ is the $J_{n+1}\times J_{n+1}$ positive, diagonal, lumped mass matrix. The matrix $D$ is the stiffness matrix with all off--diagonal entries negative. The entries of the positive, diagonal, lumped mass matrix $S$ are as follows:
	\begin{equation}
	S_{ii} = \sum_{\mathcal{X}_j \subset \supp(\varphi_{i,h})} \dfrac{h\,\alpha_{j}^{n}}{2}  \left\langle \dfrac{(\Pi_h\varphi_{i,h})^2}{1 + \widehat{Q}_1 \left|\Pi_h c_h^{n}\right|}\right\rangle_{\mathcal{X}_j},\;\; 0 \leq i \leq J_{n+1}-1,
	\end{equation}
	where $\{\varphi_{i,h}\}_{\{0\le i\le J_{n+1} - 1\}}$ is the canonical nodal basis of $\mathcal{S}_{h,0}^{n+1}$. The symbol $\langle f \rangle_{\mathcal{X}_j}$ denotes the average of $f$ over the cell $\mathcal{X}_j$. An application of Lemma~\ref{appen_id.g} shows that the discrete operator $\epsilon_{h,\delta} := (\mathbb{I}_{J_{n+1}} + \delta M^{-1}(\lambda D + QS) )^{-1}$ is positive. A use of the facts $\alpha_{h,\delta}(t_{n+1},\cdot) > 0,\,{\bf c}_{h}^{n} \geq 0$, and ${\bf b}_{h} \leq 0$ yields ${\bf c}_{h}^{n
		+1} \geq 0$.
	Next, we obtain the upper bound for $\boldsymbol{c}_h^{n+1}$. For $r=n,\,n+1$, define
	\begin{align*}
	\widehat{{\bf c}}_{h}^{r} &:= [c_h^{r}(x_0) -1 ,\,c_h^{r}(x_1) -1,\,\ldots,\, c_h^{r}(x_{J_{n+1}-1}) - 1].
	\end{align*}
It is easy to observe that $	(M + \delta\lambda D + Q\delta S)\widehat{\bf c}_{h}^{n+1} = M\widehat{\bf c}_{h}^{n} - \delta \widehat{\bf b}_h,$
	where $\widehat{\bf b}_h$ is the vector of nonnegative entries
	\begin{equation}
	\widehat{\bf b}_{h,i} = \sum_{X_j \subset \supp(\varphi_{i,h})} \dfrac{Q\,\alpha_{j}^{n}\,h}{2}  \left\langle \dfrac{\Pi_h \varphi_{i,h}}{1 + \widehat{Q}|\Pi_h^n c_h^n|} \right\rangle_{\mathcal{X}_j},\;\; 0 \leq j \leq J_{n+1}-1.
	\end{equation}
	Then, the same reasoning is used to obtain the positivity and Lemma~\ref{appen_id.g} imply ${\bf c}_h^{n+1}-1\le 0$.
\end{proof}
\subsection{Compactness results}
\label{sec:compactness_est}

The next goal is to establish necessary compactness properties for the iterates, which enables us to extract a convergent subsequence of discrete solutions, whose limit is a threshold solution. We list the main steps involved in this section. We establish
\begin{enumerate}[label= (CR.\arabic*),ref=(CR.\arabic*),leftmargin=\widthof{(CR.4)}+3\labelsep]
		\setlength\itemsep{-0.1em}
    \item\label{cr.1}
    a uniform $L^2(0,T_{\ast};H^1(0,\ell_m))$ estimate for the family $\{c_{h,\delta}\}_{h,\delta}$. 
    \item\label{cr.2} a uniform spatial BV estimate for the family $\{\alpha_{h,\delta}\}_{h,\delta}$.
    \item\label{cr.3} a uniform temporal BV estimate for the family $\{\alpha_{h,\delta}\}_{h,\delta}$.
    \item\label{cr.4}  a uniform $L^2(0,T_{\ast};H^1(0,\ell_m))$ estimate for the family $\{\widehat{u}_{h,\delta}\}_{h,\delta}$. 
    \item\label{cr.5} a uniform BV estimate for the family $\{\ell_{h,\delta}\}_{h,\delta}$.
    \item\label{cr.6}  that the family $\{\Pi_{h,\delta}c_{h,\delta}\}_{h,\delta}$ is relatively compact in $L^2(\mathscr{D}_{T_{\ast}})$.
    \item\label{cr.7}   Theorem~\ref{thm:main_thm_1}  with the help of \ref{cr.1}--\ref{cr.6}
\end{enumerate}

In this sequel, $\mathscr{C}_1$ denotes a generic constant that depends $\alpha_{0},\,c_{0},\,a_\ast$, $a^\ast$, $\ell_m$, $T_{\ast}$, and the model parameters.  Let us start with a preliminary lemma, the proof of which is an easy consequence of local Taylor expansions.
\begin{lemma}{\emph{\bf \cite[Section 8.4]{droniou2018gradient}}}
\label{lemma:pihf_f}
For any $w \in H^1(0,{{\ell_m}})$, the following estimates hold: \vspace{-0.4cm}
\begin{align}
\label{eqn:ml_interp}
    ||w - \Pi_h w_h ||_{0,(0,{{\ell_m}})} &\leq \dfrac{h}{2} ||\partial_x w||_{0,(0,{{\ell_m}})} \textrm{ and }\\
    ||\Pi_h w_h||_{0,(0,{{\ell_m}})} &\leq \dfrac{h}{2} ||\partial_x w||_{0,(0,{{\ell_m}})} + ||w||_{0,(0,{{\ell_m}})}.
    \label{eqn:ml_overbd}
\end{align}
\end{lemma}
\noindent We now prove an $L^2(0,T_{\ast};H^1(0,{{\ell_m}}))$ stability estimate for $c_{h,\delta}$.
\begin{proposition}[Step~\ref{cr.1}]
\label{prop:chd_energy}
It holds $||c_{h,\delta}||_{L^2(0,T_{\ast};H^1(0,{{\ell_m}}))} \le \mathscr{C}_1$.
\end{proposition}
\begin{proof}

Define the continuous function $\widehat{c}_h^n$ on $(0,\ell_m)$ by $\widehat{c}_h^n :=\widetilde{c}_h^n - 1$ in $\Omega_h^n$, and $\widehat{c}_h^n := 0$ on $(0,\ell_m)\setminus \Omega_h^n$. An application of Cauchy--Schwarz inequality and \eqref{eqn:split_id_3} yields
\begin{align}
\label{eqn:am_gm_f}
2 (\Pi_h \widehat{c}_h^{n-1},\Pi_h \widehat{c}_h^n)_{\Omega^n_h} &\leq ||\Pi_h \widehat{c}_h^{n-1}||_{0,\Omega_h^n}^2 + ||\Pi_h \widehat{c}_h^n||_{0,\Omega_h^n}^2.
\end{align}
If $\ell_h^n \leq \ell_h^{n-1}$, then $||\Pi_h \widehat{c}_h^{n-1}||_{0,\Omega_h^n}^2 \leq ||\Pi_h \widehat{c}_h^{n-1}||_{0,\Omega_h^{n-1}}^2$ since $\Omega_{h}^n \subseteq \Omega_{h}^{n-1}$. If $\ell_h^n = \ell_h^{n-1} + h$,  then $\Pi_h \widehat{c}_{h}^{n-1} = 0$ on $\Omega_h^n \setminus \Omega_h^{n-1}$, and $||\Pi_h \widehat{c}_h^{n-1}||_{0,\Omega_h^n}^2 = ||\Pi_h \widehat{c}_h^{n-1}||_{0,\Omega_h^{n-1}}^2$. Hence by~\eqref{eqn:am_gm_f} in any case
\begin{align}
2 (\Pi_h \widehat{c}_h^{n-1},\Pi \widehat{c}_h^n)_{\Omega^n_h} &\leq ||\Pi_h \widehat{c}_h^{n-1}||_{0,\Omega_h^{n-1}}^2 + ||\Pi_h \widehat{c}_h^n||_{0,\Omega_h^n}^2.
\label{eqn:am_gm_c}
\end{align}
Choose  $v_{h}^n = \widehat{c}_h^n\in \mathcal S^n_{h,0}$ as the test function in~\eqref{eqn:bform_ot} with a Dirichlet lift of $-1$, and use~\eqref{eqn:am_gm_c} and the observation that, since $\widehat{c}_h^n \le 0$ and $\alpha_h^n\ge 0$,  $- \frac{Q\alpha_h^n \Pi_h \widehat{c}_h^n}{1 + \widehat{Q}_1 \left|\Pi_h c_h^{n-1}\right|} \le -Q\alpha_h^n \Pi_h \widehat{c}_h^n,$
to obtain
\[
\dfrac{1}{2}|| \Pi_h \widehat{c}_h^n ||_{0,\Omega_{h}^n}^2 -  \dfrac{1}{2}|| \Pi_h \widehat{c}_h^{n-1}||_{0,\Omega_{h}^{n-1}}^2 + 
\delta \lambda || \partial_{x} \widehat{c}_h^n||_{0,\Omega_{h}^n}^2 \leq -Q\delta (\alpha^n_h,\Pi_h \widehat{c}_h^n)_{\Omega^n_h}. 
\]

\noindent A use of Young's and Poincar\'{e}'s inequalities together with \eqref{eqn:ml_overbd} and a summation on the index $n$ yield
\begin{align}
    \dfrac{1}{2}||\Pi_h \widehat{c}_h^n ||_{0,\Omega_{h}^n}^2 + \dfrac{\lambda \delta}{2} \sum_{r=0}^{n}|| \partial_{x}  \widehat{c}_h^r||_{0,\Omega_{h}^r}^2 \lesssim 1
    \label{eqn:h1c_omega}
\end{align}
 Since $\partial_x \widehat{c}_h^r=\partial_x c^r_{h}$ on $\Omega_h^r$ and $\partial_x c^r_{h}=0$ outside this set,~\eqref{eqn:h1c_omega} yields a bound on $\partial_x c_{h,\delta}$ in $L^2(\mathscr{D}_{T_*})$. We obtain the desired conclusion from the fact $c_{h,\delta}(t,\ell_m)=1$ for all $t\in (0,T_*)$ and a Poincar\'e inequality.
\end{proof}

Proposition~\ref{prop:chd_energy} is crucial in obtaining a bounded variation estimate for the piecewise constant function $\alpha_{h,\delta}$. The idea is then to use  Helly's selection theorem (see Theorem~\ref{appen_id.c}) to extract an almost everywhere convergent subsequence of functions out of the family of functions $\{\alpha_{h,\delta}\}_{h,\delta}$. Spatial and temporal BV estimates for $\alpha_{h,\delta}$ are derived separately in Propositions~\ref{prop:alpha_sbv} and~\ref{prop:alpha_tbv} for this purpose.

\begin{proposition}[Step~\ref{cr.2}]
\label{prop:alpha_sbv}
For $t \in (0,T_{\ast})$ it holds
\begin{equation}
    ||\alpha_{h,\delta}(t,\cdot)||_{BV(0,{{\ell_m}})} \le \mathscr{C}_{1}.
    \label{eqn:alpha_sbv}
\end{equation}
\end{proposition}

\begin{proof}
Let $j\in\{1,\ldots,J-1\}$ and subtract~\eqref{eqn:n1_sten_2} for $\alpha_{j-1}$ from~\eqref{eqn:n1_sten_2} for $\alpha_{j}$. This yields $T_0 = T_1 +\delta T_2 - \delta T_3$, where 
\begin{subequations}
\begin{align}
    T_0 &= (\ats{n+1}{j}  - \ats{n+1}{j-1}) +  \delta((\ats{n+1}{j} - \ath)^{+} d_{j}^{n} - (\ats{n+1}{j-1} - \ath)^{+}d_{j-1}^{n}),  \\
    T_1 &=  \mathscr{L}\left(\ats{n}{j-1},\ats{n}{j},\ats{n}{j+1}\right)    - \mathscr{L}\left(\ats{n}{j-2},\ats{n}{j-1},\ats{n}{j}\right), \\
  T_2 &= (\ats{n}{j} - \ath)^{+}(1-\ats{n}{j})b_{j}^n - (\ats{n}{j-1} - \ath)^{+}(1-\ats{n}{j-1})b_{j-1}^n, \text{ and } \\
  T_3 &= \ats{n}{j}\partial_x u_{h\vert{\mathcal{X}j}}^n - \ats{n}{j-1}\partial_x u_{h\vert{\mathcal{X}_{j-1}}}^n.
\end{align}
The terms in $T_1$ can be grouped in the following way: 
\begin{align}
    T_1 
={}& (\ats{n}{j} - \ats{n}{j-1}) ( 1 - \dfrac{\delta}{h} u_{j}^{n-} - \dfrac{\delta}{h} u_{j}^{n+} ) + \dfrac{\delta}{h} u_{j+1}^{n-} (\ats{n}{j+1} - \ats{n}{j}) \nonumber \\
&+ \dfrac{\delta}{h} u_{j-1}^{n+} (\ats{n}{j-1} - \ats{n}{j-2}).
\label{eqn:t1_grouping}
\end{align}
Split the terms in $T_0$ and $T_2$ using~\eqref{eqn:split_id_1} in Appendix \ref{appen_B} to obtain
\begin{align}
   T_0 ={}& (\ats{n+1}{j} - \ats{n+1}{j-1}) + \delta ((\ats{n+1}{j} - \ath)^{+}  - (\ats{n+1}{j-1} - \ath)^{+} )\frac{ d_{j}^{n} + d_{j-1}^{n}}{2} \nonumber \\
   \label{eqn:t0_split}
   &+ \delta ((\ats{n+1}{j} - \ath)^{+}  + (\ats{n+1}{j-1} - \ath)^{+} )\frac{ d_{j}^{n} - d_{j-1}^{n}}{2}, \text{ and }\\
    T_2 ={}& ((\ats{n}{j} - \ath)^{+}(1-\ats{n}{j}) + (\ats{n}{j-1} - \ath)^{+}(1-\ats{n}{j-1}))\frac{b_{j}^n - b_{j-1}^n}{2} \\
    &+ ((\ats{n}{j} - \ath)^{+} - (\ats{n}{j-1} - \ath)^{+})(2 - \ats{n}{j} - \ats{n}{j-1})\frac{b_{j}^n + b_{j-1}^n}{4} \\
     \label{eqn:t2_split}
    &+ ((\ats{n}{j} - \ath)^{+} + (\ats{n}{j-1} - \ath)^{+})(\ats{n}{j-1} - \ats{n}{j})\frac{b_{j}^n + b_{j-1}^n}{4}. 
\end{align}
\end{subequations}
Substitute~\eqref{eqn:t1_grouping},~\eqref{eqn:t0_split}, and~\eqref{eqn:t2_split} in $T_0 = T_1 + \delta T_2 - \delta T_3$, use the facts that $0 \leq b_j^n \leq 1$, $0 \le d_j^n \le s_2,\,0 \leq \alpha_j^n \leq 1$, the CFL condition \eqref{eqn:cfl_condition} together with the bound \eqref{eqn:u_bound} on the velocity, the Lipschitz continuity of $x\mapsto (x - \ath)^{+}$, and group the terms appropriately to obtain 
\begin{multline}
   (1 - \delta s_2)|\ats{n+1}{j} - \ats{n+1}{j-1}| \leq{} |\ats{n}{j} - \ats{n}{j-1}| ( 1 - \dfrac{\delta}{h} u_{j}^{n-} - \dfrac{\delta}{h} u_{j}^{n+} ) + \dfrac{\delta}{h} u_{j+1}^{n-} |\ats{n}{j+1} - \ats{n}{j}| \\
   + \dfrac{\delta}{h} u_{j-1}^{n+} |\ats{n}{j-2} - \ats{n}{j-1}| + \delta |d_{j}^{n} -  d_{j-1}^{n}| + \delta|b_{j}^n - b_{j-1}^n| \\
   + 2\delta|\ats{n}{j} - \ats{n}{j-1}| + \delta| \ats{n}{j}\partial_x u_{h\vert\mathcal{X}_{j}}^n - \ats{n}{j-1}\partial_x u_{h\vert\mathcal{X}_{j-1}}^n|.
      \label{eqn:space_bv_final}
\end{multline}

\noindent Sum the expression~\eqref{eqn:space_bv_final} from $j=1$ to $j=J$, and utilize $u_{0}^n = 0$, $u_{J}^n = 0$, $u_{J+1}^n = 0$ and $0 \leq (\delta/h) |\alpha_1^n - \alpha_0^n|u_{0}^{n-1}$ to obtain 
\begin{multline}
 (1 - \delta s_2)\sum_{j=1}^{J} |\ats{n+1}{j} - \ats{n+1}{j-1}| \leq{} (1 + 2\delta ) \sum_{j=1}^{J}|\ats{n}{j} - \ats{n}{j-1}| + \delta \sum_{j=1}^{J} |d_{j}^{n} -  d_{j-1}^{n}| \\
 + \delta \sum_{j=1}^{J} |b_{j}^{n} -  b_{j-1}^{n}| 
  +\,\delta\sum_{j=1}^{J}| \ats{n}{j}\partial_x u^n_{h|\mathcal{X}_j} - \ats{n}{j-1}\partial_x u^n_{h|\mathcal{X}_{j-1}}|.
  \label{eqn:sbv_sum}
\end{multline}
Further note that
\begin{align*}
     ||\mu \alpha_{h,\delta}(t_n,\cdot) \partial_{x}u_{h,\delta}(t_n,\cdot)||_{BV(0,{{\ell_m}})} 
\leq{}& ||\mu \alpha_{h,\delta}(t_n,\cdot) \partial_{x}u_{h,\delta}(t_n,\cdot) -  \mathscr{H}(\alpha_{h,\delta}(t_n,\cdot))||_{BV(0,{{\ell_m}})} \\ &+ ||\mathscr{H}(\alpha_{h,\delta}(t_n,\cdot))||_{BV(0,{{\ell_m}})}.
\end{align*}
A use of~\eqref{eqn:vel_dbv} and the fact that $\mathscr{H}$ is continuous and piecewise differentiable yield
\begin{equation}
\label{eqn:bv_h_u}
||\mu \alpha_{h,\delta}(t_n,\cdot) \partial_{x}u_{h,\delta}(t_n,\cdot)||_{BV(0,{{\ell_m}})} \lesssim 1 +  ||\alpha_{h,\delta}(t_n,\cdot)||_{BV(0,{{\ell_m}})}.
\end{equation}
\noindent The CFL condition \eqref{eqn:cfl_condition} yields $1-\delta s_2\ge \rho$. Moreover, there exists a $\eta>0$ such that, for all admissible $\delta$, $(1+2\delta)/(1-s_2\delta)\le 1+\eta\delta$. Hence~\eqref{eqn:sbv_sum} and~\eqref{eqn:bv_h_u} imply
\begin{multline}
   ||\alpha_{h,\delta}(t_{n+1},\cdot)||_{BV(0,{{\ell_m}})} \leq (1 +  \eta\delta ) ||\alpha_{h,\delta}(t_{n},\cdot)||_{BV(0,{{\ell_m}})} + \delta \mathscr{C}_1 (\rho\mu)^{-1}  \\
 +  \rho^{-1} \delta (||d_{h,\delta}(t_n,\cdot)||_{BV(0,{{\ell_m}})} +  ||b_{h,\delta}(t_n,\cdot)||_{BV(0,{{\ell_m}})}).
\end{multline}
 Induction on the right hand side of the above expression yields 
\begin{align*}
   ||\alpha_{h,\delta}(t_{n+1},\cdot)||_{BV(0,{{\ell_m}})} \leq{} \exp\left(T_{\ast}\eta\right) (||\alpha_{h,\delta}(0,\cdot)||_{BV(0,{{\ell_m}})}  + \mathscr{C}_1(\rho \mu)^{-1}T_{\ast})  \\
    + \rho^{-1} \exp\left( T_{\ast}\eta\right)\int_{0}^{T_{\ast}}  \left( |b_{h,\delta}(t,\cdot)|_{BV(0,{{\ell_m}})} + |d_{h,\delta}(t,\cdot)|_{BV(0,{{\ell_m}})}\right)\,\mathrm{d}t,
\end{align*}
and since $d_{h,\delta}$ and $b_{h,\delta}$ are smooth functions of $c_{h,\delta}$ (see~\ref{ds.d} in Discrete scheme~\ref{defn:semi_disc_soln}), the estimates for $c_{h,\delta}$ from Proposition~\ref{prop:chd_energy} conclude the proof. 
\end{proof}

\begin{proposition}[Step~\ref{cr.3}]
\label{prop:alpha_tbv}
The function $\alpha_{h,\delta}$ satisfies the upper bound 
\begin{equation}
    \int_{0}^{{{\ell_m}}} ||\alpha_{h,\delta}(\cdot,x)||_{BV(0,T_{\ast})}\,\mathrm{d}x \le \mathscr{C}_1.
    \label{eqn:alpha_tbv}
\end{equation}
\end{proposition}
\begin{proof}
Rearrange the terms~\eqref{eqn:n1_sten} and appropriately group using~\eqref{eqn:split_id_1} to obtain
\begin{align}
  \ats{n+1}{j} - \ats{n}{j} ={}& \delta (\ats{n}{j} - \ath)^{+} (1 - \ats{n}{j})\brate{n}{j} - \delta (\ats{n+1}{j}- \ath)^{+} \drate{n}{j} + \dfrac{\delta}{h} \um{n}{j+1} (\ats{n}{j+1} - \ats{n}{j}) \nonumber \\
   &+  \dfrac{\delta}{h} \up{n}{j} (\ats{n}{j-1} - \ats{n}{j}) - \dfrac{\delta}{h} \ats{n}{j}(u_{j+1}^n - u_{j}^n) \nonumber \\
   ={}& \delta  ((\ats{n}{j} - \ath)^{+} + (\ats{n+1}{j} - \ath)^{+}) \dfrac{(1-\ats{n}{j})\brate{n}{j} - \drate{n}{j}}{2} \nonumber \\
   &+ \delta((\ats{n}{j} - \ath)^{+} - (\ats{n+1}{j} - \ath)^{+}) \dfrac{(1-\ats{n}{j})\brate{n}{j} + \drate{n}{j}}{2} \nonumber \\
   &+ \dfrac{\delta}{h} \um{n}{j+1} (\ats{n}{j+1} - \ats{n}{j}) 
   + \dfrac{\delta}{h} \up{n}{j} (\ats{n}{j-1} - \ats{n}{j}) - \dfrac{\delta}{h} \ats{n}{j} (u_{j+1}^n - u_{j}^n). \nonumber
\end{align}
Use the facts that $0 \leq b_{j}^n \leq 1$, $0 \leq d_{j}^n \leq s_2$, $0 \leq \alpha_{j}^n \leq 1$, $g(x) = (x - \ath)^{+}$ is a Lipschitz function with Lipschitz constant one, and group the terms appropriately to obtain, for $j = 1,\ldots,J-1$
\begin{align}
|\ats{n+1}{j} - \ats{n}{j}| \leq{}& 
   \delta  \left( 1 + s_2 + |\ats{n}{j} - \ats{n+1}{j}| \dfrac{1 + s_2}{2}\right) + \dfrac{\delta}{h}  ||u_{h,\delta}||_{L^{\infty}(\mathscr{D}_{T_{\ast}})}|\ats{n}{j+1} - \ats{n}{j}| \nonumber \\
   &+ \dfrac{\delta}{h} ||u_{h,\delta}||_{L^{\infty}(\mathscr{D}_{T_{\ast}})} |\ats{n}{j-1} - \ats{n}{j}| + \delta ||\alpha_{h,\delta}\partial_{x}u_{h,\delta}||_{L^{\infty}(\mathscr{D}_{T_{\ast}})}.
  \label{eqn:tbv_j}
\end{align}
Since $u_{0}^n = 0$, for $j=0$  the same estimate holds with $\alpha_{-1}^n := \alpha_0^n$. Multiply~\eqref{eqn:tbv_j} by $h$ and sum over $j = 0,\ldots,J-1$ and $n = 0,\ldots,N_{\ast}-1$ with $N_{\ast} = T_{\ast}/\delta$ to obtain
\begin{align*}
 \left(1 - \delta \dfrac{(1 + s_2)}{2} \right)\sum_{j=0}^{J-1} h \sum_{n=0}^{N_{\ast}-1} |\ats{n+1}{j} - \ats{n}{j}|
 \leq{}& T_{\ast}{{\ell_m}} (1 + s_2 + ||\alpha_{h,\delta}\partial_{x}u_{h,\delta}||_{L^{\infty}(\mathscr{D}_{T_{\ast}})}) \\ &+ 2 ||u_{h,\delta}||_{L^{\infty}(\mathscr{D}_{T_{\ast}})}\sum_{n=0}^{N_{\ast}-1} \delta \sum_{j=0}^{J-1}  |\ats{n}{j+1} - \ats{n}{j}|.
\end{align*}
A use of the estimates~\eqref{eqn:u_bound},~\eqref{eqn:up_bound},~\eqref{eqn:alpha_sbv}, and~\eqref{eqn:cfl_condition} concludes the proof. 
\end{proof}

The next result is a direct consequence of Lemma~\ref{lemma:dvel_exis}, Proposition~\ref{prop:dalpha_max} and~\eqref{eqn:u_bound}. 
\begin{proposition}[Step~\ref{cr.4}]
	\label{prop:uhd_compact}
	The family of functions $\{\widehat{u}_{h,\delta}\}_{h,\delta}$ is uniformly bounded in $L^2(0,T_{\ast};H^{1}(0,{{\ell_m}}))$.
\end{proposition}

Next, we need to obtain an estimate on the total variation of $\ell_{h,\delta}$. From Proposition~\ref{prop:dalpha_max} it is evident that at each time step, $\ell_{h,\delta}$ can either increase by $h$ or decrease by any value. We show that $\ell_{h,\delta}$ can be expressed as sum of a decreasing function and a function bounded variation as discussed in the next proposition. 

\begin{proposition}[Step~\ref{cr.5}]
\label{prop:lhd_bv}
The piecewise constant function $\ell_{h,\delta} : [0,T_{\ast}] \rightarrow \mathbb{R}$ is of the form $\ell_{h,\delta} = \ell_{h,\delta,BV} + \ell_{h,\delta,D}$, where $\ell_{h,\delta,BV}$ is a function with uniform bounded variation in $(0,T_{\ast})$ and $\ell_{h,\delta,D}$ is a monotonically decreasing function. Consequently,
\begin{equation}
    \sum_{n=1}^{N_{\ast}} |\ell_h^n - \ell_h^{n-1}| \le \mathscr{C}_1.
    \label{eqn:bv_ell}
\end{equation}
\end{proposition}

\begin{proof}
Define $\ell_{h,\delta,BV}(t) = (\rho\mathscr{C}_{\CFL})^{-1}t$ and $\ell_{h,\delta,D}(t) = \ell_{h,\delta}(t) - (\rho\mathscr{C}_{\CFL})^{-1}t$ where $\rho$ and $\mathscr{C}_{\CFL}$ are defined in~\eqref{eqn:cfl_condition}.  Clearly, the function $\ell_{h,\delta,BV}$ is of uniform bounded variation. For the function $\ell_{h,\delta,D}$ note that 
\begin{align}
    \ell_{h,\delta,D}(t_{n+1}) - \ell_{h,\delta,D}(t_{n}) = \ell_h^{n+1} - \ell_h^{n} -(\rho\mathscr{C}_{\CFL})^{-1}\delta. 
\end{align}
If $\ell_h^{n+1}- \ell_h^{n} = h$, then by \eqref{eqn:cfl_condition}, $\ell_h^{n+1}- \ell_h^{n} \leq (\rho\mathscr{C}_{\CFL})^{-1}\delta$ and thus $\ell_{h,\delta,D}(t_{n+1}) \leq \ell_{h,\delta,D}(t_{n})$. If $\ell_h^{n+1} \leq \ell_h^{n}$, then $\ell_{h,\delta,D}(t_{n+1}) \leq \ell_{h,\delta,D}(t_{n})$, trivially. Since $\ell_{h,\delta,D}$ is decreasing and uniformly bounded, the bounded variation estimate~\eqref{eqn:bv_ell} follows.
\end{proof}

The compactness results for the function $c_{h,\delta}$ are proved next. Note that Proposition~\ref{prop:chd_energy} already guarantees that $c_{h,\delta} \in L^2(0,T_{\ast};H^1(0,{{\ell_m}}))$, and the Hilbert space structure of this space allows us to extract a weakly convergent subsequence. However, the right hand side of~\eqref{eqn:bform_ot} involves product of two discrete functions $\alpha_{h,\delta}$ and $\Pi_{h,\delta}c_{h,\delta}$. Therefore, the weak convergence of $\Pi_{h,\delta}c_{h,\delta}$ is not sufficient to prove that the limit of $\Pi_{h,\delta} c_{h,\delta}$ is a weak solution. Similarly,~\eqref{eqn:upwind} has non linear rational terms $b_{h,\delta}$ and $d_{h,\delta}$ that involve $\Pi_{h,\delta}c_{h,\delta}$. Therefore, we require strong $L^2(\mathscr{D}_{T_{\ast}})$ convergence for $\Pi_{h,\delta} c_{h,\delta}$.  A standard method to achieve this is to use a discrete Aubin--Simon theorem~(see Theorem~\ref{appen_id.d}).

We state the definition of a compactly and continuously embedded sequence of Banach spaces next.
\begin{definition}[Compactly--continuously embedded sequence] \emph{\bf \cite[Definition C.6]{droniou2018gradient}}.
	\label{def:com_cont}
Let $B$ be a Banach space. The families of Banach spaces $\{X_h,||\cdot||_{X_h}\}_h$ and $\{Y_h,||\cdot||_{Y_h}\}_{h}$ are such that $Y_h \subset X_h \subset B$. We say that the family $\{(X_h,Y_h)\}_h$ is compactly embedded in $B$ if  the following conditions hold. 
\begin{itemize}
		\setlength\itemsep{-0.3em}
    \item Any sequence $\{u_{h}\}_h$ such that $u_h \in X_h$ and $\{||u_{h}||_{X_h}\}_h$ uniformly bounded is relatively compact in $B$. 
    \item Any sequence $\{u_{h}\}_h$ such that $u_h \in X_h$, $\{||u_{h}||_{X_h}\}_h$ uniformly bounded,  $\{u_h\}_h$ converges in $B$, and $||u_h||_{Y_h} \rightarrow 0$, converges to zero in $B$.
\end{itemize}
\end{definition}
\noindent Define  $X_h := \Pi_h (H^1(0,{{\ell_m}}))$ with norm
\begin{subequations}
\begin{align}
\label{norm:xh}
||u||_{X_h} := \inf\,\left\{||w||_{1,(0,{{\ell_m}})} : w \in H^1(0,{{\ell_m}}), u = \Pi_h w \right\}.
\end{align}
Set $Y_h := X_h$ with the discrete dual norm $||\cdot||_{Y_h}$  defined by: $\forall\,u \in Y_h$,
\begin{align}
\label{norm:yh}
||u||_{Y_h} := \sup\,\left\{\int_{0}^{{\ell_m}} u\,\Pi_h v\,\dx\,:\,v \in H^1(0,{{\ell_m}}),\,||v||_{1,(0,{{\ell_m}})} \leq 1\right\}.
\end{align}
\end{subequations}
\begin{lemma}
\label{lemma:com_cont_embedding}
The family of Banach spaces $\{(X_h,Y_h)\}$ with $X_h = \Pi_h (H^1(0,{{\ell_m}}))=Y_h$ and $||\cdot||_{X_h}$ and $||\cdot||_{Y_h}$ as defined in~\eqref{norm:xh} and~\eqref{norm:yh}, respectively, is compactly--continuously embedded in $B = L^2(0,{{\ell_m}})$.
\end{lemma}

\begin{proof}
We verify the conditions in Definition~\ref{def:com_cont}. Let $\{u_h\}_h \subset B$ be a sequence of functions such that $u_h \in X_h$ and $\{||u_{h}||_{X_h}\}_h$ is bounded. Consider the corresponding sequence $\{w_h\} \subset H^1(0,{{\ell_m}})$ such that $u_h = \Pi_h w_h$ and $||u_{h}||_{X_h} = ||w_h||_{1,(0,{{\ell_m}})}$. The boundedness of $\{||u_{h}||_{X_h}\}_h$ shows that $\{||w_h||_{1,(0,{{\ell_m}})}\}$ is also bounded. Since $H^1(0,{{\ell_m}})$ is compactly embedded in $L^2(0,{{\ell_m}})$, there exists a subsequence $\{w_h\}_h$ up to re--indexing  such that $w_h \halfarrow w$ weakly in $H^1(0,{{\ell_m}})$ and $w_h \rightarrow w$ strongly in $L^2(0,{{\ell_m}})$. We claim that $u_h \rightarrow w$ strongly in $L^2(0,{{\ell_m}})$. To prove this, use the triangle inequality and then apply~\eqref{eqn:ml_interp} and~\eqref{eqn:ml_overbd} to obtain
\begin{align}
    ||u_h - w||_{0,(0,{{\ell_m}})} \leq{}& ||u_h - \Pi_h w||_{0,(0,{{\ell_m}})} + ||\Pi_h w - w||_{0,(0,{{\ell_m}})} \\
    \leq{}& ||\Pi_h(w_h - w)||_{0,(0,{{\ell_m}})} + ||\Pi_h w - w||_{0,(0,{{\ell_m}})} \\
     \label{eqn:compact_xh}
    \leq{}& ||w_h - w||_{0,(0,{{\ell_m}})} + h||\partial_x (w_h - w)||_{0,(0,{{\ell_m}})}. 
\end{align}
Since $w_h\to w$ in $L^2(0,\ell_m)$ while being bounded in $H^1(0,\ell_m)$, \eqref{eqn:compact_xh} shows that $||u_h - w||_{0,(0,{{\ell_m}})} \rightarrow 0$ as $h \rightarrow 0$. This proves the first condition in Definition~\ref{def:com_cont}.

Let $\{u_h\} \subset B$ be such that $u_h \in X_h$, $\{||u_h||_{X_h}\}_h$ is bounded, $||u_h||_{Y_h} \rightarrow 0$ as $h \rightarrow 0$, and $u_h$ converges in $B$. Let $w_h \in X_h$ be such that $\Pi_h w_h = u_h$ and $||w_h||_{1,(0,{{\ell_m}})} = ||u_h||_{X_h}$. Then, note that 
\begin{align}
    ||u_h||_{0,(0,{{\ell_m}})}^2 &= \int_{0}^{{\ell_m}} \,u_h\,\Pi_h w_h\,\dx  \leq ||u_{h}||_{Y_h}||w_h||_{1,(0,\ell_m)}\le  ||u_{h}||_{Y_h} ||u_{h}||_{X_h}.
\end{align}
The assumed properties on $\{u_h\}_h$ then show that $u_h \to 0$ in $L^2(0,\ell_m)$, which concludes the proof.
\end{proof}

To obtain the relative compactness of  $\{\Pi_{h,\delta}c_{h,\delta}\}_{h,\delta}$ in $L^2(\mathscr{D}_{T_{\ast}})$, we start with an auxiliary function $\varphi_{h,\epsilon}^n : [0,\ell_m] \rightarrow [0,1]$  defined by for a fixed $\epsilon>0$  (see Figure~\ref{fig:aux_fuction}) 
	\begin{equation}
	\varphi_{h,\epsilon}^n(x) = \left\{ 
	\begin{array}{c c}
	1 & 0 \leq x \leq \ell_h^n - \epsilon,  \\
	(\ell_h^n - x)/\epsilon & \ell_h^n - \epsilon < x \leq \ell_h^n, \\
	0 & \ell_h^n< x \leq {{\ell_m}}.
	\end{array} \right.
	\label{eqn:aux_ftn}
	\end{equation}
	\begin{figure}[h!]
		\centering
		\includegraphics[scale=0.8]{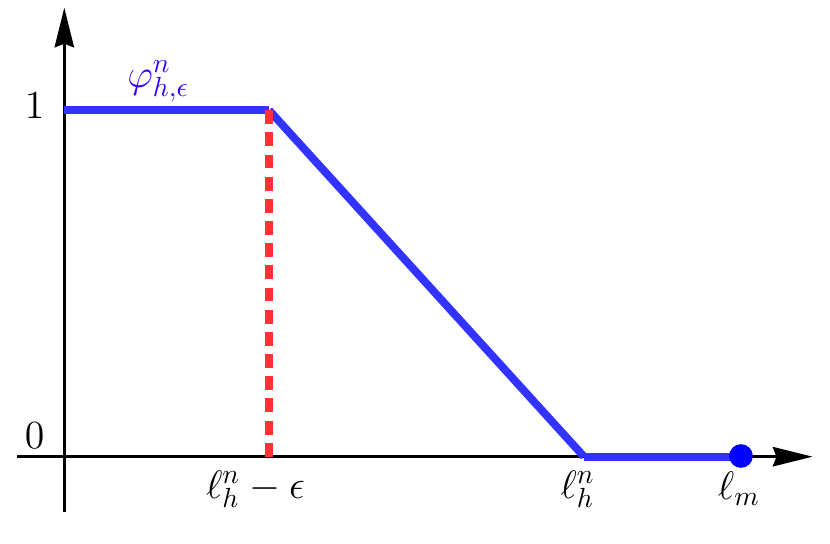}
		\caption{The auxiliary function $\varphi_{h,\epsilon}^n$.}
		\label{fig:aux_fuction}
	\end{figure}

\noindent For $\widehat{c}_{h,\delta} = c_{h,\delta} - 1$ the mass lumped function can be split into
\begin{equation}
\Pi_{h,\delta} \widehat{c}_{h,\delta} = \Pi_{h,\delta} (\widehat{c}_{h,\delta} \varphi_{h,\epsilon}) + \Pi_{h,\delta} (\widehat{c}_{h,\delta} (1 - \varphi_{h,\epsilon})),
\end{equation} where $\varphi_{h,\epsilon} = \varphi_{h,\epsilon}^n$ on $\mathcal{T}_n = (t_{n},t_{n+1})$ for $0 \leq n \leq N_{\ast} - 1$.  Consider the second term $\Pi_{h,\delta} (\widehat{c}_{h,\delta} (1 - \varphi_{h,\epsilon}))$, which is equal to $\Pi_h\,(\widehat{c}_{h}^n(1 - \varphi_{h,\epsilon}^n))$ on $\mathcal{T}_n$. A use of the facts $1 - \varphi_{h,\epsilon}^n  = 0 \text{ on }[0,\ell_h^n - \epsilon)$, $\Pi_h \widehat{c}_{h}^n = 0$ (see Figure~\ref{fig:aux_fuction}) on $(\ell_h^n,{{\ell_m}}]$ and the property $\Pi_h(fg) = (\Pi_hf)\,(\Pi_h g)$ yield
\begin{align}
    ||\Pi_h\,(\widehat{c}_{h}^n(1 - \varphi_{h,\epsilon}^n))||_{0,(0,{{\ell_m}})}^2
    &= \int_{\ell_h^n - \epsilon}^{\ell_h^n} |\Pi_h\,(\widehat{c}_{h}^n(1 - \varphi_{h,\epsilon}^n))|^2\dx \nonumber \\
    \label{eqn:t2_linfty}
    &\leq{}\epsilon\,||\Pi_h\,(\widehat{c}_{h}^n(1 - \varphi_{h,\epsilon}^n))||_{L^\infty(0,{{\ell_m}})}^2. 
\end{align}
Multiply~\eqref{eqn:t2_linfty} by $\delta$, sum over $n=0,\ldots,N_{\ast}-1$, and use the bounds $||\Pi_h(1 - \varphi_{h,\epsilon}^n)||_{L^\infty(0,{{\ell_m}})} \leq 1$ and $||\Pi_h \widehat{c}_{h}^n||_{L^\infty(0,{{\ell_m}})} \leq 1$ to obtain
\begin{align}
    ||\Pi_{h,\delta} (\widehat{c}_{h,\delta} (1-\varphi_{h,\epsilon}))||_{L^2(\mathscr{D}_{T_{\ast}})} \leq \sqrt{T_{\ast}\epsilon}.
    \label{eqn:chd_l2ball}
\end{align} 
 \noindent Proposition~\ref{prop:phi_c_compact} establishes that the family of functions $\{\Pi_{h,\delta} (\varphi_{h,\epsilon}\widehat{c}_{h,\delta})\}_{h,\delta}$ is relatively compact in $L^2(\mathscr{D}_{T_{\ast}})$. Then, Proposition~\ref{prop:phi_c_compact} and~\eqref{eqn:chd_l2ball} are used to prove Proposition~\ref{prop:step_ce5}.

 \begin{definition}[Discrete time derivative]
 	The discrete time derivative of a function $f$ on $\mathscr{D}_{T_{\ast}}$ is defined as follows: on $\mathcal{T}_n$,
 	\begin{align}
 	D_{h,\delta}^n f := \dfrac{\Pi_h f(t_{n+1},\cdot) - \Pi_h f(t_{n},\cdot)}{\delta}.
 	\label{eqn:disc_time_der}
 	\end{align}
 \end{definition} 
\begin{definition}[Piecewise linear interpolant operator]
The piecewise linear interpolant operator $\mathcal{I}_h : H^1(0,{{\ell_m}}) \rightarrow \mathcal{S}_h$ is defined by
\begin{align}
    \mathcal{I}_hf(x) = f(x_j)\dfrac{x_{j+1} - x}{h}+ f(x_{j+1})\dfrac{x - x_{j}}{h}\quad\forall\,x \in \mathcal{X}_j,\,j = 0,\ldots,J-1.
    \label{eqn:lin_interp}
\end{align} 
\end{definition}
\noindent We are now in a position to prove the relative compactness of $\{\Pi_{h,\delta} (\varphi_{h,\epsilon}\widehat{c}_{h,\delta})\}_{h,\delta}$  in $L^2(\mathscr{D}_{T_{\ast}})$, which is required to prove Step~\ref{cr.5}.
\begin{proposition}
\label{prop:phi_c_compact}
The family of functions $\{\Pi_{h,\delta} (\varphi_{h,\epsilon}\widehat{c}_{h,\delta})\}_{h,\delta}$ is relatively compact in $L^2(\mathscr{D}_{T_{\ast}})$.
\end{proposition}

\begin{proof}
The desired result follows from the discrete Aubin--Simon theorem (see Theorem~\ref{appen_id.d}), for which we need to verify the conditions \eqref{as.1}--\eqref{as.3} with $B = L^2(0,\ell_m)$ and $Y_h = X_h = \Pi_h(H^1(0,{{\ell_m}}))$. The family
\begin{subequations}
\begin{align}
\label{as.1}
&\;\; \text{$\{\Pi_{h,\delta} (\varphi_{h,\epsilon} \widehat{c}_{h,\delta})\}_{h,\delta}$ is bounded in $L^2(0,T_{\ast};B)$.} \\
\label{as.2}
&\;\; \text{$\{||\Pi_{h,\delta} (\varphi_{h,\epsilon} \widehat{c}_{h,\delta})||_{L^2(0,T_{\ast};X_h)}\}_{h,\delta}$ is bounded.} 
\\
\label{as.3}
&\;\; \text{$\{||D_{h,\delta} (\varphi_{h,\epsilon} \widehat{c}_{h,\delta})||_{L^1(0,T_{\ast};Y_h)}\}_{h,\delta}$ is bounded.}
\end{align}
\end{subequations}
Proposition~\ref{prop:chd_energy} and the bound $|\varphi_{h,\epsilon}|\le 1$ yields~\eqref{as.1}. We have $|\varphi_{h,\epsilon}|\le 1$ and $|\partial_x \varphi_{h,\epsilon}|\le 1/\epsilon$, so for all $t\in (0,T_\ast)$,
\[
|\varphi_{h,\epsilon}(t,\cdot)\widehat{c}_{h,\delta}(t,\cdot)|_{1,(0,\ell_m)}\le |\widehat{c}_{h,\delta}(t,\cdot)|_{1,(0,\ell_m)}+\epsilon^{-1}|\widehat{c}_{h,\delta}(t,\cdot)|_{0,(0,\ell_m)}.
\]
 The facts $||\varphi_{h,\epsilon}^n||_{L^\infty(0,{{\ell_m}})} \le 1$, $||\widehat{c}_{h}^n||_{L^\infty(0,{{\ell_m}})} \le 1$, $|\partial_x \varphi_{h,\epsilon}^n| \le 1/\epsilon$ and $\partial_x \varphi_{h,\epsilon}^n = 0$ on $[0,\ell_h^n - \epsilon - h)$, and $(\ell_h^n + h,{{\ell_m}})$ yield
\begin{align}
    |\varphi_{h,\epsilon}^n \widehat{c}_{h}^n|_{1,(0,{{\ell_m}})}^2 &\le 2 \int_{0}^{{\ell_m}} |\partial_x \widehat{c}_{h}^n|^2\,\mathrm{d}x + 2 \int_{\ell_h^n - \epsilon - h}^{\ell_h^n+h} \dfrac{1}{\epsilon^2} |\widehat{c}_{h}^n|^2\,\mathrm{d}x \\
    &\le 2 |\widehat{c}_{h}^n|_{1,(0,{{\ell_m}})}^2 + \dfrac{2(\epsilon + 2h)}{\epsilon^2},
\end{align}
and hence a use of \eqref{norm:xh},~Remark~\ref{rem:eqviv_pihnorm}, and  Proposition~\ref{prop:chd_energy} leads to 
\begin{align}
    ||\Pi_{h,\delta}(\varphi_{h,\epsilon} \widehat{c}_{h,\delta})||_{L^2(0,T_{\ast};X_h)} \leq ||\varphi_{h,\epsilon} \widehat{c}_{h,\delta}||_{L^2(0,T_{\ast};H^1(0,{{\ell_m}}))} \le \mathscr{C}_1 + \dfrac{2T_{\ast}(\epsilon + 2h)}{\epsilon},
\end{align}
which verifies \eqref{as.2}. To verify~\eqref{as.3}, we start with the estimation of  $||D_{h,\delta}^{n-1}(\varphi_{h,\epsilon}\widehat{c}_{h,\delta})||_{Y_h}$. Let $v_h \in H^1(0,{{\ell_m}})$ with $||v_h||_{1,(0,{{\ell_m}})} \leq 1$. Note that~\eqref{eqn:disc_time_der} along with the identity~\eqref{eqn:split_id_2} yields
\begin{align}
    D_{h,\delta}^{n-1}(\varphi_{h,\epsilon}\widehat{c}_{h,\delta}) =  (D_{h,\delta}^{n-1}\widehat{c}_{h,\delta})\Pi_h \varphi_{h,\epsilon}^n + (D_{h,\delta}^{n-1}\varphi_{h,\epsilon}) \Pi_h \widehat{c}_{h}^{n-1},
\end{align}
and hence
\begin{align*}
    \int_{0}^{{{\ell_m}}} D_{h,\delta}^{n-1} (\varphi_{h,\epsilon} \widehat{c}_{h,\delta}) \Pi_h v_{h}\dx ={}& \int_{0}^{{{\ell_m}}}  (D_{h,\delta}^{n-1}\widehat{c}_{h,\delta})\Pi_h \varphi_{h,\epsilon}^n \Pi_h v_{h}\dx \\
    &+ \int_{0}^{{{\ell_m}}}(D_{h,\delta}^{n-1}\varphi_{h,\epsilon}) \Pi_h \widehat{c}_{h}^{n-1} \Pi_h v_{h}\dx =: T_1 + T_2.
\end{align*}
To estimate $T_1$, observe that  $\varphi_{h,\epsilon}^n$ is zero on $[\ell_h^n,{{\ell_m}}]$. Use the result $(\Pi_h f)\,(\Pi_h g) = \Pi_h(fg)$ to obtain
\begin{align}
    T_1 = \int_{0}^{\ell_h^n} (D_{h,\delta}\widehat{c}_{h}^{n-1}) \Pi_h(\varphi_{h,\epsilon}^{n}\,v_{h})\dx. 
\end{align}
Now observe that $\Pi_h(\varphi_{h,\epsilon}^{n}\,v_{h}) = \Pi_h(\mathcal{I}_h(\varphi_{h,\epsilon}^{n}\,v_{h}))$, where $\mathcal{I}_h$ is defined by~\eqref{eqn:lin_interp}. Therefore, \eqref{eqn:bform_ot} with a Dirichlet lift of $-1$ tested against $\mathcal{I}_h(\varphi_{h,\epsilon}^{n}\,v_{h}) \in S_{h,0}^n$ yields
\begin{multline}
  T_1
  = -\lambda \int_{0}^{\ell_h^n} \partial_{x} \widehat{c}_{h}^{n-1} \partial_{x} (\mathcal{I}_h(v_{h}\varphi_{h,\epsilon}^n))\,\mathrm{d}x \,-\, Q\int_{0}^{\ell_h^n} \dfrac{\alpha_{h,\delta}(t_n,\cdot)\Pi_h \widehat{c}_{h}^n}{1 + \widehat{Q}_1|\Pi_h c_h^{n-1}|} \Pi_h (v_{h} \varphi_{h,\epsilon}^n)\,\dx \\
    \label{eqn:dual_norm_3}
  - Q\int_{0}^{\ell_h^n} \dfrac{\alpha_{h,\delta}(t_n,\cdot)}{1 + \widehat{Q}_1|\Pi_h c_h^{n-1}|} \Pi_h (v_{h} \varphi_{h,\epsilon}^n)\,\dx.
\end{multline}
We have $|| \mathcal{I}_{h} w||_{1,(0,\ell_h^n)} \le ||w||_{1,(0,\ell_h^n)}$ and $|| \varphi_{h,\epsilon}^n v_h ||_{1,(0,\ell_h^n)} \le \mathscr{C}_{2}(\epsilon)$, where $\mathscr{C}_{2}(\epsilon)$ is a generic constant that depends on $\epsilon$. Also, it holds $(1 + \widehat{Q}_1|\Pi_h c_h^{n-1}|)^{-1} \le 1$. Hence,
\begin{align}
\label{eqn:time_dual_t3}
   T_1 \le \mathscr{C}_2(\epsilon)||\partial_{x} \widehat{c}_{h}^{n-1}||_{0,(0,\ell_h^n)} +  \frac32 Q || \Pi_h \widehat{c}_{h}^n||_{0,(0,\ell_h^n)} + (3/2)Q\sqrt{{{\ell_m}}}. 
\end{align}
The constant $(3/2)$ in~\eqref{eqn:time_dual_t3} results from the application of the Cauchy--Schwarz inequality to integral $(\Pi_h \widehat{c}_{h}^n,\Pi_h (v_{h} \varphi_{h,\epsilon}^n))_{(0,\ell_m)}$, the facts $\Pi_h (v_{h} \varphi_{h,\epsilon}^n) = (\Pi_h v_{h} )(\Pi_h\varphi_{h,\epsilon}^n)$, $|\Pi_h\varphi_{h,\epsilon}^n| \le 1$, and~\eqref{eqn:ml_overbd}.  Next, we estimate the term $T_2$. The function $\varphi_{h,\epsilon}$ has the property $\varphi_{h,\epsilon}^{n-1}(x) = \varphi_{h,\epsilon}^n(x - \ell_h^{n-1} + \ell_h^n)$ by definition. This with the fact that $\varphi_{h,\epsilon}^n$ is $1/\epsilon$--Lipschitz, implies $|D_{h,\delta}^{n-1}\varphi_{h,\epsilon}| \leq |\ell_h^n - \ell_h^{n-1}|/(\delta\epsilon)$. Consequently, 
\begin{equation}
 \label{eqn:time_dual_t4}
    |T_2| \le \dfrac{\ell_m }{\delta \epsilon} |\ell_h^n - \ell_h^{n-1}|.
\end{equation}
Now let us conclude the argument. The estimates~\eqref{eqn:time_dual_t3} and~\eqref{eqn:time_dual_t4} yield
\begin{multline}
    \int_{0}^{{{\ell_m}}} D_{h,\delta} (\varphi_{h,\epsilon}^n \widehat{c}_{h}^n)(t_{n-1},\cdot) \Pi_h v_{h}\dx
 \le
   \mathscr{C}_2(\epsilon)||\partial_{x} \widehat{c}_{h}^{n-1}||_{0,(0,\ell_h^n)} \\+ (3/2)Q|| \Pi_h \widehat{c}_{h}^n||_{0,(0,\ell_h^n)} + (3/2)Q\sqrt{{{\ell_m}}} 
    + \dfrac{\ell_m }{\delta \epsilon}|\ell_h^{n} - \ell_h^{n-1}|.
    \label{eqn:time_dual_penult}
\end{multline}
Therefore, taking the supremum over the considered $v_h$, multiplying~\eqref{eqn:time_dual_penult} by $\delta$ and summing over $n = 1,\ldots,N_{\ast}$ yield
\begin{multline}
    \int_{0}^{T_{\ast}} ||D_{h,\delta} (\varphi_{h,\epsilon} \widehat{c}_{h,\delta})||_{Y_h}\,\mathrm{d}t \le{} \mathscr{C}_2(\epsilon)\left[1 +  \sum_{n=1}^{N_{\ast}} |\ell_h^{n} - \ell_h^{n-1}| \right.\\
    + \left. \sum_{n=1}^{N_{\ast}} \delta(|| \Pi_h \widehat{c}_{h}^n||_{0,(0,\ell_h^n)} + ||\partial_{x} \widehat{c}_{h}^n||_{0,(0,\ell_h^n)}) \right].
\end{multline}
Then,  \eqref{as.3} follows from an application of discrete Cauchy--Schwarz inequality,~\eqref{eqn:bv_ell}, and Proposition~\ref{prop:chd_energy}.
\end{proof}

\begin{proposition}[Step~\ref{cr.6}]
	\label{prop:step_ce5}
The family of functions $\{\Pi_{h,\delta} c_{h,\delta}\}_{h,\delta}$ is relatively compact in ${L}^2(\mathscr{D}_{T_{\ast}})$.
\end{proposition}
\begin{proof}
Since~\eqref{eqn:chd_l2ball} holds true, for any $\epsilon > 0$,
\begin{equation}
    \label{eqn:tot_bound}
    \{\Pi_{h,\delta} \widehat{c}_{h,\delta}\}_{h,\delta} \subset \{\Pi_{h,\delta} (\varphi_{h,\epsilon} \widehat{c}_{h,\delta})\}_{h,\delta} + B_{{L}^2(\mathscr{D}_{T_{\ast}})}\left(0;\sqrt{T_{\ast}\epsilon}\right),
\end{equation}
where $ B_{{L}^2(\mathscr{D}_{T_{\ast}})}\left(0;\sqrt{T_{\ast}\epsilon}\right)$ is the ball in ${L}^2(\mathscr{D}_{T_{\ast}})$ centered at the zero function with radius $\sqrt{T_{\ast}\epsilon}$. The relative compactness of the set $\{\Pi_{h,\delta} (\varphi_{h,\epsilon} \widehat{c}_{h,\delta})\}_{h,\delta}$ from Proposition~\ref{prop:phi_c_compact} and~\eqref{eqn:tot_bound} show that $\{\Pi_{h,\delta} \widehat{c}_{h,\delta}\}_{h,\delta}$ can be covered by finite number of ${L}^2(\mathscr{D}_{T_{\ast}})$ balls with radius $\eta$ for any $\eta > 0$, hence is totally bounded in ${L}^2(\mathscr{D}_{T_{\ast}})$, and thus relatively compact. Then, the relation $c_{h,\delta} = \widehat{c}_{h,\delta} + 1$ yields the desired result.
\end{proof}

We use Helly's selection theorem for $\{\alpha_{h,\delta}\}$ and $\{\ell_{h,\delta}\}$, weak compactness of $\{\widehat{u}_{h,\delta}\}$ in $L^2(0,T_\ast;H^1(0,\ell_m))$, and relative compactness of $\{\Pi_{h,\delta} c_{h,\delta}\}$ in $L^2(\mathscr{D}_{T_{\ast}})$ to prove Theorem~\ref{thm:main_thm_1}. \medskip \\
\noindent\textbf{Proof of Theorem \ref{thm:main_thm_1}} (Step~\ref{cr.7}. convergence of the iterates). \medskip \\
Proposition~\ref{prop:dalpha_max} establishes the existence of a time $T_{\ast}$ such that $\alpha_{h,\delta} \in L^\infty(\mathscr{D}_{T_{\ast}})$. Propositions~\ref{prop:alpha_sbv} and~\ref{prop:alpha_tbv} show that $\alpha_{h,\delta} \in BV(\mathscr{D}_{T_{\ast}})$. Therefore, Helly's selection theorem  guarantees the existence of a subsequence $\{\alpha_{h,\delta}\}$ up to re--indexing and a function $\alpha \in BV(\mathscr{D}_{T_{\ast}})\cap L^\infty(\mathscr{D}_{T_{\ast}})$ such that $\alpha_{h,\delta} \rightarrow \alpha$ in $L^1(\mathscr{D}_{T_{\ast}})$ and almost everywhere in $\mathscr{D}_{T_{\ast}}$. 

Proposition~\ref{prop:lhd_bv} shows that the family $\{\ell_{h,\delta}\}_{h,\delta}$ is bounded in $BV(0,T_{\ast})$. Therefore, Helly's selection theorem guarantees the existence of a function $\ell \in BV(0,T_{\ast})\cap L^\infty(0,T_{\ast})$ such that $\ell_{h,\delta} \rightarrow \ell$ strongly in $L^1(0,T_{\ast})$ and almost everywhere in $(0,T_{\ast})$.

An application of Proposition~\ref{prop:uhd_compact} shows that there exist a subsequence $\{\widehat{u}_{h,\delta}\}_{h,\delta}$ and a function $\widehat{u} \in L^2(0,T_{\ast};H^{1}(0,{{\ell_m}}))$ such that $\widehat{u}_{h,\delta} \halfarrow \widehat{u}$ weakly and $\partial_{x} \widehat{u}_{h,\delta} \halfarrow \partial_x \widehat{u}$ weakly in $L^2(\mathscr{D}_{T_\ast})$

Proposition~\ref{prop:chd_energy} yields a subsequence $\{c_{h,\delta}\}_{h,\delta}$, up to re--indexing, and a function $c \in L^2(0,T_{\ast};H^1(0,{{\ell_m}}))$ such that $c_{h,\delta} \halfarrow c$ and $\partial_x c_{h,\delta} \halfarrow \partial_x c$ weakly in $L^2(\mathscr{D}_{T_{\ast}})$. Proposition~\ref{prop:step_ce5} establishes the strong convergence of $\Pi_{h,\delta} c_{h,\delta}$ in $L^2(\mathscr{D}_{T_{\ast}})$ and, by \eqref{eqn:ml_interp}, $c_{h,\delta}-\Pi_{h,\delta} c_{h,\delta}\to 0$ in this space; hence, the strong limit of $\Pi_{h,\delta} c_{h,\delta}$ is $c$.
\qed

\section{Proof of Theorem~\ref{thm:main_thm_2}}
\label{sec:convergence}

The proof of Theorem~\ref{thm:main_thm_2} involves four main steps which are listed below.
\begin{enumerate}[label= (CA.\arabic*),ref=(CA.\arabic*),leftmargin=\widthof{(CA.4)}+3\labelsep]
	\setlength\itemsep{-0.1em}
    \item\label{ca.1} The domains $A_{h,\delta} := \{(t,x) : x < \ell_{h,\delta}(t), t \in (0,T_{\ast})\}$ converge to $D^\thr_{T_\ast} := \{ (t,x) : x < \ell(t), t \in (0,T_{\ast}) \}$ as defined in Theorem~\ref{thm:main_thm_2}.
    \item\label{ca.2} The limit function $\alpha$ satisfies~\eqref{eqn:weak_vf} with $T = T_{\ast}$.
    \item\label{ca.3} The restricted limit function $\widehat{u}_{\vert{D^\thr_{T_\ast}}}$ satisfies~\eqref{eqn:weak_vel} with $T = T_{\ast}$.
    \item\label{ca.4} The limit function $c_{\vert{D^\thr_{T_\ast}}}$ satisfies~\eqref{eqn:weak_ot} with $T = T_{\ast}$.
\end{enumerate}

\begin{proposition}[Step~\ref{ca.1}]
\label{prop:cv_lh_ind}
The characteristic functions $\pmb{\chi}_{A_{h,\delta}}$ of $A_{h,\delta}$ converge (up to a subsequence) almost everywhere to the characteristic function $\pmb{\chi}_{D^\thr_{T_\ast}}$ of $D^\thr_{T_\ast}$.
\end{proposition}
\begin{proof}
Theorem~\ref{thm:main_thm_1} yields a subsequence $\{ \ell_{h,\delta} \}$ (up to re-indexing) such that $\ell_{h,\delta} \rightarrow \ell$ almost everywhere, where $\ell \in BV(0,T_{\ast})$. Define the set $E = \{t \in (0,T_{\ast})\,:\, \ell_{h,\delta}(t) \not\rightarrow \ell(t) \}$. Let $\mu_{d}$ denotes the $d$--dimensional Lebesgue measure. The almost everywhere convergence of $\ell_{h,\delta}(t)$ to $\ell(t)$ implies that $\mu_{1}(E) = 0$. Tonelli's theorem applied to $\pmb{\chi}_{E \times (0,\ell_m)}$ yields $\mu_{2}(E \times (0,{{\ell_m}})) = 0$. Define the graph of $\ell$ as $F_{\ell} = \{ (t,x) \in \mathscr{D}_{T_{\ast}}\,:\,x = \ell(t),\, t\in (0,T_{\ast}) \}$ (see Figure~\ref{fig:lhd_l}). Again an application of the Tonelli's theorem shows  $\mu_{\mathbb{R}^2}(F_{\ell}) = 0$. Let $(t,x) \not\in (E \times (0,{{\ell_m}})) \cup F_{\ell}$. Then, either $\ell(t) > x$ or $\ell(t) < x$.
When $\ell(t) < x$, $\pmb{\chi}_A(t,x) = 0$. Since $(t,x) \not\in E \times (0,{{\ell_m}})$, $\ell_{h,\delta}(t) \rightarrow \ell(t)$. Therefore, for $h$ and $\delta$ small enough $\ell_{h,\delta}(t) < x$. That is, $\pmb{\chi}_{A_{h,\delta}}(t,x) = 0$, and hence $\pmb{\chi}_{A_{h,\delta}}(t,x) \rightarrow \pmb{\chi}_{A}(t,x)$. A similar argument yields the convergence for the case $\ell(t) > x$. Hence we have the almost everywhere convergence $\pmb{\chi}_{A_{h,\delta}} \rightarrow \pmb{\chi}_{A}$. 
\end{proof}
\begin{figure}[htp]
	\includegraphics[scale=1]{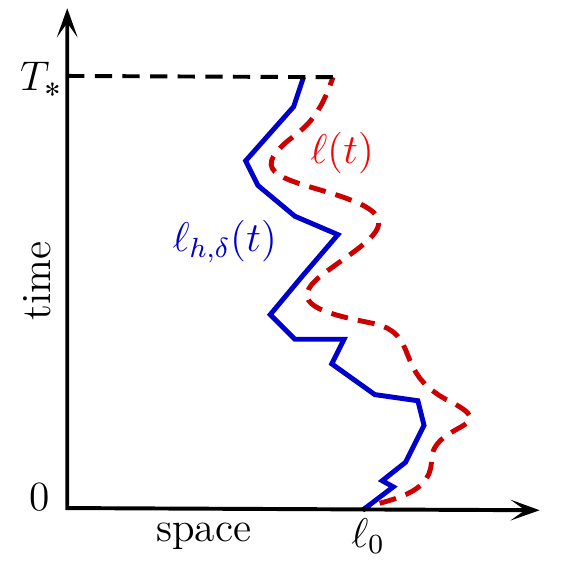}
	\centering
	\caption{Continuous tumour radius $\ell$ and discrete tumour radius $\ell_{h,\delta}$.}
	\label{fig:lhd_l}
\end{figure}

\begin{proposition}[Step~\ref{ca.2}]
\label{prop:alpha_conv}
Let $\alpha : \mathscr{D}_{T_{\ast}} \rightarrow \mathbb{R}$ be a limit provided by Theorem~\ref{thm:main_thm_1} such that $\alpha_{h,\delta} \rightarrow \alpha$ almost everywhere in $\mathscr{D}_{T_{\ast}}$. Then, $\alpha$ satisfies~\eqref{eqn:weak_vf} with $T = T_{\ast}$ for every $\varphi \in \mathscr{C}_c^\infty([0,T_{\ast})\times(0,{{\ell_m}}))$.
\end{proposition}
\begin{proof}
Let $\varphi\in \mathscr{C}_c^{\infty} ([0,T_{\ast})\times (0,{{\ell_m}}))$. Multiply~\eqref{eqn:upwind} between $t_{n+1}$ and $t_n$ by $\varphi_{j}^{n} := \langle \varphi(n\delta,\cdot) \rangle_{\mathcal{X}_j}$ and sum over the indices to obtain $T_1 + T_2 = T_3$, where 
 \begin{align}
 T_1 &:= h\sum_{n=0}^{N_{\ast}-1} \sum_{j=0}^{J-1} (\ats{n+1}{j} - \ats{n}{j}) \varphi_{j}^{n}, \\
 T_2 &:= \delta \sum_{n=0}^{N_{\ast}-1} \sum_{j=0}^{J-1}  \left( \up{n}{j+1}\ats{n}{j} - \um{n}{j+1}\ats{n}{j+1} - \up{n}{j}\ats{n}{j-1} + \um{n}{j}\ats{n}{j} \right)\varphi_{j}^{n}, \text{ and } \\
 T_3 &:= h\delta \sum_{n=0}^{N_{\ast}-1} \sum_{j=0}^{J-1} \left( (\ats{n}{j} - \ath)^{+}(1-\ats{n}{i})b_{j}^n - (\ats{n+1}{j} - \ath)^{+}d_j^n \right) \varphi_{j}^{n},
\end{align}
 with $N_{\ast} = T_{\ast}/\delta$. The fact $\varphi^{N_\ast}_j=0$ for all $j$ and a use of \eqref{eqn:disc_int_parts} yield
\begin{align}
      T_1 = -h\sum_{n=0}^{N_{\ast}-1} \sum_{j=0}^{J-1} (\varphi_{j}^{n+1} - \varphi_{j}^{n})\ats{n+1}{j} -  \int_{0}^{\ell_0} \alpha^0_h(x) \varphi(0,x)\,\mathrm{d}x 
      \label{eqn:ca2-t1}
\end{align}
where $\alpha_h^0$ is a piecewise constant function defined by $\alpha^0_{h\vert{\mathcal{X}_j}} = \langle \alpha_0 \rangle_{\mathcal{X}_j}$ for $j = 0,\ldots,J-1$ (see Discrete scheme~\ref{defn:semi_disc_soln}). A direct calculation shows the first term in the right hand side of~\eqref{eqn:ca2-t1} is equal to
\begin{align}
-\sum_{n=0}^{N_{\ast}-1} \sum_{j=0}^{J-1} \alpha_{j}^{n+1} \int_{\mathcal{X}_j} \int_{n\delta}^{(n+1)\delta} \partial_{t} \varphi(t,x)\,\mathrm{d}t = -\int_{0}^{\ell_m} \int_{\delta}^{T_{\ast} + \delta} \alpha_{h,\delta}(t,x) \partial_{t} \varphi(t-\delta,x)\,\mathrm{d}t\,\mathrm{d}x. \nonumber 
\end{align}
Since $\alpha_{h,\delta} \rightarrow \alpha$ almost everywhere (see Theorem~\ref{thm:main_thm_1}) as $h,\delta\to 0$, a use of Lebesgue's dominated convergence theorem shows that the first term in the right hand side of~\eqref{eqn:ca2-t1} converges to
 $-\int_{0}^{{\ell_m}} \int_{0}^{T_{\ast}} \alpha(t,x) \partial_{t} \varphi(t,x) \,\mathrm{d}t\,\mathrm{d}x.$

Since $\alpha_h^0 \rightarrow \alpha_0$ in $L^2(0,\ell_0)$, the second term in the right hand side of~\eqref{eqn:ca2-t1} converges to $-\int_{0}^{\ell_0} \alpha_0(x) \varphi(0,x)\,\mathrm{d}x$. An application of~\eqref{eqn:split_id_1} on $T_2$ yields
\begin{align}
T_2  ={}&   \delta \sum_{n=0}^{N_{\ast}-1} \sum_{j=0}^{J-1} \varphi_{j}^{n} \left( |u_{j+1}^n|\dfrac{\ats{n}{j} - \ats{n}{j+1}}{2} - |u_{j}^n|\dfrac{\ats{n}{j-1} -\ats{n}{j}}{2} \right) \nonumber \\
&+ \delta \sum_{n=0}^{N_{\ast}-1} \sum_{j=0}^{J-1} \varphi_{j}^{n} \left( u_{j+1}^n\dfrac{\ats{n}{j} + \ats{n}{j+1}}{2} - u_{j}^n\dfrac{\ats{n}{j-1} + \ats{n}{j}}{2} \right) \nonumber =: T_{21} + T_{22}.
\end{align}
 A use of $u_{0}^n = 0$ and $u_{J}^n = 0$ leads to
\begin{align}
|T_{21}| ={}&  \left| \delta \sum_{n=0}^{N_{\ast}-1} \sum_{j=0}^{J-2} (\varphi_{j}^{n} - \varphi_{j+1}^n) |u_{j+1}^n|\dfrac{\ats{n}{j} - \ats{n}{j+1}}{2}\right| \\
&\le \dfrac{h}{2} ||u_{h,\delta}||_{L^\infty(\mathscr{D}_{T_{\ast}})} ||\partial_x \varphi(t,x)||_{L^\infty(\mathscr{D}_{T_{\ast}})} \sum_{n=0}^{N_{\ast}-1} \delta \sum_{j=0}^{J-2} |\ats{n}{j} - \ats{n}{j+1}|,
\end{align}
and hence~\eqref{eqn:u_bound} and~\eqref{eqn:alpha_sbv} yield $|T_{21}| \rightarrow 0$ as $h \rightarrow 0$. Use \eqref{eqn:disc_int_parts} and $u_{0}^n = 0$ and $\varphi_{J}^n = 0$ to obtain
\begin{align}
T_{22} = -\delta \sum_{n=0}^{N_{\ast}-1} \sum_{j=0}^{J-1} (\varphi_{j+1}^n - \varphi_j^n) u_{j+1}^n \dfrac{\alpha_{j}^n + \alpha_{j+1}^n}{2}. 
\label{eqn:ii1_3}
\end{align}
Add and subtract $\delta \sum_{n=0}^{N_{\ast}-1} \sum_{j=0}^{J-1} (\varphi_{j+1}^n - \varphi_j^n) \frac{u_{j}^n}{2} \alpha_{j}^n$ to~\eqref{eqn:ii1_3} to obtain
\begin{align}
T_{22} ={}& \delta \sum_{n=0}^{N_{\ast}-1} \sum_{j=0}^{J-1} \dfrac{u_{j+1}^n\alpha_{j+1}^n}{2} (\varphi_{j+1}^n - \varphi_{j}^n - \varphi_{j+2}^n + \varphi_{j+1}^n) \nonumber \\
&- \delta\sum_{n=0}^{N_{\ast}-1} \sum_{i=0}^{J-1} (\varphi_{j+1}^n - \varphi_j^n) \dfrac{u_{j+1}^n + u_{j}^{n}}{2} \alpha_{j}^n 
\label{eqn:ca2-ta2}
\end{align}
We show that the first term on the right hand side of~\eqref{eqn:ca2-ta2} converges to zero. A use of the definition of $\varphi_{j}^n$, mean value theorem, and the CFL condition~\eqref{eqn:cfl_condition} yields
\begin{align} 
&  \left |\delta \sum_{n=0}^{N_{\ast}-1} \sum_{j=0}^{J-1} \dfrac{u_{j+1}^n\alpha_{j+1}^n}{2} (\varphi_{j+1}^n - \varphi_{j}^n - \varphi_{j+2}^n +  \varphi_{j+1}^n) \right| \nonumber \\
&\lesssim \delta ||u_{h,\delta}\alpha_{h,\delta}||_{L^{\infty}(\mathscr{D}_{T_{\ast}})} ||\partial_{xx} \varphi||_{L^\infty(\mathscr{D}_{T_{\ast}})} \sum_{n=0}^{N_{\ast}-1}\delta \sum_{j=0}^J h \rightarrow 0 \text{ as } \delta \rightarrow 0, \nonumber
\end{align}
where $\mathscr{C}_g$ is a constant independent of $h$ and $\delta$. Define $\partial_{h,\delta} \varphi : \mathscr{D}_{T_{\ast}} \rightarrow \mathbb{R}$ by $\partial_{h,\delta} \varphi := (\varphi_{j+1}^n - \varphi_j^n)/h$ on $\mathcal{T}_n \times \mathcal{X}_j$. Use the fact $u_{h,\delta} = \pmb{\chi}_{A_{h,\delta}}\widehat{u}_{h,\delta}$ and the trapezoidal quadrature rule on the piecewise linear function $u_{h,\delta}$ to express the second term in the right hand side of~\eqref{eqn:ca2-ta2} as
\begin{align}
-\int_{0}^{T_{\ast}}\int_{0}^{{{\ell_m}}} u_{h,\delta}\alpha_{h,\delta}\partial_{h,\delta}\varphi\,\mathrm{d}x\,\mathrm{d}t  ={}& -\int_{0}^{T_{\ast}}\int_{0}^{{{\ell_m}}} \pmb{\chi}_{A_{h,\delta}}\widehat{u}_{h,\delta}\alpha_{h,\delta}\partial_{h,\delta}\varphi\,\mathrm{d}x\,\mathrm{d}t \\
&\rightarrow -\int_{0}^{T_{\ast}}\int_{0}^{{{\ell_m}}} u\,\alpha\,\partial_{x}\varphi\,\mathrm{d}x\,\mathrm{d}t,
\end{align}
where Lemmas~\ref{appen_id.e} and \ref{appen_id.f} are applied in the last step. Write  $T_{3}$ as 
\begin{small}
\begin{align}
T_3 ={}& h\delta \sum_{n=0}^{N_{\ast}-1} \sum_{j=0}^{J-1}  (\ats{n}{j} - \ath)^{+}(1-\ats{n}{j}) b_{j}^n\varphi_{j}^{n} - h\delta \sum_{n=0}^{N_{\ast}-1} \sum_{j=0}^{J-1}  (\ats{n+1}{j} - \ath)^{+}d_j^n\varphi_{j}^{n}. 
\label{eqn:ca2-t3}
\end{align}
\end{small}
\noindent Use definitions of $b_{j}^n$, $d_{j}^n$, and $\varphi_{j}^n$ to rewrite the first term in the right hand side of~\eqref{eqn:ca2-t3} and use Lemmas~\ref{appen_id.e} and \ref{appen_id.f} (see Appendix~\ref{appen_B}) to arrive at the following convergence
\begin{gather}
\int_{0}^{T_{\ast}} \int_{0}^{{\ell_m}} (\alpha_{h,\delta}(t,x) - \ath)^{+}(1 - \alpha_{h,\delta}(t,x)) \dfrac{(1 + s_1)\Pi_{h,\delta}c_{h,\delta}(t,x)}{1 + s_1\Pi_{h,\delta}c_{h,\delta}(t,x)} \varphi(t,x)\,\mathrm{d}x\,\mathrm{d}t \nonumber \\
 \rightarrow \int_{0}^{T}\int_{0}^{{\ell_m}} (\alpha - \ath)^{+}(1 - \alpha) \dfrac{(1 + s_1)c}{1 + s_1c} \varphi\,\mathrm{d}x\,\mathrm{d}t. \nonumber
\end{gather}
A similar argument shows that the second term in the right hand side of~\eqref{eqn:ca2-t3} converges to $-\int_{0}^{T}\int_{0}^{{\ell_m}} (\alpha - \ath)^{+} \frac{s_2 + s_3c}{1 + s_1c} \varphi\,\mathrm{d}x\,\mathrm{d}t.$ Plugging the above in $T_1+T_2=T_3$ concludes the proof.
\end{proof}
\begin{proposition}[Step~\ref{ca.3}]
	\label{prop:u_conv}
	Let $\widehat{u} : \mathscr{D}_{T_\ast} \rightarrow \mathbb{R}$ be a limit provided by Theorem~\ref{thm:main_thm_1} such that $\widehat{u}_{h,\delta} \halfarrow \widehat{u}$ weakly in $L^2(\mathscr{D}_{T_\ast})$ and $\partial_{x}\widehat{u}_{h,\delta} \halfarrow \partial_x \widehat{u}$ weakly in $L^2(\mathscr{D}_{T_\ast})$. Then, for every $v \in \Hdx{u}(D_{T}^{\thr})$ such that $v(\cdot,0) = 0$,  $\widehat{u}_{\vert{D_{T}^{\thr}}}$ satisfies~\eqref{eqn:weak_vel}.
\end{proposition}
\begin{proof}
Let $v \in \mathscr{C}^\infty(D_{T_\ast}^\thr)$ with $v(\cdot,0) = 0$. Redefine $v$ to be a smooth extension to $\mathscr{D}_{T_{\ast}}$ for ease of notation. Define $v_{h,\delta}(t,x) = \mathcal{I}_h v (t_n,x)$ on $\mathcal{T}_n \times \mathcal{X}_j$ for $n,j \ge 0$. The piecewise linear in space and piecewise constant in time function $v_{h,\delta}$ satisfies $v_{h,\delta} \rightarrow v$ and $\partial_x v_{h,\delta} \rightarrow \partial_x v$ strongly in $L^2(\mathscr{D}_{T_{\ast}})$.

Take the test function as $v_{h,\delta}(t_n,\cdot)$ in~\eqref{eqn:dweak_vel}, multiply with $\delta \pmb{\chi}_{A_{h,\delta}}(t_n,\cdot)$, use the fact that $u_{h,\delta} = \pmb{\chi}_{A_{h,\delta}}\widehat{u}_{h,\delta}$, and sum over $n = 1,\ldots,N_{\ast}-1$ to obtain $T_1 + T_2 = T_3$, where 
		\begin{align}
		T_1 &:= \int_0^{T_\ast} \int_{0}^{\ell_m} \pmb{\chi}_{A_{h,\delta}} \dfrac{k \alpha_{h,\delta}}{1 - \alpha_{h,\delta}} \widehat{u}_{h,\delta}v_{h,\delta}\,\mathrm{d}x\mathrm{d}t, \\
		T_2 &:= \int_0^{T_\ast} \int_{0}^{\ell_m} \pmb{\chi}_{A_{h,\delta}} \mu \alpha_{h,\delta} \partial_x \widehat{u}_{h,\delta}\partial_{x} v_{h,\delta}\,\mathrm{d}x\mathrm{d}t,  \text{ and }\\
		T_3 &:= \int_0^{T_\ast} \int_{0}^{\ell_m} \pmb{\chi}_{A_{h,\delta}} \mathscr{H}({\alpha_{h,\delta}}) \partial_x v_{h,\delta} \mathrm{d}x\mathrm{d}t.
		\end{align}
		
		\noindent We have $\pmb{\chi}_{A_{h,\delta}} \rightarrow  \pmb{\chi}_{D_{T_\ast}^\thr}$ almost everywhere and $\alpha_{h,\delta} \rightarrow  \alpha$ in $L^2(\mathscr{D}_{T_{\ast}})$. Therefore, Lemmas~\ref{appen_id.e} and \ref{appen_id.f}  show that
		\[
		T_1 \rightarrow \int_{0}^{T_{\ast}}\int_{0}^{{\ell_m}}  \pmb{\chi}_{D_{T_\ast}^\thr} \dfrac{k \alpha}{1 - \alpha} \widehat{u} v\,\mathrm{d}x\mathrm{d}t = \iint_{D_{T_\ast}^\thr} \dfrac{k \alpha}{1 - \alpha} u v\,\mathrm{d}x\mathrm{d}t.
		\]
		A similar argument for $T_2$ shows that
		\[
		T_2 \rightarrow \int_{0}^{T_{\ast}}\int_{0}^{{\ell_m}}  \pmb{\chi}_{D_{T_\ast}^\thr}\,\mu\,\alpha\,\partial_x \widehat{u} \,\partial_x v\,\mathrm{d}x\mathrm{d}t = \iint_{D_{T_\ast}^\thr} \mu \alpha \partial_x u\,\partial_x v\,\mathrm{d}x\mathrm{d}t.
		\]
		Since $\mathscr{H}$ is continuous, $\mathscr{H}(\alpha_{h,\delta}) \rightarrow \mathscr{H}(\alpha)$ almost everywhere in $\mathscr{D}_{T_{\ast}}$. Therefore,
		\[
		T_3\rightarrow \int_{0}^{T_{\ast}}\int_{0}^{{\ell_m}} \pmb{\chi}_{D_{T_\ast}^\thr} \mathscr{H}(\alpha) \partial_x v \mathrm{d}x\mathrm{d}t
	= \iint_{D_{T_\ast}^\thr} \mathscr{H}(\alpha) \partial_x v \mathrm{d}x\mathrm{d}t.
		\]
These convergences, the relation $T_1 + T_2 = T_3$, and the density of $\mathscr{C}^\infty(D_{T_\ast}^\thr)$ in $\Hdx{u}(D_{T}^{\thr})$  yield the desired result.
\end{proof}

To establish~\eqref{eqn:weak_ot} we start with a definition and a covering lemma.

\begin{figure}[h!]
    \centering
    \includegraphics[scale=0.7]{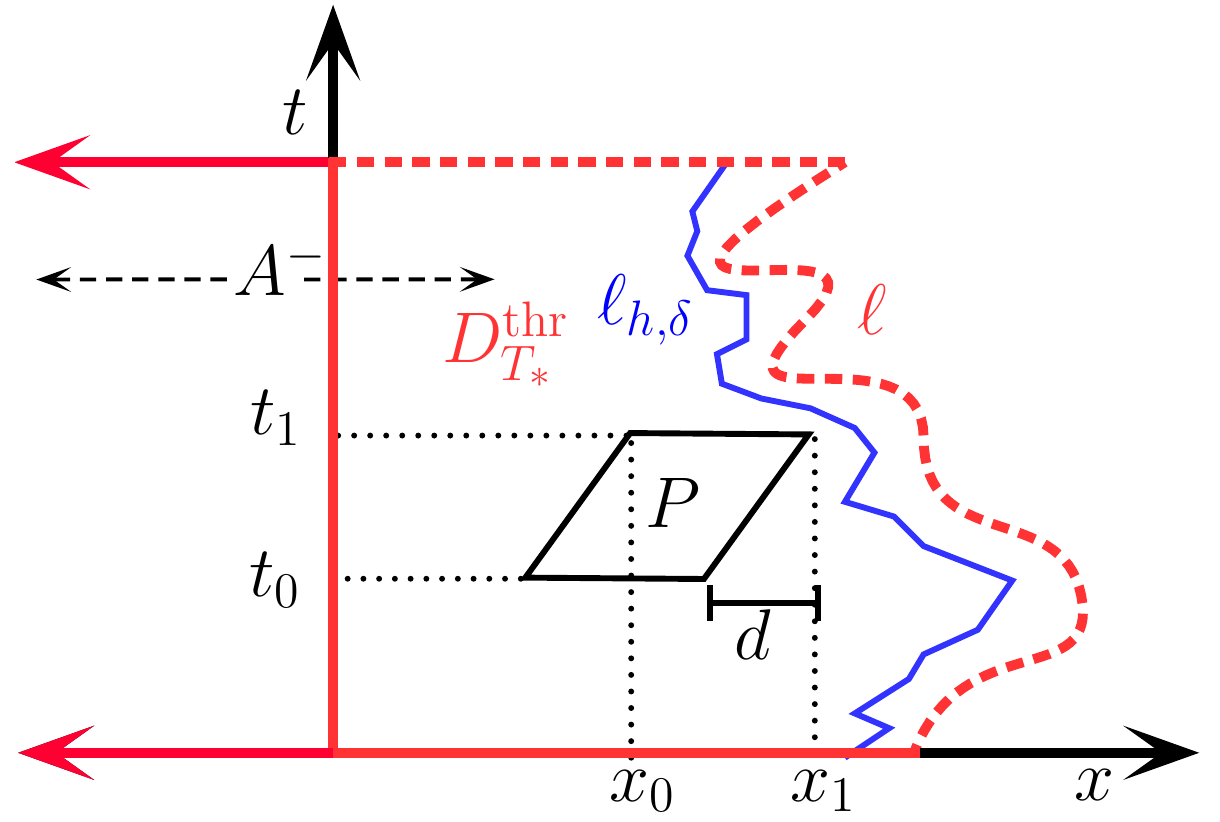}
    \caption{The domain $A$ and $A^{-}$ are the geometries described in Lemma~\ref{lemma:covering}, and $P$ is a right--leaning parallelogram, and $d=(\rho\mathscr{C}_{\CFL})^{-1}(t_1-t_0)$.}
    \label{fig:left_sln_p}
\end{figure}
\begin{lemma}[Covering lemma]
\label{lemma:covering}
For $x_0 < x_1$ and $t_0 < t_1$, let  
\begin{equation}
P := \bigcup_{t_0 \leq t \leq t_1} \{t\} \times [x_0 - (\rho\mathscr{C}_{\CFL})^{-1}(t_1 - t),x_1 - (\rho\mathscr{C}_{\CFL})^{-1}(t_1 - t)]
\label{eqn:left_slanted}
\end{equation}
be a right--leaning parallelogram (see Figure~\ref{fig:left_sln_p}) contained in $A^{-} := D_{T_\ast}^\thr\cup(\{0\}\times[0,\ell(0))\cup([0,T)\times\mathbb{R}^{-})$. Then, there exists an $ h_P>0$ and a $\delta_P>0$ such that, for every $h \leq h_P$ and $\delta \leq \delta_P$, $P \subset A_{h,\delta}^{-}:= A_{h,\delta}\cup(\{0\}\times[0,\ell(0))\cup([0,T)\times\mathbb{R}^{-})$.
\end{lemma}
\begin{proof}
  From~\eqref{eqn:left_slanted} and $P \subset A^{-}$, we have $\ell(t_1) > x_1 + \epsilon$ for some $\epsilon>0$. Without loss of generality, assume that $\ell_{h,\delta}(t_1) \rightarrow \ell(t_1)$ or consider a $\widetilde{t}_1$ arbitrarily close to $t_1$ such that $\ell_{h,\delta}(\widetilde{t}_1) \rightarrow \ell(\widetilde{t}_1)$. The existence of $\widetilde{t}_1$ is guaranteed by the fact that $\ell_{h,\delta} \rightarrow \ell$ almost everywhere. In this case, there exists an $h_P$ and a $\delta_P$ such that $\ell_{h,\delta}(t_1)>x_1$ for every $h \leq h_{P}$ and $\delta \leq \delta_P$, which means that $\ell_{h,\delta,D}(t_1) > x_1 - \ell_{h,\delta,BV}(t_1)$, where $\ell_{h,\delta,D}$ and $\ell_{h,\delta,BV}$ are obtained from the proof of Proposition~\ref{prop:lhd_bv}. Since $\ell_{h,\delta,D}$ is decreasing, for $t \in [t_0,t_1]$ we have $\ell_{h,\delta,D}(t) > x_1 - \ell_{h,\delta,BV}(t_1)$ and
\begin{align}
    \ell_{h,\delta,D}(t) + \ell_{h,\delta,BV}(t) &> x_1  - \ell_{h,\delta,BV}(t_1) + \ell_{h,\delta,BV}(t) \\
    &\geq x_1 - (\rho\mathscr{C}_{\CFL})^{-1}(t_1 - t) .
\end{align}
Therefore, for $t\in [t_0,t_1]$, $\ell_{h,\delta}(t) > x_1 - (\rho \mathscr{C}_{\CFL})^{-1}(t_1 - t)$ , which yields $P \subset A_{h,\delta}^{-}$.  
\end{proof}

\begin{remark}
\label{rem:c_smooth_app}
Let $v \in \mathscr{C}_c^\infty(A^{-})$. Then, $\text{supp}(v)$ is compact in $A^{-}$ and  can be covered by a finite number of right leaning type parallelograms $\{P_i\}_{i}$. Since there exists a $C_{c}^\infty$ partition of unity $\{\zeta_{i}\}_i$ subordinate to $\{P_i\}_{i}$, we can write $v = \sum_{i} v \zeta_i$ and $\text{supp}(v\zeta_i) \subset P_i$. Then, for any $h < h_0$ and $\delta < \delta_0$, where $h_0 = \min_{i}{h_{P_i}},\,\delta_{0} = \min_{i}{\delta_{P_i}}$, the support of $v$ is contained in $A_{h,\delta}^{-}$, and $v \in \mathscr{C}_c^\infty(A_{h,\delta}^{-})$.
\end{remark}

\begin{remark}
\label{rem:a_smooth_app}
The fact that oxygen tension satisfies the Neumann boundary condition~\eqref{eqn:bdr_cond_2} forces a test function in~\eqref{eqn:weak_ot} not to vanish at the boundary $(0,T_{\ast}] \times \{0\}$ of $D_{T_\ast}^\thr$. This requirement forces us to consider $A^{-}$ instead of $D_{T_\ast}^\thr$ in Lemma~\ref{lemma:covering}. Since we can extend any function $v \in \mathscr{C}^\infty(D_{T_\ast}^\thr)$ with  $v(t,\ell(t)) = 0$ smoothly to $A^{-}$, the proof of Proposition~\ref{prop:c_conv} is not affected by this consideration of $A^{-}$.
\end{remark}
Next, we show that oxygen tension $c$ satisfies~\eqref{eqn:weak_ot}.
\begin{proposition}[Step~\ref{ca.4}]
\label{prop:c_conv}
Let $c : \mathscr{D}_{T_\ast} \rightarrow \mathbb{R}$ be the limit provided by Theorem~\ref{thm:main_thm_1}. Then, for every $v \in \Hdx{c}(D_{T}^{\thr})$ such that $\partial_t v\in L^2(D_{T_\ast}^\thr)$, $c_{\vert D_{T_\ast}^\thr}$ satisfies~\eqref{eqn:weak_ot}.
\end{proposition}
\begin{proof}
Since $v \in \Hdx{c}(D_{T}^{\thr})$ can be approximated by functions in $\mathscr{C}^\infty(D_{T_\ast}^\thr)$ with $v(t,\ell(t)) = 0$ for all $t \in (0,T_{\ast})$, by Remarks~\ref{rem:c_smooth_app} and~\ref{rem:a_smooth_app}  it is sufficient to consider functions $v \in \mathscr{C}_c^{\infty}(P)$, where $P \subset A^{-}$ is a right--leaning parallelogram.

Choose $v \in \mathscr{C}_c^\infty(P)$. There exists an $h$ and  a $\delta$ small enough such that $v  \in \mathscr{C}_c^\infty(A_{h,\delta}^{-})$ by Remark~\ref{rem:c_smooth_app}. Define $v_{h,\delta}(t,x)= \mathcal{I}_h v(t_n,x)$ for $(t,x)\in \mathcal{T}_n \times \mathcal{X}_j$ for $n,\,j \ge 0$. The piecewise linear in space and piecewise constant in time function $v_{h,\delta}$ satisfies the following properties: (a) $v_{h,\delta} \in L^2(0,T_{\ast};H^{1}(0,{{\ell_m}}))$, (b) for $n \geq 0$, $v_{h,\delta}(t_n,\ell_h^n) = 0$, (c) $v_{h,\delta} = 0$ on $\mathscr{D}_{T^{\ast}}\setminus \overline{A_{h,\delta}}$, and (d) $v_{h,\delta}(T_{\ast},\cdot) = 0$.

    In \eqref{eqn:bform_ot}, take the test function as $v_{h,\delta}(t_n,\cdot)$ and sum over $n = 1,\ldots,N_{\ast}$ to obtain $T_1 + T_2 = T_3$, where 
\begin{align}
T_1 &= \sum_{n=1}^{N_{\ast}} \int_{0}^{{\ell_m}}  (\Pi c_{h,\delta}(t_n,x) - \Pi c_{h,\delta}(t_{n-1},x)) \Pi v_{h,\delta} (t_n,x)\,\mathrm{d}x, \\
   T_2 &:= \sum_{n=1}^{N_{\ast}} \lambda \delta \int_{0}^{{\ell_m}} \partial_x c_{h,\delta}(t_n,x) \partial_{x} v_{h,\delta}(t_n,x)\,\mathrm{d}x, \text{ and } \\
T_3 &:= -Q\sum_{n=1}^{N_{\ast}} \delta \int_{0}^{{\ell_m}} \dfrac{\alpha_{h,\delta}(t_n,x) \Pi_h c_{h,\delta}(t_n,x)}{1 + \widehat{Q}_1 |\Pi_h c_{h,\delta}(t_{n-1},x)|} \Pi_h v_{h,\delta} (t_n,x)\,\mathrm{d}x.
\end{align}
Note that the space integrals in $T_1$, $T_2$, and $T_3$ are on $(0,\ell_h^n)$ for each $t_n$ by the property (c). A use of~\eqref{eqn:disc_int_parts} leads to
\begin{multline}
    T_1 = -\sum_{n=1}^{N_{\ast}} \int_{0}^{{\ell_m}} (\Pi_h v_{h,\delta} (t_n,x) - \Pi_h v_{h,\delta} (t_{n-1},x)) \Pi_h c_{h,\delta}(t_n,x)\,\mathrm{dx}  \\
    +\int_{0}^{{\ell_m}}  \Pi_h v_{h,\delta}(T_{\ast},x) \Pi_h c_{h.\delta} (T_{\ast},x)\,\mathrm{d}x - \int_{0}^{{\ell_m}} \Pi_h v_{h,\delta}(0,x) \Pi_h c_{h.\delta} (0,x)\,\mathrm{d}x. 
\end{multline}
Using the property (c) and the strong convergences $\Pi_{h}c_{h,\delta}(0,\cdot) \rightarrow c_0(\cdot)$, $\Pi_{h} v_{h,\delta}(0,\cdot) \rightarrow v(0,\cdot)$, $\partial_t v_{h,\delta} \rightarrow \partial_t v$, $\Pi_h c_{h,\delta} \rightarrow c$ in $L^2(\mathscr{D}_{T_\ast})$, we deduce
\begin{align}
    T_1 \rightarrow{}& -\int_{0}^{T_{\ast}} \int_{0}^{{\ell_m}} c\,\partial_t v\,\mathrm{d}x\,\mathrm{d}t - \int_{0}^{{{\ell_m}}}c_{0}(x)v(0,x)\,\mathrm{d}x \\
    &= -\iint_{D_{T_\ast}^\thr} c\,\partial_t v\,\mathrm{d}x\,\mathrm{d}t - \int_{0}^{\ell(0)}c_0(x)\,v(0,x)\,\mathrm{d}x.
\end{align}

\noindent The weak convergence $\partial_x c_{h,\delta} \halfarrow c$, the strong convergence $\partial_x v_{h,\delta} \rightarrow \partial_x v$ in $L^2(\mathscr{D}_{T_\ast})$, and an application of Lemma~\ref{appen_id.e} yield
\begin{align}
    T_2 = \lambda \int_{0}^{T_{\ast}} \int_{0}^{{{\ell_m}}}  \partial_x c_{h,\delta} \partial_{x} v_{h,\delta} \,\mathrm{d}x\,\mathrm{d}t &\rightarrow \lambda \int_{0}^{T_{\ast}} \int_{0}^{{{\ell_m}}}  \partial_x c\,\partial_{x} v \,\mathrm{d}x\,\mathrm{d}t \\
    &\qquad= \lambda \iint_{D_{T_\ast}^\thr} \partial_x c \,\partial_x v\,\mathrm{d}x\,\mathrm{d}t.
\end{align}
It is easily observed that $\Pi_{h,\delta} c_{h,\delta}/(1 + \widehat{Q}_1 |\Pi_{h,\delta} c_{h,\delta}|) \rightarrow c/(1 + \widehat{Q}_1|c|)$ in $L^2(\mathscr{D}_{T_{\ast}})$. Then, use of Lemma~\ref{appen_id.f} shows that $\alpha_{h,\delta}\Pi_{h,\delta} c_{h,\delta}/(1 + \widehat{Q}_1 |\Pi_{h,\delta} c_{h,\delta}|) \rightarrow \alpha c/(1 + \widehat{Q}_1|c|)$ in $L^2(\mathscr{D}_{T_\ast})$. Since $\Pi_h v_{h,\delta} \rightarrow v$ in $L^2(\mathscr{D}_{T_\ast})$ we obtain
\[
    T_3  \rightarrow -Q\int_{0}^{T_{\ast}} \int_{0}^{{{\ell_m}}}  \dfrac{\alpha c}{1 + \widehat{Q}_1|c|}\, v \,\mathrm{d}x\,\mathrm{d}t 
    = -Q\iint_{D_{T_\ast}^\thr}  \dfrac{\alpha c}{1 + \widehat{Q}_1|c|} \,v\,\mathrm{d}x\,\mathrm{d}t.
\]

\noindent Plugging the above in $T_1 + T_2 = T_3$ yields the desired result.
\end{proof}

 This concludes the proof of Theorem~\ref{thm:main_thm_2}, and thereby convergence of the Discrete scheme~\ref{defn:semi_disc_soln} to a threshold solution (see Definition~\ref{defn:ext_soln}).

\section{Numerical results}
\label{sec:num_results}
In Subsection~\ref{sec:numerical_example}, we present the solution of the Discrete scheme~\ref{defn:semi_disc_soln} for a fixed set of parameters and discretisation factors, and discuss it's important physical and numerical features.  In Subsection~\ref{sec:opt_ex}, we study the dependency  of $T_{\ast}()$, the time below which a threshold solution exists, on the parameters $a_{\ast}$, $a^{\ast}$, $m_{02}$ and $\alast$. 
\subsection{Numerical example}
\label{sec:numerical_example}
The parameters are chosen as in~\cite{breward_2002}: $k = 1$, $\mu = 1$, $Q = 0.5$, $\widehat{Q}_1 = 0$, $s_1 = 10 = s_4$, $s_2 = 0.5 = s_3$, and $\alast = 0.8$. The bounds of the cell volume fraction are set to be $a_{\ast} = 0.4$ and $a^{\ast} = 0.82$. The extended domain length $\ell_m$ is set as 10. The threshold value is taken as $\ath = 0.1$. With these choices the constant $\mathscr{C}_{\CFL}$ is  $0.0361$. Set $\rho = 0.1$ and choose $\delta = 10^{-3}$ and $h = 5\times 10^{-2}$, so that the condition~\eqref{eqn:cfl_condition} is satisfied. 

\begin{figure}[h!]
	\centering
	\begin{subfigure}[b]{0.48\textwidth}
		\includegraphics[scale=0.85]{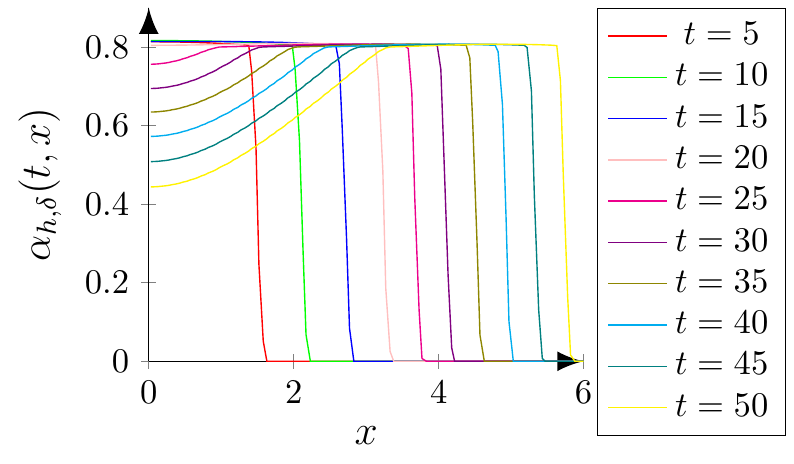}
		\subcaption{cell volume fraction}
		\label{fig:cvf}
	\end{subfigure}
	~ 
	\begin{subfigure}[b]{0.48\textwidth}
		\includegraphics[scale=0.85]{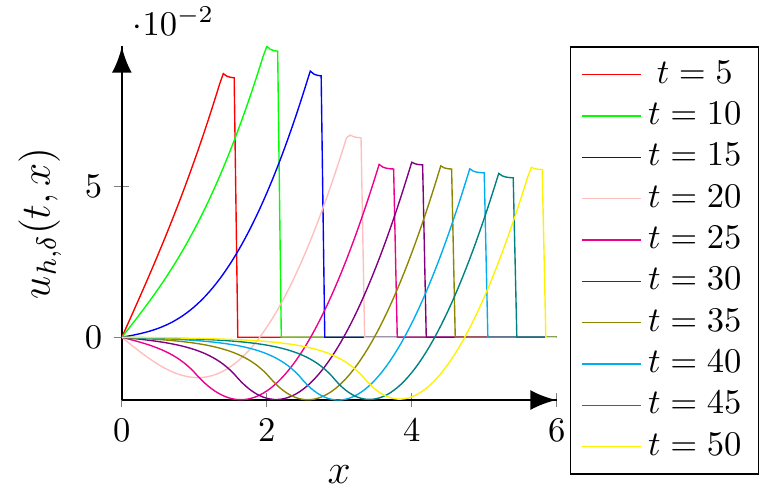}
		\subcaption{cell velocity}
		\label{fig:cvel}
	\end{subfigure}
	\begin{subfigure}[b]{0.48\textwidth}
		\includegraphics[scale=0.85]{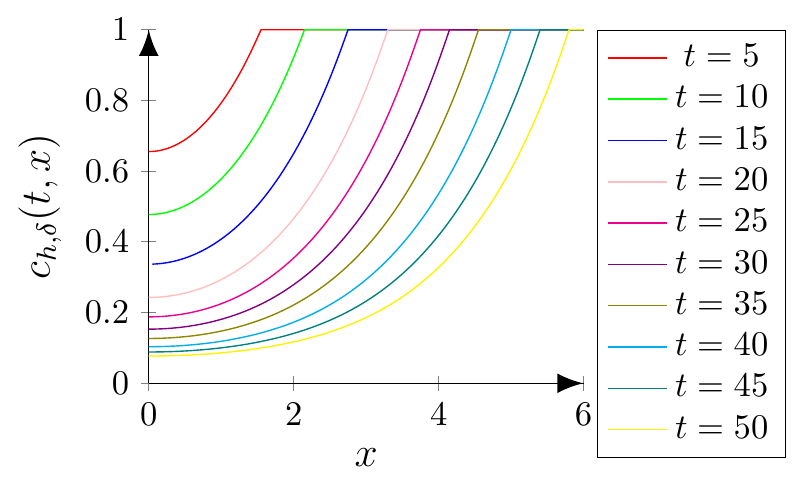}
		\subcaption{oxygen tension}
		\label{fig:oten}
	\end{subfigure}
	~ 
	\begin{subfigure}[b]{0.48\textwidth}
		\includegraphics[scale=0.85]{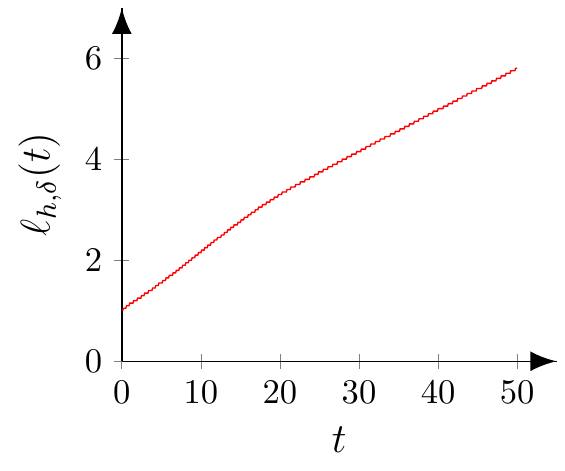}
		\subcaption{tumour radius}
		\label{fig:rval}
	\end{subfigure}
	~ 
	\caption{Numerical solution of the Discrete scheme~\ref{defn:semi_disc_soln} with $\delta = 10^{-3}$ and $h = 5\times 10^{-2}$ is depicted. A curve in each of the  Figures~\ref{fig:cvf},~\ref{fig:cvel}, and~\ref{fig:oten} represents the spatial variation of cell volume fraction, cell velocity, and oxygen tension, respectively on the tumour domain $(0,\ell_{h,\delta}(t))$ at a time $t$ as colour-coded in the legends. Figure~\ref{fig:rval} represents the evolution of the tumour radius $\ell(t)$ with respect to the time.}
	\label{fig:m_ge_alpha}
\end{figure}
The final time is set to be $T_{\ast} = 50$. We plot the variation of $\alpha_{h,\delta}(t,\cdot)$, $u_{h,\delta}(t,\cdot)$ and $c_{h,\delta}(t,\cdot)$ for the times $t \in \{5,10,\ldots,50\}$ on the corresponding domains $(0,\ell_{h,\delta}(t))$ in Figures~\ref{fig:cvf},~\ref{fig:cvel}, and~\ref{fig:oten}, respectively. The variation of $\ell_{h,\delta}(t)$ with respect to time is depicted in~\ref{fig:rval}. We observe from Figures~\ref{fig:cvf} and~\ref{fig:oten} that the volume fraction and oxygen tension decrease towards $x = 0$ due to the slower diffusion of oxygen towards $x = 0$  and the accelerated cell death owing to nutrient starvation. This effect is more noticeable in larger tumours than smaller ones. The positive value of cell velocity towards the tumour boundary and negative value towards the interior suggests that the outermost cells flow outwards and the internal cells flow inwards.  Note that $c_{h,\delta}$ is unity at $\ell_{h,\delta}(t)$, and this unlimited supply of nutrient results in the steady increase of tumour size as illustrated in Figure~\ref{fig:rval}.

\subsection{Optimal time of existence}
\label{sec:opt_ex}
The time $T_{\ast}$ below which a threshold solution exists (obtained in Proposition~\ref{prop:dalpha_max}) depends on the parameters $a_{\ast}$, $a^{\ast}$, $m_{02}$, and $\alast$. We can always fix $\ell_m$ large enough so that $\rho\mathscr{C}_{\CFL}(\ell_m - \ell_0)$ is larger than $T_m$ and $T_M$, so that $T_\ast=\min(T_m,T_M)$ (see Proposition~\ref{prop:dalpha_max}). The time $T_m$ provided by~\eqref{eqn:max_t_tm} is a decreasing function of $\mathcal{F}_{\rm min}$. The fact that $\mathcal{F}_{\rm min} \geq 0$ yields $T_m \le \log(\ath/a_{\ast})/s_2$, which precisely occurs when $a^{\ast} =\alast$ (if and only if $\mathcal{F}_{\rm min} = 0)$. The time $T_M$ provided by~\eqref{eqn:tmax} requires a more careful analysis. The domain of $T_M$ as a function of $a^{\ast}$ is $(m_{02},1]$. However, $T_{M}$ is zero at both $a^{\ast} = m_{02}$ and $a^{\ast} = 1$ (since $\lim_{a^{\ast} \rightarrow 1}\mathcal{F}_{\rm max} = \infty$). Therefore, $T_M$ has the maximum between $a^{\ast} = m_{02}$ and $a^{\ast} = 1$. Here, we need to consider three cases. If $m_{02} > \alast$, then $T_{\ast}$ attains the maximum at an $a^{\ast}$ between $m_{02}$ and $1$ (see Figure~\ref{fig:Tstar}).
\begin{figure}[h!]
	\centering
	\begin{subfigure}[b]{0.48\textwidth}
		\includegraphics[scale=0.9]{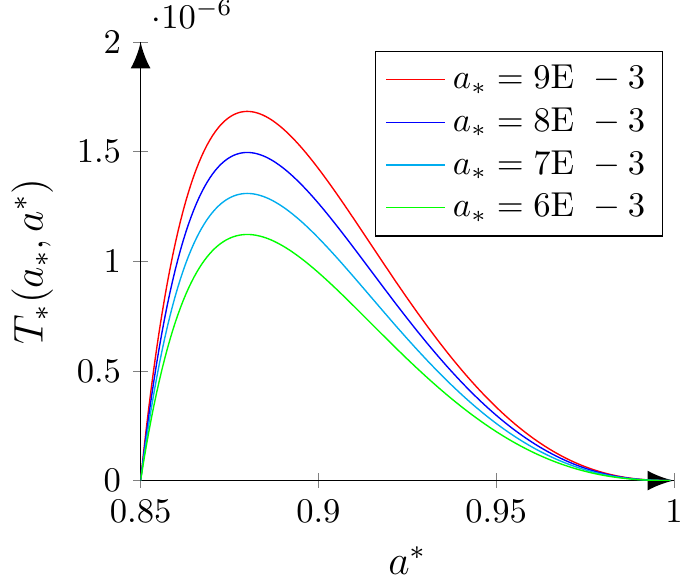}
		\subcaption{$m_{02} = 0.85$}
	\end{subfigure}
	~ 
	\begin{subfigure}[b]{0.48\textwidth}
		\includegraphics[scale=0.9]{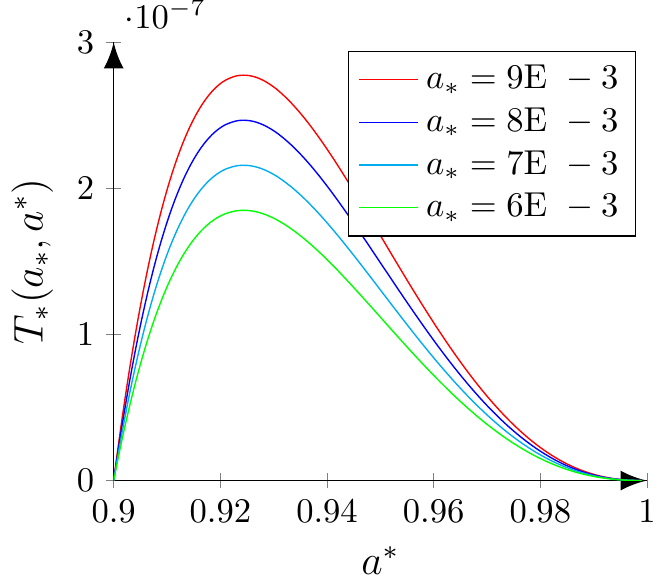}
		\subcaption{$m_{02} = 0.9$}
	\end{subfigure}
	~ 
	\caption{Variation of $T_{\ast}$ with respect to $a^{\ast}$ and $a_{\ast}$ when $m_{02} > \alast = 0.8$. }
	\label{fig:Tstar}
\end{figure}
If $m_{02} = \alast$, then $T_M$ attains the maximum between $a^{\ast} = \alast$ and $a^{\ast} = 1$. Since $T_{m}$ is decreasing on $[\alast,1]$, $T_{\ast}$ attains the maximum at an $a_{\ast}$ in $(\alast,1)$ (see Figure~\ref{fig:case_08_all}).
\begin{figure}[h!]
	\centering
	\begin{subfigure}[b]{0.48\textwidth}
		\includegraphics[scale=0.88]{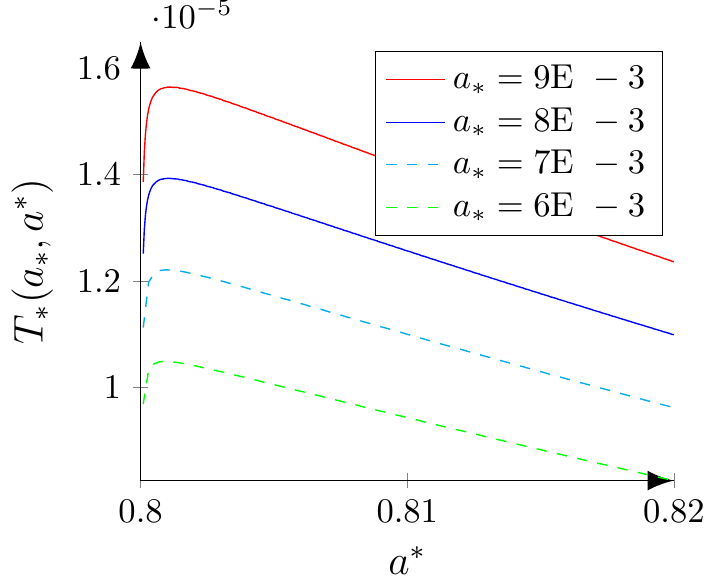}
		\subcaption{$m_{02} = 0.8$}
		\label{fig:case_08_all}
	\end{subfigure}
	~ 
	\begin{subfigure}[b]{0.48\textwidth}
		\includegraphics[scale=0.88]{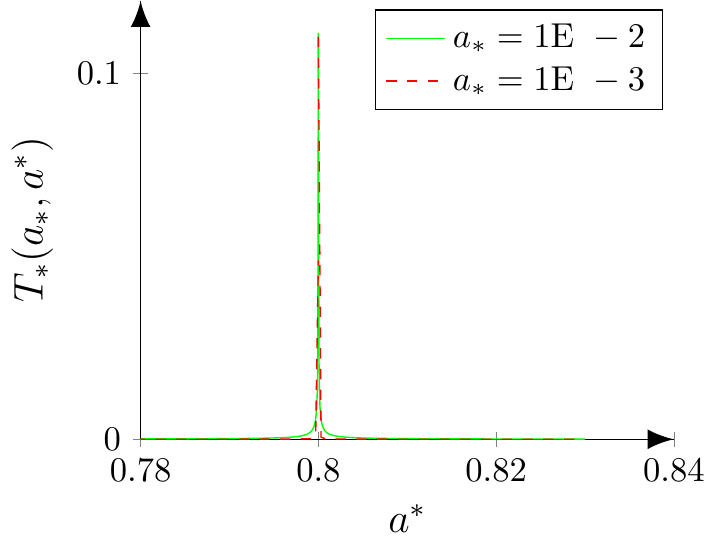}
		\subcaption{$m_{02} = 0.7$}
		\label{fig:case_07_all}
	\end{subfigure}
	~ 
	\caption{Variation of $T_{\ast}$ with respect to $a^{\ast}$ and $a_{\ast}$ when $m_{02} \leq \alast  = 0.8$. }
	\label{fig:Tstar.2}
\end{figure}
However, if $m_{02} < \alast$, then $T_{\ast}$ attains maximum exactly at $\alast$ since $\mathcal{F}_{\rm max}$ is minimal at $\alast$ and $a^{\ast} - m_{02}$ is increasing on $(m_{02},1)$ (see Figure~\ref{fig:case_07_all}). 

The time $T_M$ depends also on the lower bound $a_{\ast}$. The range of $a_{\ast}$ is $(0,\ath)$. From~\eqref{eqn:fmax} it is easy to observe that $\mathcal{F}_{\rm max}$ is a decreasing function of $a_{\ast}$. Hence $T_{\ast}$ increases as $a_{\ast}$ approaches $\ath$ which is evident from Figures~\ref{fig:Tstar},~\ref{fig:Tstar.2}, and~\ref{fig:t_decrease}.

\begin{figure}[h!]
	\centering
	\includegraphics[scale=1]{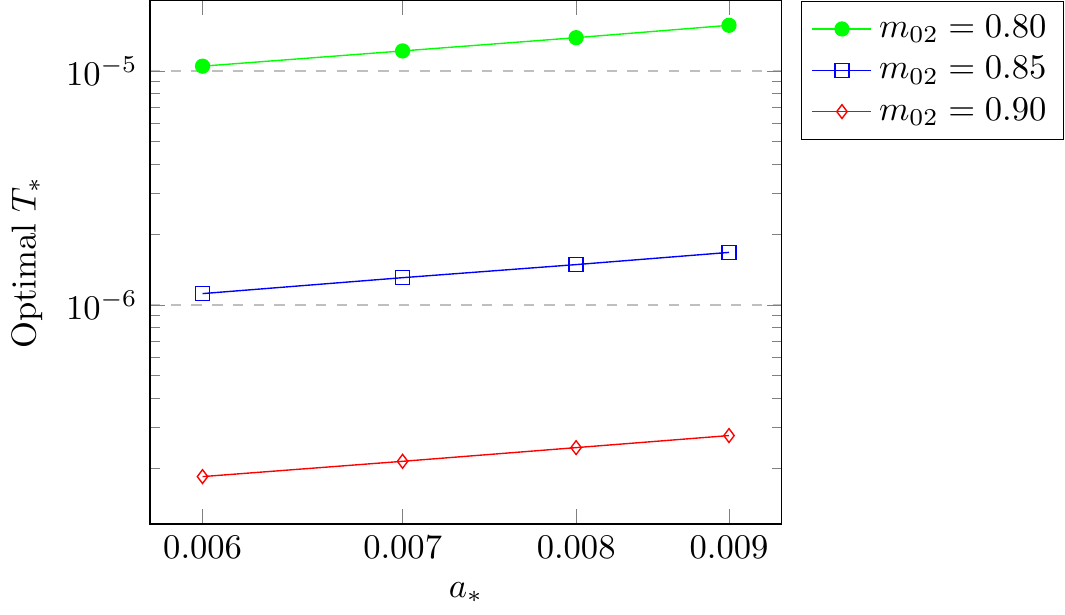}
	\caption{The dependence of optimal $T_{\ast}$ on $a_{\ast}$.}
	\label{fig:t_decrease}
\end{figure}

\begin{remark}[Sufficiency of Theorem~\ref{thm:main_thm_1}]
	The optimal value of $T_{\ast}$ found here is of order of $10^{-7}$ to $10^{-5}$, except when $m_{02} < \alast$ in which case $T_{\ast}\approx 0.12$. However, in practice, we observe that the Discrete  scheme~\ref{defn:semi_disc_soln} is stable, and thus convergent, up to at least a time of the order of $10^2$, as shown in Section~\ref{sec:numerical_example}. 
	In other words, the time $T_\ast$ derived in the proof of Proposition~\ref{prop:dalpha_max} is not restrictive, and only provides a sufficient condition for the convergence.
	
Also, it must be noted that $T_{\ast}$ is only restricted by the estimates on the model variables, in particular on cell volume fraction (see Proposition~\ref{prop:dalpha_max}). The convergence analysis (Theorem~\ref{thm:main_thm_2} and proofs) does not impose any restriction on $T_{\ast}$. Consequently, if the Discrete scheme~\ref{defn:semi_disc_soln} is stable (the proper norms remain bounded) up to a certain time, which can be partially assessed during numerical simulations, then the convergence analysis shows the limits of subsequences are threshold solutions of the continuous model.
\end{remark}

\section{Discussion}
\label{discussion}
 The flexible design of the tools in Sections~\ref{sec:main_thm} and~\ref{sec:convergence} allows us to apply Theorems~\ref{thm:main_thm_1} and~\ref{thm:main_thm_2} to models similar to~\eqref{system:nd_sys}; for instance the cut--off model \begin{align}
 \dfrac{k u \widetilde{\alpha}}{1 - \widetilde{\alpha}} - \mu \dfrac{\partial}{\partial x} \left(\widetilde{\alpha} \dfrac{\partial u}{\partial x}\right) &= -\dfrac{\partial}{\partial x} \left(  \mathscr{H}(\widetilde{\alpha}) \right), \\
 \dfrac{\partial c}{\partial t} - \lambda \dfrac{\partial^2 c}{\partial x^2} &= -\dfrac{Q \widetilde{\alpha} c}{1 + \widehat{Q}_1 c},
 \end{align}
 where the cut--off function is defined by $\widetilde{\alpha} := \min(\max(\alpha,\alpha_{m}),\alpha_{M})$, $\widetilde{\alpha}$ is governed by~\eqref{eqn:vol_fraction}, and $0 < \alpha_m  < \alpha_{M} < 1$ are fixed positive numbers.

 Another example is the growth model, wherein the oxygen tension is governed by 
 \begin{gather}
 \label{eqn:c_invivo}
 \dfrac{\partial c}{\partial t} - \lambda \dfrac{\partial^2 c}{\partial x^2} = -\dfrac{Q \alpha c}{1 + \widehat{Q}_1 c} \quad \forall\,(t,x) \in \mathscr{D}_T,\\
 \dfrac{\partial c}{\partial x}(t,0) = 0,\;c(t,\ell_{m}) = 1 \quad \forall\,t \in (0,T), \textrm{ and } \\
 c(0,x) = c_0(x)\;\;\forall\,x \in [0,{{{\ell_m}}}],
 \end{gather}
 where $\ell_{m}$ can be physically interpreted as the dimension of the growth platform in the \emph{in vitro case} or the location of the nearest capillary in the \emph{in vivo} case. The oxygen tension equation is defined in a fixed domain in this case.
 
A prospective research direction is to derive the results in this article for higher dimensional models. However, a higher dimensional setting offers many difficulties and a few important ones are briefly discussed here. We frequently use the embedding result that every function in $H^1(0,\ell_{m})$ is continuous and bounded. But, this result is not valid in $\mathbb{R}^2$ or $\mathbb{R}^3$. Consequently, we cannot use the energy norm estimates to obtain the boundedness of velocity in supremum norm, which in turn is essential to obtain boundedness and bounded variation of estimates on cell volume fraction. Secondly, to control the bounds on cell volume fraction, we need an additional supremum norm and bounded variation estimate on the divergence of the cell velocity field. This is a difficult task in two and three dimensions since the cell volume fraction that appear as a coefficient in the operators in the cell velocity equation is not a smooth function. Moreover, the challenges offered by the moving boundary are many fold. For instance, the moving boundary can make loops or knots, and these situations demand careful theoretical investigations.

\section{Conclusion}
In this paper, we achieved the following objectives: (a) designed a scheme for the threshold model and proved its convergence (up to a subsequence), and (b) established the existence of a threshold solution up to a finite time.  It is possible to extend the results derived in this article to similar models.  A few embedding results used in here apply only to the one--dimensional case, and hence a direct extension to higher dimensional models is challenging.   However, the article provides a proper framework to approach similar coupled problems of elliptic, hyperbolic, and parabolic equations in single or several spatial dimensions. It remains mostly open to develop a general theory for problems with degenerate equations; for instance,~\eqref{eqn:cel_velocity} which is only non--uniformly elliptic,  defined in time--dependent domains, which includes the study of well-posedness, design, and analysis of numerical schemes.

\subsubsection*{Acknowledgement}
The authors are grateful to Dr.\ Jennifer Anne Flegg, University of Melbourne, Australia for her valuable suggestions. The work of the first author was partially supported by the Australian Government through the Australian Research Council's Discovery Projects funding scheme (project number DP170100605). The second author gratefully acknowledges the local hospitality provided by Monash University, Australia during her visit, in May 2019.

\bibliographystyle{plain}
\bibliography{1d_theory_bib}

\appendix
\section*{Appendix}
\label{sec:appen}
\renewcommand{\thesubsection}{\Alph{subsection}}
\subsection{Expansions of abbreviations and notations}
\label{appen_A}
 Description of the notations used to denote model variables are tabulated in Table~\ref{tab:not}. The symbols $\alpha$, $u$, $c$, and $\ell$, with or without any math accents, always represent the cell volume fraction, cell velocity, nutrient concentration, and tumour radius. 
 \begin{longtable}{|M{1.7cm}|M{1.5cm}|M{4cm}|M{3.8cm}|}
		\hline
		Variables              & Domain                & Meaning      & Location of definition \\ \hline \hline
		$\check{\alpha},\check{u},\check{c},\check{\ell}$&  $D_T$                     &    model variables at the continuous level                 &              Model~\eqref{system:nd_sys}           \\ \hline
		$\alpha, u, c, \ell$	&          $D_T^{\thr}$             &  threshold solution                     &      Definition~\ref{defn:ext_soln}                 \\ \hline
		$\ell_{h}^n$	&  scalar                   &         discrete tumour radius at time $t_n$              &                   \ref{ds.b} of the Discrete scheme~\ref{defn:semi_disc_soln}      \\ \hline 
		$\widetilde{u}_h^n,\,\widetilde{c}_h^n$ & $(0,\ell_h^n)$ & discrete finite element solutions of the cell velocity and oxygen tension equation, resp. & ~\eqref{eqn:dweak_vel} and~\eqref{eqn:bform_ot} of the Discrete scheme~\ref{defn:semi_disc_soln} \\ \hline 
		$\alpha_h^n,\,u_h^n,\,c_h^n$ & $(0,\ell_m)$ & spatial discrete solutions at time $t_n$ & ~\eqref{eqn:uhn} and~\eqref{eqn:chn} of the Discrete scheme~\ref{defn:semi_disc_soln}\\ \hline
		$ \begin{array}{c}
		\alpha_{h,\delta},\,u_{h,\delta}, \\
		c_{h,\delta},\,\ell_{h,\delta}
		\end{array}$ & $\mathscr{D}_{T}$ & time--space discrete solutions &        Definitions~\ref{defn:time-recon} and~\ref{defn:disc_soln}                 \\ \hline
		$\widehat{u}_{h,\delta}$ & $\mathscr{D}_T$ & constant extension of $u_{h,\delta}(t,\cdot)$ to $(\ell_{h,\delta}(t),\ell_m)$, $t \in (0,T)$ & Eq.~\eqref{eqn:uhd_const} \\ \hline
			\caption{Notations used to denote the continuous and discrete variables}
			\label{tab:not}
\end{longtable}

\noindent The physical interpretations of the boundary conditions~\eqref{eqn:bdr_cond_1} -- \eqref{eqn:bdr_cond_2} are presented in Table~\ref{tab:boundary}. For further details, refer to~\cite[Section 2.2]{breward_2003,breward_2002}.
\begin{longtable}{|M{1.2cm}|>{\raggedright}M{2.7cm}|m{8cm}|}
		\hline \hline
		Variable
		& Boundary cond.
		& \multicolumn{1}{>{\centering\arraybackslash}m{80mm}|}{Interpretation}   \\ \hline \hline
		\multirow{2}{*}{$\check{c}$} &        $\partial_x \check{c}(t,0) = 0$               &        The tumour is radially symmetric. Therefore, there is no gradient of oxygen present at the tumour centre.               \\ \cline{2-3} 
		&        $\check{c}(t,\ell(t)) = 1$                 &     Constant external supply of oxygen. The unit value is because of nondimensionalisation.                \\ \hline
		\multirow{2}{*}{$\check{u}$} &           $\check{u}(t,0) = 0$           & Radial symmetry of the tumour implies no advection of tumour cells across the centre.                      \\ \cline{2-3} 
		& \small{$\displaystyle \mu \dfrac{\partial \check{u}}{\partial x}(t,\check{\ell}(t)) = \dfrac{(\check{\alpha}(t,\check{\ell}(t))  - \alast)^{+}}{(1 - \check{\alpha}(t,\check{\ell}(t)))^2}$} & Continuity of stress across the time--dependent boundary. \\ \hline
			\caption{Physical interpretations of the boundary conditions~\eqref{eqn:bdr_cond_1} -- \eqref{eqn:bdr_cond_2}}
		\label{tab:boundary}
	\end{longtable}

\noindent For $\mathrm{x} = 1,2,\ldots$  expansions of the abbreviations are as follows.
\begin{small}
\begin{longtable}{|M{2cm}|M{3.5cm}||M{2cm}|M{3.5cm}|}
			\hline
			\textbf{Abbreviation} & \textbf{Definition}   & \textbf{Abbreviation} & \textbf{Definition}    \\ \hline\hline
			TS.x                    & Threshold Solution.x     &    AS.x                  &    Aubin--Simon.x     \\ \hline
			DS.x                    & Discrete Solution.x     &      CA.x               & Convergence Analysis.x   \\ \hline
		    CR.x             
			&       Compactness Results.x
			&                    & \\         \hline
			\caption{Expansions of abbreviations}
			\label{tab:abrreviations}
\end{longtable}
\end{small}

\subsection{Physical properties of the model} 
\label{appen_C}
\renewcommand{\theequation}{$\mathrm{B}$.\arabic{equation}}
\renewcommand{\thetheorem}{$\mathrm{B}$.\arabic{theorem}}

	Define the continuous function spaces $\mathscr{C}^{1,2}(D_T)$ and $\mathscr{C}^{1,2}(D_T)$ by
	\begin{align}
	\mathscr{C}^{1}(\overline{D_T}) :={}& \big\{ c: \overline{D_T}\rightarrow \mathbb{R}\;:\; \dfrac{\partial c}{\partial t}, \dfrac{\partial c}{\partial x} \in \mathscr{C}(\overline{D_T}) \big\}, \text{ and }
	\label{eqn:c1_class} \\
	\mathscr{C}^{1,2}(\overline{D_T}) :={}& \big\{ c: \overline{D_T}\rightarrow \mathbb{R}\;:\; \dfrac{\partial c}{\partial t}, \dfrac{\partial^2 c}{\partial x^2} \in \mathscr{C}(\overline{D_T}) \big\}.
	\label{eqn:c12_class}
	\end{align}
\subsubsection*{Conservation of mass by the cell volume fraction equation}
\begin{lemma}[Continuous case]
	\label{lemma:mass_con}
If $(\check{\alpha},\check{u},\check{c},\check{\ell})$ is a solution of \eqref{system:nd_sys} such that $\check{\alpha}$ and $\check{u}$ belong to $\mathscr{C}^{1}(\overline{D_T})$, 
then $\check{\alpha}$ satisfies the mass conservation property 
	\begin{align}
	\int_{0}^{\check{\ell}(T)}\check{\alpha}(T,x)\,\mathrm{d}x = \int_{0}^{\ell_0} \alpha_0(x)\,\mathrm{d}x + \int_{0}^T \int_{0}^{\check{\ell}(t)} f(\check{\alpha},\check{c})\,\mathrm{d}x\,\mathrm{d}t.
	\label{eqn:mc}
	\end{align}
\end{lemma}
\begin{proof}
	Integrate~\eqref{eqn:vol_fraction} over $D_T$ to obtain
	\begin{small}
	\begin{align}
 \hspace{-0.7cm}\int_{0}^T \int_{0}^{\check{\ell}(t)} f(\check{\alpha},\check{c})\,\mathrm{d}x\,\mathrm{d}t = \int_{0}^T \int_{0}^{\check{\ell}(t)} \dfrac{\partial \check{\alpha}}{\partial t}\,\mathrm{d}x\,\mathrm{d}t + \int_{0}^T \int_{0}^{\check{\ell}(t)} \dfrac{\partial }{\partial x}\left(\check{u}\check{\alpha}\right)\,\mathrm{d}x\,\mathrm{d}t.
	\label{eqn:mc_1}
	\end{align}
\end{small}
	In~\eqref{eqn:mc_1}, apply Leibniz integral rule for the first term on the right--hand side and integrate $\frac{\partial}{\partial x}(\check{u}\check{\alpha})$ in the second term over the interval $(0,\check{\ell}(t))$ to arrive at
	\begin{align}
	\int_{0}^T \dfrac{\partial }{\partial t} \left(\int_{0}^{\check{\ell}(t)} \check{\alpha}(t,x)\,\mathrm{d}x\right)\,\mathrm{d}x  - \int_{0}^T \left[\check{\ell}'(t) - \check{u}(t,\check{\ell}(t)) \right] \check{\alpha}(t,\check{\ell}(t))\,\mathrm{d}t \\
	- \int_{0}^T  \check{u}(t,0)  \check{\alpha}(t,0)\,\mathrm{d}t = \int_{0}^T \int_{0}^{\check{\ell}(t)} f(\check{\alpha},\check{c})\,\mathrm{d}x\,\mathrm{d}t.
	\label{eqn:mc_2}
	\end{align}
	In the left hand side of~\eqref{eqn:mc_2}, carry out the time integration over the interval $(0,T)$ in the first term, use the conditions $\check{\ell}'(t) = \check{u}(t,\check{\ell}(t))$ on the second term, and $\check{u}(t,0) = 0$ on the third term  obtain~\eqref{eqn:mc}. 
\end{proof}
\begin{remark}
	The result~\eqref{eqn:mc} states that the total cell volume fraction at time $T$ is the sum of two quantities:  (a) total cell volume fraction present initially and (b) the total cell volume fraction produced by the source term $f(\check{\alpha},\check{c})$ during the time interval $(0,T)$, which is precisely the mass conservation property.
\end{remark}
\begin{lemma}[Discrete case]
	\label{lemma:disc_conser}
	Let $\alpha_{h,\delta} : \mathscr{D}_T \rightarrow \mathbb{R}$ and $c_{h,\delta} : \mathscr{D}_T \rightarrow \mathbb{R}$ be the time--reconstructs corresponding to the family of functions $(\alpha_h^n)_{n}$ obtained from~\eqref{eqn:upwind} and $(c_{h}^n)_n$ obtained from~\eqref{eqn:bform_ot}, respectively. Then, $\alpha_{h,\delta}$ satisfies the discrete mass conservation property
	\begin{multline}
	\int_{0}^{\ell_m} \alpha_{h,\delta}(T,x)\,\mathrm{d}x  =  \int_{0}^{\ell_0} \alpha_0(x)\,\mathrm{d}x  \\
	+\int_{0}^{T} \int_{0}^{{\ell_m}} (\alpha_{h,\delta}(t,x) - \ath)^{+}(1 - \alpha_{h,\delta}(t,x)) \dfrac{(1 + s_1)\Pi_{h,\delta}c_{h,\delta}(t,x)}{1 + s_1\Pi_{h,\delta}c_{h,\delta}(t,x)} \,\mathrm{d}x\,\mathrm{d}t \\
	- \int_{\delta}^{T+\delta} \int_{0}^{{\ell_m}} (\alpha_{h,\delta}(t,x) - \ath)^{+}\dfrac{s_2 + s_3\Pi_{h,\delta}c_{h,\delta}(t,x)}{1 + s_4\Pi_{h,\delta}c_{h,\delta}(t,x)} \,\mathrm{d}x\,\mathrm{d}t.
	\label{eqn:dmass}
	\end{multline}
\end{lemma}
\begin{proof}
	Sum~$h\times$\eqref{eqn:upwind} written for $j = 0,\ldots,J-1$ and $n = 1,\ldots,N$ and use the fact that $u_{0}^{n-1} = 0 = u_{J}^{n-1}$ to obtain
	\begin{align}
	\sum_{j=0}^{J-1}  h\ats{N}{j}  - \sum_{j=0}^{J-1}h\ats{0}{j}  ={}& \sum_{n=1}^{N}\delta \sum_{j=0}^{J-1}  h(\ats{n-1}{j} - \ath)^{+}(1-\ats{n-1}{j}) b_{j}^{n-1} \\
	&- \sum_{n=1}^{N}\delta \sum_{j=0}^{J-1} h(\ats{n}{j} - \ath)^{+}d_{j}^{n-1}.
	\label{eqn:dmc_1}
	\end{align}
	Note that each term in the sum $[\up{(n-1)}{j+1}\ats{n-1}{j} - \um{(n-1)}{j+1}\ats{n-1}{j+1} - \up{(n-1)}{j}\ats{n-1}{j-1} + \um{(n-1)}{j}\ats{n-1}{j}]$ in~\eqref{eqn:upwind} cancels with the same term of opposite sign coming from~\eqref{eqn:upwind} written for $j+1$ or $j-1$, and that boundary terms vanish due to the boundary conditions. Use the definitions of $b_j^n$ and $d_j^n$ (see \ref{ds.a} in Definition~\ref{defn:disc_soln}) and the definition of the time--reconstruct (see Definition~\eqref{defn:disc_soln}) to arrive at~\eqref{eqn:dmass} from~\eqref{eqn:dmc_1}.
\end{proof}
\subsubsection*{Nonnegativity and boundedness of the oxygen tension equation}
\begin{lemma}[Continuous case]
	If $\check{c}$ satisfies~\eqref{eqn:oxygen_tension} with $\check{\alpha}\ge 0$ and belongs to $\mathscr{C}^{2}(\overline{D_T})$, then $0 \le \check{c} \le 1$.
	\label{lemma:c_pos}
\end{lemma}
\begin{proof}
	\textbf{Positivity: }Multiply~\eqref{eqn:oxygen_tension} by the test function $-\check{c}^{-} = \min(\check{c},0)$ and integrate the product on the domain $D_T$ to obtain
	\begin{small}
	\begin{align}
\hspace{-0.5cm}-\int_{0}^T \int_{0}^{\check{\ell}(t)} \check{c}^{-} \dfrac{\partial \check{c}}{\partial t}\,\mathrm{d}x\,\mathrm{d}t \;+\; \lambda \int_{0}^T \int_{0}^{\check{\ell}(t)} \check{c}^{-} \dfrac{\partial^2 \check{c}}{\partial x^2} \,\mathrm{d}x\,\mathrm{d}t = \int_{0}^T \int_{0}^{\check{\ell}(t)} \check{c}^{-} \dfrac{Q\check{\alpha}\check{c}}{1 + \widehat{Q}_1 |\check{c}|}\,\mathrm{d}x\,\mathrm{d}t.
	\label{eqn:pc_1}
	\end{align}
	\end{small}
	In~\eqref{eqn:pc_1}, use $-\check{c}^{-} \frac{\partial \check{c}}{\partial t} = \frac{1}{2} \frac{\partial}{\partial t} (\check{c}^{-})^2$  to transform the first term on the left--hand side and apply integration by parts to spatial integral in second term to obtain 
	\begin{multline}
\int_{0}^T \int_{0}^{\check{\ell}(t)}  \dfrac{1}{2} \dfrac{\partial}{\partial t} (\check{c}^{-})^2\,\mathrm{d}x\,\mathrm{d}t\,+\, \lambda \int_{0}^T \int_{0}^{\check{\ell}(t)} \left| \dfrac{\partial \check{c}}{\partial x} \right|^2 \,\mathrm{d}x\,\mathrm{d}t \,+\, \lambda \int_{0}^T \check{c}^{-}(t,\check{\ell}(t)) \dfrac{\partial \check{c}}{\partial x} (t,\check{\ell}(t))\,\mathrm{d}t   \\ 
	-\lambda \int_{0}^T \check{c}^{-}(t,0) \dfrac{\partial \check{c}}{\partial x} (t,0)\,\mathrm{d}t = \int_{0}^T \int_{0}^{\check{\ell}(t)} \check{c}^{-} \dfrac{Q\check{\alpha}\check{c}}{1 + \widehat{Q}_1 |\check{c}|}\,\mathrm{d}x\,\mathrm{d}t.
	\label{eqn:pc_2}
	\end{multline}
	Apply Leibniz integral rule on the first term in the left hand side of~\eqref{eqn:pc_2} and use the facts that $\check{c}^{-}(t,\check{\ell}(t)) = 0$ and $\frac{\partial \check{c}}{\partial x} (t,0) = 0$ to arrive at 
	\begin{align}
\dfrac{1}{2}\int_{0}^{T} \dfrac{\partial }{\partial t} \left(  \int_{0}^{\check{\ell}(t)} (\check{c}^{-})^2\,\mathrm{d}x\right)\mathrm{d}t \,+\, \lambda \int_{0}^T \int_{0}^{\check{\ell}(t)} \left| \dfrac{\partial \check{c}}{\partial x} \right|^2\,\mathrm{d}x\,\mathrm{d}t \\
= \int_{0}^T \int_{0}^{\check{\ell}(t)} \check{c}^{-} \dfrac{Q\check{\alpha}\check{c}}{1 + \widehat{Q}_1 |\check{c}|}\,\mathrm{d}x\,\mathrm{d}t.
	\label{eqn:pc_3}
	\end{align}
	Carry out the time integration over the interval $(0,T)$ in first term in the left hand side of~\eqref{eqn:pc_3} and use the fact that $\check{c}^{-}(0,\cdot) = 0$ to obtain
	\begin{align}
	\lambda \int_{0}^T \int_{0}^{\check{\ell}(t)} \left| \dfrac{\partial \check{c}}{\partial x} \right|^2\,\mathrm{d}x\,\mathrm{d}t + \int_{0}^T \int_{0}^{\check{\ell}(t)}  \dfrac{Q\check{\alpha}}{1 + \widehat{Q}_1 |\check{c}|} \left(\check{c}^{-}\right)^2\,\mathrm{d}x\,\mathrm{d}t \le 0. 
	\label{eqn:pc_4}
	\end{align}
	This relation shows that $\partial_x\check{c}^-=0$ and thus, since $\check{c}^{-}(t,\check{\ell}(t)) = 0$, that $\check{c}^-=0$. This proves that $\check{c} \ge 0$ almost everywhere on $D_T$.
	
	\medskip
	\noindent\textbf{Boundedness: } Multiply~\eqref{eqn:oxygen_tension} by the test function $(\check{c} - 1)^{+} = \max(\check{c} - 1,0)$ and integrate the product on the domain $D_T$ to obtain 
	\begin{align} \int_{0}^T\int_{0}^{\check{\ell}(t)} (\check{c} - 1)^{+}\dfrac{\partial \check{c}}{\partial t}\,\mathrm{d}x\,\mathrm{d}t - \lambda  \int_{0}^T\int_{0}^{\check{\ell}(t)} (\check{c} - 1)^{+} \dfrac{\partial^2 \check{c}}{\partial x^2} \,\mathrm{d}x\,\mathrm{d}t \\
	= -\int_{0}^T\int_{0}^{\check{\ell}(t)}  \dfrac{Q \check{\alpha}}{1 + \widehat{Q}_1|\check{c}|} \check{c} (\check{c} - 1)^{+}.
	\label{eqn:bp_1}
	\end{align}
	In~\eqref{eqn:bp_1}, use $(\check{c} - 1)^{+}\frac{\partial \check{c}}{\partial t} = \frac{1}{2} \frac{\partial }{\partial t} \left((\check{c} - 1)^{+}\right)^2$ to transform the first term in the left--hand side, apply integration by parts to the spatial integral in the second term, and use the condition~\eqref{eqn:bdr_cond_2} to obtain
	\begin{align}
\int_{0}^T \int_{0}^{\check{\ell}(t)} \dfrac{1}{2} \dfrac{\partial }{\partial t} ((\check{c} - 1)^{+} )^2\,\mathrm{d}x\,\mathrm{d}t + \int_{0}^T \int_{0}^{\check{\ell}(t)} \left|\dfrac{\partial }{\partial x} (\check{c} - 1)^{+} \right|^2\,\mathrm{d}x\,\mathrm{d}t \\
	=  -\int_{0}^T\int_{0}^{\check{\ell}(t)}  \dfrac{Q \check{\alpha}}{1 + \widehat{Q}_1|\check{c}|} \check{c} (\check{c} - 1)^{+}.
	\label{eqn:bp_2}
	\end{align}
	Apply Leibniz integral rule on the first term in the left hand side of~\eqref{eqn:bp_2}, carry out the time integration over the interval $(0,T)$, and use the condition~\eqref{eqn:in_cond} to obtain 
	\begin{align}
	\int_{0}^T \int_{0}^{\check{\ell}(t)} \left|\dfrac{\partial }{\partial x} (\check{c} - 1)^{+} \right|^2\,\mathrm{d}x\,\mathrm{d}t + \int_{0}^T\int_{0}^{\check{\ell}(t)}  \dfrac{Q \check{\alpha}}{1 + \widehat{Q}_1|\check{c}|} ((\check{c} - 1)^{+})^2 \le 0.
	\label{eqn:bp_3}
	\end{align}
	Result~\eqref{eqn:bp_3} implies that $(\check{c}  -1)^{+} = 0$, which yields that $\check{c} \le 1$ almost everywhere on $D_T$.
\end{proof}
The positivity and boundedness results corresponding to the discrete oxygen tension $c_{h,\delta}$, obtained from the numerical scheme~\eqref{eqn:bform_ot}, is provided in Lemma~\ref{lemma:dot_max}.

\setcounter{equation}{0}
\subsection{Identities and results}
\label{appen_B}
\renewcommand{\theequation}{$\mathrm{C}$.\arabic{equation}}
\begin{enumerate}[label= $\mathrm{C}$.\Roman*.,ref=$\mathrm{C}$.\Roman*]
	\item\label{appen_id.a}  If $a,b,c,d \in \mathbb{R}$, then the following identities hold:
	\begin{subequations}
	\begin{align}
	ab - cd  &= \dfrac{(a+c)(b-d)}{2} + \dfrac{(a-c)(b+d)}{2},
	\label{eqn:split_id_1} \\
	ab - cd &= (a - c)b + (b - d)c, 
	\label{eqn:split_id_2} \\
	2ab &\le a^2 + b^2, \text{ and } \label{eqn:split_id_3}\\
	a &= a^{+} - a^{-},\,|a| = a^{+} + a^{-},
	\label{eqn:maxmin}
	\end{align}
		\end{subequations}
	where  $a^{+} = \max(a,0)$ and $a^{-} = -\min(a,0)$.
	\item\label{appen_id.b} \textbf{Discrete integration by parts formula. \cite[Section D.1.7]{droniou2018gradient}} 
	For any families $(a_n)_{n=0,\ldots,N}$ and $(b_n)_{n=0,\ldots,N}$ of real numbers, it holds
	\begin{align}
	\sum_{n=0}^{N-1} (a_{n+1} - a_{n})b_{n} = -\sum_{n=0}^{N-1} a_{n+1} (b_{n+1} - b_{n}) + a_{N}b_{N} - a_{0}b_{0}.
	\label{eqn:disc_int_parts}
	\end{align}
	\item\label{appen_id.c} \textbf{Theorem (Helly's selection theorem). \cite[Theorem 4, p.~176]{Evans20151}.} 
		Let $\Omega \subset \mathbb{R}^d$ ($d \ge 1$) be an open and bounded set with a Lipschitz boundary $\partial \Omega$, and $(f_n)_{n\in \mathbb{N}}$ be a sequence in $BV(\Omega)$ such that $(||f_{n}||_{BV(\Omega)})_n$ is uniformly bounded. Then, there exists a subsequence $(f_n)_n$ up to re-indexing and a function $f \in BV(\Omega)$ such that as $n \rightarrow \infty$, $f_n \rightarrow f$ in $L^1(U)$ and almost everywhere in $\Omega$.
	\item\label{appen_id.d} \textbf{Theorem (discrete Aubin--Simon theorem). \cite[Theorem C.8]{droniou2018gradient}.}
		Let $p \in [1,\infty)$,  $(X_m,Y_m)_{m\in\mathbb{N}}$ be a compactly--continuously embedded sequence in a Banach space $B$, and $(f_m)_{m \in \mathbb{N}}$ be a sequence in $L^p(0,T;B)$, where $T > 0$ such that the  \edit{assumptions} \ref{as:a},~\ref{as:b}, and~\ref{as:c} are satisfied.
		\begin{enumerate}[label= (\alph*),ref=(\alph*)]
			\item\label{as:a} Corresponding to each $m \in \mathbb{N}$, there exists an $N \in \mathbb{N}$, a partition $0 = t_{0} < \cdots < t_{N} = T$, and a finite sequence $(g_n)_{n=0,\cdots,N}$ in $X_m$ such that $\forall\,n \in \{0,\ldots,N-1\}$ and almost every $t \in (t_{n},t_{n+1})$, $f_{m}(t) = g_{n}$. Then, the discrete derivative $\delta_{m}f_{m}$ is defined almost everywhere by $\delta_m f_m(t) := (g_{n+1} - g_{n})/(t_{n+1}-t_{n})$ on $(t_n,t_{n+1})$ for all $n \in \{0,\ldots,N-1\}$.
			\item\label{as:b}
			The sequence $(f_{m})_{m \in \mathbb{N}}$ is bounded in $L^p(0,T;B)$.
			\item \label{as:c}The sequences $(||f_m||_{L^p(0,T;X_m)})_m$ and $(||\delta_m f_m||_{L^1(0,T;Y_m)})_m$ are bounded.
		\end{enumerate}
	Then, $(f_m)_{m\in N}$ is relatively compact in $L^p(0,T;B)$.
	\item\label{it:con} \begin{enumerate}[label= (\alph*),ref=\ref{it:con}(\alph*)]
		\item \label{appen_id.e} \textbf{Lemma (weak-strong convergence). \cite[Lemma D.8]{droniou2018gradient}.}
		If $p \in [0,\infty)$ and $q := p/(1-p)$ are conjugate exponents, $f_n \rightarrow f$ strongly in $L^p(X)$, and $g_n \halfarrow g$ weakly in $L^{q}(X)$, where $(X,\mu)$ is a measured space, then 
		\begin{align}
		\int_{X} f_{n}g_{n}\,\mathrm{d}\mu \rightarrow \int_{X} fg\,\mathrm{d}\mu.
		\end{align}
		The next result follows from Lebesgue's dominated convergence theorem. 
		\item\label{appen_id.f} \textbf{Lemma (bounded-strong convergence).}
		If $f_n \rightarrow f$ in $L^2(X)$, $g_n \rightarrow g$ almost everywhere on $X$, $||g_n||_{L^\infty(X)}$ is uniformly bounded, then $f_ng_n$ converges to $fg$ in $L^2(X)$.	
	\end{enumerate}
\item \label{appen_id.g} \textbf{Lemma \cite[Theorems~3.1, 3.2]{Thomee200811}.} Let $D$ be an $n \times n$ diagonal matrix with positive entries, $A$ be an $n \times n$ matrix with all off--diagonal entries nonpositive, and $\mathbb{I}_n$ be $n \times n$ identity matrix. Then, the operator $(\mathbb{I}_n + k D^{-1}S)^{-1}$ is positive for sufficiently small $k > 0$. 
	\end{enumerate}
\end{document}